\documentclass[12pt]{article}

\usepackage[T1]{fontenc}          
\usepackage[utf8]{inputenc}

\usepackage{amsmath,amsfonts,amssymb,amsthm,mathtools}
\usepackage{mathrsfs}             
\usepackage{bm}

\usepackage[a4paper]{geometry}
\usepackage{graphicx}
\usepackage{float}
\usepackage{placeins}
\usepackage{caption}
\usepackage{marginnote}
\usepackage{comment}

\usepackage{array}
\usepackage{multirow}
\usepackage{booktabs}
\usepackage{longtable}
\usepackage{tabularx}
\usepackage{needspace}

\usepackage[shortlabels]{enumitem}

\usepackage{algorithm}
\usepackage{algpseudocode}

\usepackage{tikz}
\usetikzlibrary{positioning}
\usetikzlibrary{calc}

\usepackage[unicode]{hyperref}
\hypersetup{
	colorlinks=true,
	linkcolor=blue,
	citecolor=blue,
	urlcolor=blue
}

\numberwithin{figure}{section}
\numberwithin{equation}{section}

\makeatletter

\newtheoremstyle{myplain}
{6pt}{6pt}%
{\itshape}
{}%
{\bfseries}
{.}%
{0.5em}%
{}%

\newtheoremstyle{mydefinition}
{6pt}{6pt}%
{\normalfont}
{}%
{\bfseries}
{.}%
{0.5em}%
{}%

\newtheoremstyle{myremark}
{6pt}{6pt}%
{\normalfont}
{}%
{\bfseries}
{.}%
{0.5em}%
{}%

\renewenvironment{proof}[1][Proof]{%
	\par\pushQED{\qed}\normalfont
	\topsep6pt\relax
	\trivlist
	\item[\hskip\labelsep\bfseries #1.]%
}{%
	\popQED\endtrivlist\@endpefalse
}

\makeatother

\theoremstyle{myplain}
\newtheorem{theorem}{Theorem}[section]
\newtheorem{proposition}[theorem]{Proposition}
\newtheorem{lemma}[theorem]{Lemma}
\newtheorem{corollary}[theorem]{Corollary}

\theoremstyle{mydefinition}
\newtheorem{definition}[theorem]{Definition}

\theoremstyle{myremark}
\newtheorem{remark}[theorem]{Remark}

\theoremstyle{myplain}

\theoremstyle{mydefinition}

\theoremstyle{myremark}

\graphicspath{{GM_figures/}}
\title{A Canonical $m$-Atomic Decomposition of Bipartite Graphs via a Grid Model}
\author{Béla Jónás\footnote{ORCID: \href{https://orcid.org/0009-0009-4608-4434}{0009-0009-4608-4434}}}

\date{August 2026 \hspace{3pt} Budapest, Hungary}

\begin{document}
	
	\maketitle
	
	\begin{abstract}
		We study finite, connected, simple bipartite graphs in a grid model, in which a graph
		is drawn as a rectangular array and its structure is read off from empty subrectangles,
		called holes. In this model we attach to every brick a numerical invariant, its
		\emph{characteristic} $m$, the difference between the number of rows and the largest
		proper independent set. A brick is \emph{excessive} if $m>0$.

		Our main results concern this invariant. We determine the characteristic of a
		disconnected excessive brick from those of its components, showing that
		$m=\min_i\min\{m_i,\operatorname{imb}(W_i)\}$ while the imbalance is additive; and we
		prove that an $m$-excessive brick is $m$-extendable, that is, every matching of size $m$
		extends to a maximum matching. Since Plummer's notion of $n$-extendability is defined
		only for graphs carrying a perfect matching, and our proof nowhere uses balance, the
		characteristic extends that notion canonically to unbalanced bipartite graphs. Using the
		characteristic we partition bipartite graphs into eleven structural classes.

		The underlying decomposition into atomic blocks is the classical decomposition into
		elementary components, and the description of the maximum proper independent sets by
		ideals of the block poset is likewise classical; Section~\ref{sec:related} states
		precisely which results are classical and are not claimed here. What the grid model adds
		is a single geometric framework in which holes, characteristics and the block triangular
		form are read off from one picture.
	\end{abstract}
	
	\textbf{MSC (2020)}: 05C70 (primary); 05C38, 15A21 (secondary)
	
	\textbf{Keywords:} bipartite graph, Dulmage--Mendelsohn decomposition, 
	\par Hall deficiency, matchings, Block Triangular Form (BTF), atomic
	\par decomposition, $m$-excessive property
	
	\newpage
	\tableofcontents
	\newpage
	
\section{Introduction}
	Matching theory has long been a cornerstone of combinatorial optimization and structural graph theory. Among its most fundamental results is the Dulmage--Mendelsohn (DM) decomposition, which provides a canonical structural description of bipartite graphs. In the matrix- and constraint-system literature, the three DM regions are known as the underconstrained ($D$), overconstrained ($E$), and well-constrained ($C$) parts. While this partitioning, and its associated block triangular form (BTF), has proven invaluable in fields ranging from matrix theory to the analysis of overdetermined systems \cite{PothenFan1990, ParkerEtAl2023}, the structural properties of the connected components of $D$ and $E$ --- in particular their excessiveness and their geometric characterization --- are the focus of the present work.

	In this paper, we introduce the concept of \emph{$m$-atomic bricks} (Definition~\ref{def:m-atomic}): connected $m$-excessive induced subgraphs, generalizing the classical notion of elementary graphs \cite{LovaszPlummer1986} to both balanced and unbalanced bipartite graphs.

	We prove that every finite, connected, simple bipartite graph admits a unique canonical decomposition into $m$-atomic bricks. First, an \emph{excessive decomposition} (Definition~\ref{def:exc-decomp-prelim}) partitions $G_0$ into a BTF with blocks $D$, $C_1, \ldots, C_k$, $E$, each excessive; the square blocks $C_i$ are balanced, hence connected and atomic (Lemma~\ref{lem:balanced-excessive-atomic}), while $D$ and $E$ need not be connected. Second, the \emph{atomic decomposition} refines $D$ and $E$ into the connected components of $G_0[D]$ and $G_0[E]$, each of which is unbalanced, excessive, and hence atomic. The value of $m$ may differ between blocks.

	This is made transparent by a geometric, grid-based framework in which the edges of $G_0$ fall into two structurally distinct classes: edges within atomic blocks (each belonging to some maximum matching, by Proposition~\ref{extendable} and Proposition~\ref{prop:atomic-iff-connected}), and forbidden edges between atomic blocks from different regions (belonging to no maximum matching, by Corollary~\ref{cor:forbidden-dot}); see Figure~\ref{fig:gridmodel}.

	The atomic blocks themselves are not new: they are the elementary components of classical matching theory, and the factorization of maximum matchings over them, as well as the description of the maximum stable sets by ideals of a canonical order, belong to the classical Dulmage--Mendelsohn theory. Section~\ref{sec:related} says precisely which statements below are classical and are not claimed here.

	What the classical theory does not supply is a numerical invariant attached to the blocks. The contribution of this paper is the \emph{characteristic} $m$ (Definition~\ref{def:characteristic}) and its arithmetic: Proposition~\ref{prop:excess-components} computes the characteristic of a disconnected excessive brick from those of its components, $m=\min_i\min\{m_i,\mathrm{imb}(W_i)\}$, and Proposition~\ref{extendable} shows that $m$-excessiveness implies $m$-extendability with the same $m$, which extends Plummer's notion of $n$-extendability \cite{Plummer1980} from balanced to unbalanced bipartite graphs. From the characteristic we further obtain a classification of bipartite graphs into eleven structural classes, and a geometric reading of the whole structure in the grid model. The main results are stated precisely in Section~\ref{sec:main-results}.

	\section{Related Work}\label{sec:related}
		The Dulmage--Mendelsohn (DM) decomposition \cite{Dulmage1958, Dulmage1959, Dulmage1963, LovaszPlummer1986, Murota2000} partitions bipartite graphs into $D$, $E$, and $C$ regions. Pothen and Fan \cite{PothenFan1990} provide a comprehensive algorithmic treatment of the block triangular form, distinguishing a \emph{coarse decomposition} (identifying the $D$, $C$, $E$ regions) from a \emph{fine decomposition}. Their fine decomposition of the $C$-region uses Tarjan's SCC algorithm \cite{Tarjan1972}; for the $D$- and $E$-blocks they explicitly identify the connected components and give algorithms for computing them.

		\medskip
		\noindent\textbf{What is classical.} We state at the outset which of the structures below are already part of the classical theory, so that our own contribution can be located precisely against them. An edge of a bipartite graph is \emph{allowed} if it lies in some maximum matching, and the connected components of the subgraph formed by the allowed edges are the \emph{elementary components}; see \cite{LovaszPlummer1986, Murota2000} and, for a compact modern account, \cite[Sections~1.2 and~2]{Kita2016}. Two facts about them are classical:
		\begin{enumerate}[label=(C\arabic*),leftmargin=*]
			\item\label{classical:factorization} A set of edges is a maximum matching if and only if it is the disjoint union of maximum matchings of the elementary components.
			\item\label{classical:ideals} The elementary components carry a canonical partial order, and the maximum stable sets --- equivalently, by K\H{o}nig duality, the minimum vertex covers --- are exactly the sets obtained from complementary pairs of \emph{normalized} ideals of that order \cite[Theorem~2.9]{Kita2016}. That the resulting family is a distributive lattice then follows from Birkhoff's representation theorem \cite{Birkhoff1937}; see also \cite{Minasyan2013}.
		\end{enumerate}
		The atomic blocks of the present paper coincide with the elementary components. Indeed, every edge inside an atomic block is allowed (Proposition~\ref{extendable}), every edge between blocks of different regions is forbidden (Corollary~\ref{cor:forbidden-dot}), and distinct components of the same region are joined by no edge at all; hence the allowed-edge subgraph is the disjoint union of the atomic blocks, and, each block being connected, its components are precisely those blocks. Consequently property~(4) of Theorem~\ref{thm:main_decomp} is a restatement of~\ref{classical:factorization}, and Corollary~\ref{cor:MPIS} is~\ref{classical:ideals}, with the $C$-block poset in place of the poset of all elementary components (the $D$- and $E$-blocks are forced, being minimal resp.\ maximal in that order). We do not claim these as new.

		\medskip
		\noindent\textbf{What is new here.} The classical theory identifies the elementary components and the order between them, but attaches no numerical invariant to them. Our contribution is to supply one, and to develop its consequences:
		\begin{itemize}[leftmargin=*]
			\item \textbf{The characteristic $m$.} Every brick, and in particular every atomic block, carries a characteristic $m = b - \alpha_p$ (Definition~\ref{def:characteristic}). Excessiveness ($m>0$) of the blocks is, given the classical fact that $D$ and $E$ have positive surplus \cite[Theorem~3.2.4]{LovaszPlummer1986}, a short consequence, since the surplus function is computed componentwise (Lemma~\ref{lem:excessive-components}). What does \emph{not} follow from the classical theory is the arithmetic of the characteristic: Proposition~\ref{prop:excess-components} determines the characteristic of a disconnected excessive brick from those of its components, $m = \min_i \min\{m_i, \mathrm{imb}(W_i)\}$, together with the additivity of the imbalance.
			\item \textbf{Extendability.} $m$-excessiveness implies $m$-extendability with the same $m$ (Proposition~\ref{extendable}). Plummer's notion of $n$-extendability \cite{Plummer1980} is defined only for graphs carrying a perfect matching, hence in the bipartite case only for balanced graphs; the proof given here nowhere uses balance, so the characteristic extends the notion canonically to unbalanced bipartite graphs, with the perfect matching replaced by a maximum matching of size $\mathrm{small}(T_0)$ (Remark~\ref{rem:plummer}).
			\item \textbf{Classification.} The eleven structural classes (Theorem~\ref{thm:main_classification}), together with the impossibility arguments for the excluded combinations.
			\item \textbf{Geometry.} The grid model itself, in which holes, characteristics and the block triangular form are read off from a single picture, and in which the intersection of all maximum holes is identified (Theorem~\ref{thm:ExD-intersection}).
		\end{itemize}

		Recent extensions of the DM decomposition to general graphs \cite{Kita2012, Kita2024} address the internal structure of factor-components rather than the bipartite $D$- and $E$-blocks that are the focus here. The classical elementary graph theory \cite{LovaszPlummer1986} is restricted to the balanced case; the characteristic extends it to all bipartite graphs via a single geometric invariant.

		Table~\ref{tab:literature-comparison} summarizes the relationship between the present work and the closest prior references.

		\begin{table}[ht]
		\centering
		\footnotesize
		\setlength{\tabcolsep}{4pt}
		\begin{tabular}{p{5.6cm}cccc}
		\toprule
		\textbf{Result} & \textbf{L--P} & \textbf{P--F} & \textbf{Classical} & \textbf{Here} \\
		& \textbf{(1986)} & \textbf{(1990)} & \textbf{DM} & \\
		\midrule
		$D$ and $E$ blocks excessive & $\checkmark$ (Thm.~3.2.4) & --- & --- & --- \\
		Atomic blocks $=$ elementary components & --- & --- & $\checkmark$ & --- \\
		Max.\ matchings factor over the blocks & --- & --- & $\checkmark$ \ref{classical:factorization} & --- \\
		MPIS $\leftrightarrow$ ideals of the block poset & --- & --- & $\checkmark$ \ref{classical:ideals} & --- \\
		\midrule
		Characteristic $m$ of a block & --- & --- & --- & $\checkmark$ Def.~\ref{def:characteristic} \\
		$m=\min_i\min\{m_i,\mathrm{imb}(W_i)\}$; $\mathrm{imb}$ additive & --- & --- & --- & $\checkmark$ Prop.~\ref{prop:excess-components} \\
		$m$-excessive $\Rightarrow$ $m$-extendable (unbalanced) & --- & --- & --- & $\checkmark$ Prop.~\ref{extendable} \\
		11-class structural classification & --- & --- & --- & $\checkmark$ Thm.~\ref{thm:main_classification} \\
		$\bigcap\mathrm{MH} = \mathrm{row}(E)\times\mathrm{col}(D)$ & --- & --- & --- & $\checkmark$ Thm.~\ref{thm:ExD-intersection} \\
		Unified geometric framework & --- & --- & --- & $\checkmark$ \\
		\bottomrule
		\end{tabular}
		\caption{Scope of the present work. The upper block lists results that are classical and are \emph{not} claimed here; the lower block lists the contributions of this paper. L--P = Lov\'asz--Plummer; P--F = Pothen--Fan; ``Classical DM'' refers to \cite{Dulmage1958, Dulmage1959, Dulmage1963, Murota2000} as presented in \cite[Sections~1.2, 2]{Kita2016}.}
		\label{tab:literature-comparison}
		\end{table}
		
	\section{Preliminaries}\label{sec:prelim}
	We follow the terminology of Lovász and Plummer~\cite{LovaszPlummer1986}.
	In addition:
	
	For $k \in \mathbb{N}$, let $[k] := \{1,2,\dots,k\}$. 
	Unless stated otherwise, $G_0 = (A,B,E)$ denotes a finite, simple bipartite graph with no isolated vertices, $a = |A|$, $b = |B|$ and $a\le b$, and $G_0$ is \emph{not} the complete bipartite graph $K_{a,b}$. The extremal case $G_0 = K_{a,b}$ is excessive with characteristic $m = b$, is DM-irreducible, and requires no decomposition; it is excluded from the general theory unless stated otherwise.
	Subgraphs of $G_0$ are not assumed to satisfy these conditions.

	\textbf{Notation convention.} The fixed graph $G_0$ and its grid $T_0$ serve as the \emph{ambient space}: all subgraphs and sub-bricks under study are subgraphs of $G_0$ (resp.\ sub-bricks of $T_0$), unless explicitly stated otherwise. In particular, when we refer to a brick, hole, or strip without further qualification, these are understood as sub-bricks of $T_0$.

	For a brick $T = Y \times X$ (where $Y \subseteq B$, $X \subseteq A$), we write $\mathrm{row}(T) := Y$ and $\mathrm{col}(T) := X$; the full definition appears in Section~\ref{sec:definitions}.
	Every brick $T$ (as a subgrid of $T_0$) is identified with its corresponding induced bipartite subgraph $G_0[\mathrm{col}(T), \mathrm{row}(T)]$. We therefore use $T$ uniformly for both the geometric and graph-theoretic object, and employ standard graph-theoretic notation: $V(T)$, $E(T)$, $\alpha(T)$, etc. Definitions and properties stated for a brick $T$ apply equally to the corresponding induced subgraph, and vice versa.

	A sub-brick $T \subseteq T_0$ inherits from $G_0$ the properties of being finite, simple, and bipartite. It is not assumed to be connected or vertically oriented, and unlike $T_0$ it may contain isolated rows or columns; such assumptions are stated explicitly when required. Results proved for $T_0$ apply equally to any induced sub-brick $T$, regarded as a $T_0$ in its own right, with $a$, $b$, $\delta$ replaced by the corresponding values of $T$.
	
	For $X \subseteq A \cup B$, let $N(X)$ denote its neighborhood, and set $\mathrm{imb}(T) = ||\mathrm{col}(T)| - |\mathrm{row}(T)||$.
	
	\begin{definition}[Proper independent set]
		Let $T$ be a brick. An independent set $F \subseteq \mathrm{col}(T) \cup \mathrm{row}(T)$ is \emph{proper} if
		\[
		F \cap \mathrm{col}(T) \neq \emptyset, \quad F \cap \mathrm{row}(T) \neq \emptyset, \quad \mathrm{col}(T) \not\subseteq F, \quad \mathrm{row}(T) \not\subseteq F.
		\]
		A \emph{maximum proper independent set} is a proper independent set of maximum cardinality, and
		\[
		\alpha_p(T) = \max\{\, |F| : F \text{ is a proper independent set of } T \,\}.
		\]
	\end{definition}

	\begin{definition}[Characteristic, excess, excessive]\label{def:characteristic}
		Let $T_0$ be a brick and set $m = b - \alpha_p(T_0)$, where $b = |\mathrm{row}(T_0)|$. The value $m$ is called the \emph{characteristic} of $T_0$. Then $T_0$ is:
		\begin{itemize}
			\item \emph{$m$-excessive}, with \emph{excess} $m$, if $m > 0$; we also say $T_0$ is \emph{excessive} if it is $m$-excessive for some $m > 0$;
			\item \emph{liminal} if $m = 0$;
			\item \emph{deficient}, with \emph{deficit} $\delta = -m$, if $m < 0$.
		\end{itemize}
	\end{definition}

	This trichotomy defines the characteristic of $T_0$; its structural consequences are developed in Section~\ref{sec:grid}. The term \emph{liminal} (from Latin \emph{limen}, threshold) indicates that $m = 0$ is the boundary between excess and deficiency; the condition is equivalent to the Hall condition ($\alpha_p(T_0) = b$). We adopt it as a concise term, no standard one existing for this case.

	The three cases describe how the short side of $T_0$ relates to its matching capacity, via the hole condition (Corollary~\ref{cor:matching-from-hole}). A \emph{deficient} brick fails the Hall condition: a dot-free rectangle of $\mathrm{vcount}$ exceeding the long side $b$ witnesses a set of columns with fewer neighbors than columns, so no matching saturates the short side. A \emph{liminal} brick is exactly Hall-tight: a matching saturates the short side, but the largest hole already reaches $\mathrm{vcount} = b$, leaving no slack. An \emph{excessive} brick has strict surplus --- every nonempty proper set of columns has strictly more neighbors than columns (the Strong Marriage Condition, in the balanced case) --- so its largest hole has $\mathrm{vcount} < b$. Complete bipartite bricks are the most excessive, whereas any set of columns with fewer neighbors than columns makes a brick deficient.

	A word on why the characteristic $m$, rather than the Hall deficiency, is taken as the organizing invariant. The two are closely related --- for a deficient brick, $m = -\delta$ where $\delta$ is the maximum Hall deficiency --- but they play different structural roles. The Hall deficiency $\delta$ is a single global quantity: it records the worst violation of the marriage condition over all of $A$, and is naturally attached to the graph as a whole. The characteristic $m$, by contrast, is defined for \emph{every} sub-brick and is preserved under the decomposition: each atomic block carries its own characteristic, and these may differ from block to block. It is this \emph{locality and additivity across blocks} that makes $m$ the appropriate invariant for a structural theory --- it survives restriction to sub-bricks and refinement of the decomposition, whereas the global Hall deficiency does not localize in this way. In the grid model, $m$ has a direct geometric meaning as the distance between the characteristic segment and the Hall line (Section~\ref{sec:grid}), which is what allows the excessive and atomic decompositions to be read off geometrically.

	\begin{definition}[$m$-atomic brick]\label{def:m-atomic}
		Let $T$ be a brick.
		We say that $T$ is \emph{$m$-atomic} if it is connected and $m$-excessive.
		We say that $T$ is \emph{atomic} if it is $m$-atomic for some $m > 0$.
	\end{definition}

	\begin{remark}\label{rem:atomic-vs-elementary}
		In the balanced case ($|\mathrm{col}(T)| = |\mathrm{row}(T)|$), atomicity coincides with the classical notion of an elementary graph in the sense of Lovász and Plummer~\cite{LovaszPlummer1986}: a balanced graph is atomic if and only if every edge belongs to some perfect matching. For unbalanced graphs ($|\mathrm{col}(T)| < |\mathrm{row}(T)|$), this provides a canonical generalization: the condition of a perfect matching is replaced by a maximum matching of size $|\mathrm{col}(T)|$, and the classical expanding condition is replaced by the (weaker) $m$-excessive condition, defined via maximum holes.

		Note that the value of $m$ may differ between atomic blocks in the atomic decomposition; what is required is only that each block is $m$-atomic for some $m > 0$.

		A remark on terminology is in order, since several nearby terms are reserved in the classical literature. In the Kotzig--Lov\'asz--Plummer decomposition of general graphs \cite{LovaszPlummer1986}, a \emph{brick} is a $3$-connected bicritical graph (necessarily non-bipartite), and its bipartite analogue is a \emph{brace}: a connected bipartite graph in which every matching of size at most two extends to a perfect matching. Both notions presuppose a \emph{balanced} graph carrying a perfect matching. Our $m$-atomic units are, by contrast, typically \emph{unbalanced} ($m$-excessive with $|\mathrm{col}(T)| < |\mathrm{row}(T)|$); in the balanced case they reduce exactly to braces (elementary bipartite graphs), and for $m > 0$ they are their canonical unbalanced generalization. We therefore do not adopt the term \emph{brace}, which is balanced by definition and would not cover the unbalanced units. Throughout, \emph{brick} denotes the purely geometric object of the grid model --- a rectangular sub-grid (Definition~\ref{Brick}) --- and never the bicritical graph of \cite{LovaszPlummer1986}; the atomic graph-theoretic units are the \emph{$m$-atomic bricks}, equivalently the \emph{$m$-excessive components}, of Definition~\ref{def:m-atomic}.
	\end{remark}

	The following structural notions are used in the statement of the main results; their detailed construction is given in Sections~\ref{sec:excessive-decomp} and~\ref{sec:atomic}. The grid-model formulations are deferred to those sections.

	\begin{definition}[Block Triangular Form]\label{def:BTF}
		An ordered partition of a brick $T$ into sub-bricks $B_1, B_2, \ldots, B_r$ is called a \emph{block triangular form} (BTF) if there are no edges from $V(B_i)$ to $V(B_j)$ for $i > j$; that is, all edges between blocks run ``forward'' in the ordering. Equivalently, the adjacency matrix of $T$ (with vertices ordered according to the block sequence) is block lower-triangular --- the transpose of the upper-triangular convention common in numerical linear algebra, since here all inter-block edges run forward, from earlier to later blocks. The full grid-model formulation is given in Section~\ref{sec:excessive-decomp}.
	\end{definition}

	\begin{definition}[Excessive decomposition]\label{def:exc-decomp-prelim}
		The \emph{excessive decomposition} of $G_0$ is a BTF partition into excessive induced subgraphs $D, C_1, \ldots, C_k, E$ (any of which may be absent), where $D$ is unbalanced with $|A \cap V(D)| > |B \cap V(D)|$ (horizontal: more columns than rows), each $C_i$ is balanced ($|A \cap V(C_i)| = |B \cap V(C_i)|$), and $E$ is unbalanced with $|A \cap V(E)| < |B \cap V(E)|$ (vertical: more rows than columns). Its existence, uniqueness, and full construction are established in Section~\ref{sec:excessive-decomp}.
	\end{definition}

	\begin{definition}[Atomic decomposition]\label{def:atom-decomp-prelim}
		The \emph{atomic decomposition} of $G_0$ is the refinement of the excessive decomposition obtained by splitting $D$ and $E$ into the connected components of the induced subgraphs $G_0[D]$ and $G_0[E]$. The resulting blocks $D_1,\ldots,D_p$, $C_1,\ldots,C_k$, $E_1,\ldots,E_q$ are all atomic. Its properties are established in Section~\ref{sec:atomic}.
	\end{definition}

	In a more modern interpretation \cite{Berczi2018}, a connected bipartite graph is DM-irreducible if its DM-decomposition consists of the graph itself; in our framework this coincides with $G_0$ being excessive (the blocks of the excessive decomposition are excessive by construction, and an excessive graph admits no further splitting by Lemma~\ref{lem:excessive-irreducible}). This includes both balanced excessive graphs (a single $C$-block) and unbalanced ones (a single $E$-block, no $D$-block). We follow this convention throughout; the classical definition \cite{LovaszPlummer1986} restricted DM-irreducibility to the balanced case (see p.~137).

	\begin{definition}[DM-irreducible]\label{def:dm-irred-prelim}
		A connected bipartite graph $G_0$ (or brick $T_0$) is \emph{DM-irreducible} if its excessive decomposition consists of a single block, i.e., $G_0$ itself is excessive. Equivalently, either $G_0$ is balanced with $k=1$ and no $D$- or $E$-block, or $G_0$ is unbalanced with $k=0$, no $D$-block, and $G_0$ itself serving as the single $E$-block.
	\end{definition}

	\noindent\textbf{Grid--graph dictionary.} Every grid-model primitive has a precise
	graph-theoretic meaning; each row below reads as ``grid notion $\equiv$ graph notion''.
	This correspondence is used silently throughout, so that geometric statements may
	always be read back as statements about $G_0 = (A,B,E)$.
	{\small
	\begin{longtable}{@{}p{0.40\textwidth}p{0.50\textwidth}@{}}
	\toprule
	\textbf{Grid model} & \textbf{Bipartite graph $G_0=(A,B,E)$} \\
	\midrule
	\endfirsthead
	\toprule
	\textbf{Grid model} & \textbf{Bipartite graph $G_0=(A,B,E)$} \\
	\midrule
	\endhead
	dot at cell $(b,a)$ & edge $ab \in E$ \\
	brick $Y \times X$ ($X\subseteq A$, $Y\subseteq B$) & induced subgraph $G_0[X \cup Y]$ \\
	hole $Y \times X$ (dot-free) & no edges between $X$ and $Y$ \\
	$\mathrm{vcount}$ of a hole, $|X|+|Y|$ & size of the proper independent set $X \cup Y$ \\
	maximum hole (MH) & maximum proper independent set (MPIS) \\
	remote mate of an MH & minimum proper vertex cover (MPCS) \\
	matching (no two dots see each other) & matching of $G_0$ \\
	characteristic $m = b - \alpha_p$ & $|B|$ minus the maximum proper independent set size \\
	excessive / liminal / deficient & positive / zero / negative surplus from the short side \\
	\bottomrule
	\end{longtable}}

	\vspace{0.5\baselineskip}
	\section*{Notation Summary}
	\addcontentsline{toc}{section}{Notation Summary}
	\nopagebreak
	\begin{center}
	\small
	\begin{longtable}{@{}lp{7.5cm}l@{}}
	\toprule
	\textbf{Symbol} & \textbf{Description} & \textbf{Reference} \\
	\midrule
	\endhead
	\multicolumn{3}{l}{\textit{Graphs}} \\[2pt]
	$G_0 = (A,B,E)$ & Fixed finite, simple bipartite graph & Sec.~\ref{sec:prelim} \\
	$a = |A|$, $b = |B|$ & Column and row counts of $G_0$; $a \leq b$ & Sec.~\ref{sec:prelim} \\
	$G$ & Arbitrary finite bipartite graph (brick) & Sec.~\ref{sec:prelim} \\
	$N(X)$ & Neighborhood of $X \subseteq A \cup B$ & Sec.~\ref{sec:prelim} \\
	$\delta$ & Hall deficiency of $G$; $\delta = -m$ when $m < 0$ & Def.~\ref{def:characteristic} \\[6pt]
	\multicolumn{3}{l}{\textit{Grid model}} \\[2pt]
	$T_0 = B \times A$ & Ambient grid of $G_0$ & Sec.~\ref{sec:grid} \\
	$T$ & Arbitrary brick (sub-grid of $T_0$) & Def.~\ref{Brick} \\
	$\mathrm{row}(T)$, $\mathrm{col}(T)$ & Row and column index sets of $T$ & Sec.~\ref{sec:definitions} \\
	$\mathrm{vstrip}(T)$, $\mathrm{hstrip}(T)$ & Vertical and horizontal strips of $T$ & Sec.~\ref{sec:definitions} \\
	$\mathrm{vcount}(T)$ & $|\mathrm{row}(T)| + |\mathrm{col}(T)|$; vertex count of $T$ & Sec.~\ref{sec:definitions} \\
	$\mathrm{imb}(T)$ & $\bigl||\mathrm{col}(T)| - |\mathrm{row}(T)|\bigr|$; imbalance of $T$ & Sec.~\ref{sec:prelim} \\
	$\mathrm{long}(T)$ & $\max(|\mathrm{row}(T)|, |\mathrm{col}(T)|)$; longer side & Sec.~\ref{sec:grid} \\[6pt]
	\multicolumn{3}{l}{\textit{Characteristic and excess}} \\[2pt]
	$m$ & Characteristic of $T$; $m = |\mathrm{row}(T)| - \alpha_p(T)$ & Def.~\ref{def:characteristic} \\
	$\alpha_p(T)$ & Maximum proper independent set size of $T$ & Def.~\ref{def:characteristic} \\
	MH & Maximum hole in $T$ & Sec.~\ref{sec:grid} \\
	MPIS & Maximum proper independent set & Def.~\ref{def:characteristic} \\[6pt]
	\multicolumn{3}{l}{\textit{Decomposition}} \\[2pt]
	BTF & Block triangular form & Def.~\ref{def:BTF} \\
	$D$, $C_1,\ldots,C_k$, $E$ & Blocks of the excessive decomposition & Def.~\ref{def:exc-decomp-prelim} \\
	$D_1,\ldots,D_p$ & Atomic sub-blocks of $D$ (connected components) & Def.~\ref{def:atom-decomp-prelim} \\
	$E_1,\ldots,E_q$ & Atomic sub-blocks of $E$ (connected components) & Def.~\ref{def:atom-decomp-prelim} \\
	\bottomrule
	\end{longtable}
	\end{center}

\section{Main Results}\label{sec:main-results}

	The theorems below are stated using the preliminary definitions above
	(Definitions~\ref{def:BTF}--\ref{def:dm-irred-prelim}). Proofs of the constituent lemmas
	(Lemma~\ref{lem:balanced-excessive-atomic}, Lemma~\ref{lem:excessive-components},
	Corollary~\ref{cor:forbidden-dot}, Lemma~\ref{lem:matching-union})
	are given in Sections~\ref{sec:grid}--\ref{sec:excessive-decomp}.

	Theorem~\ref{thm:main_decomp} records the canonical decomposition into atomic blocks in the form in which it will be used below. As explained in Section~\ref{sec:related}, its blocks are the classical elementary components and its property~(4) is the classical factorization of maximum matchings over them; we restate it here in grid-model language, and because the later sections refer to it constantly, rather than as a new result. The results of this paper proper are the characteristic $m$ and its arithmetic (Theorem~\ref{thm:excess_fundamental} and Proposition~\ref{prop:excess-components}), the extendability property (Proposition~\ref{extendable}), and the classification into eleven structural classes (Theorem~\ref{thm:main_classification}). Corollary~\ref{cor:MPIS}, which describes the maximum proper independent sets by ideals of the $C$-block poset, is the classical ideal characterization \cite[Theorem~2.9]{Kita2016} in grid-model form; we include it with proof because the bijection with maximum holes is what the grid model makes visible, and because Section~\ref{sec:consequences} builds on it.
	
	\begin{theorem}[Fundamental Properties of Excessiveness]\label{thm:excess_fundamental}
		The excessive property provides a unified characterization of matching-related structural features:
		\begin{enumerate}[label=(\roman*)]
			\item \textbf{Generalization of expansion:} If a bipartite graph $G = (A,B,E)$ is $k$-expanding ($|N(X)| \ge |X|+k$ for all $X \subseteq A$), then $G$ is at least $m$-excessive with $m \geq k$.
			\item \textbf{Equivalence to the Strong Marriage Condition:} A balanced bipartite graph $G$ has positive excess ($m > 0$) if and only if it satisfies the Strong Marriage Condition.
			\item \textbf{Excessiveness as extendability and indecomposability:} Excessiveness is
			\emph{quantitatively} an extendability property: if $G$ is $m$-excessive, then $G$ is
			$m$-extendable --- every matching of size $m$ extends to a maximum matching, with the
			\emph{same} $m$ (Proposition~\ref{extendable}). For a connected bipartite graph $G$,
			positive excess ($m > 0$) is moreover equivalent to $m$-extendability for some $m > 0$
			(Proposition~\ref{prop:atomic-iff-connected}), and also to indecomposability: $G$ admits no partition into two or more excessive sub-bricks (Proposition~\ref{prop:atomic-minimal}). Since \emph{atomic} is defined as connected and excessive, these give two non-definitional characterizations of atomicity: a matching-theoretic one (extendability) and a structural one (minimal indecomposability under the excessive decomposition).
		\end{enumerate}
	\end{theorem}
	
	\begin{theorem}[Main theorem: Canonical Atomic Decomposition]\label{thm:main_decomp}
		Every finite, connected, simple bipartite graph $G_0$ admits a unique canonical decomposition into induced subgraphs, called \emph{atomic blocks}, arranged in a block triangular form (BTF). This decomposition has the following properties:
		\begin{enumerate}[label=(\arabic*),leftmargin=*]
			\item Each block is atomic: it is $m$-atomic for some $m > 0$, that is, connected
			and $m$-excessive. The value of $m$ may differ between blocks.
			\item The decomposition refines the classical Dulmage--Mendelsohn (DM) structure
			in two steps. First, the excessive decomposition (Section~\ref{sec:excessive-decomp})
			partitions $G_0$ into blocks $D, C_1,\ldots,C_k, E$ arranged in BTF, where each block
			is excessive. The square blocks $C_1,\ldots,C_k$ are balanced and excessive, hence
			connected and atomic (Lemma~\ref{lem:balanced-excessive-atomic}). The blocks $D$ and
			$E$ are excessive but not necessarily connected. Second, the atomic decomposition
			(Section~\ref{sec:atomic}) refines $D$ and $E$ by decomposing the induced subgraphs
			$G_0[D]$ and $G_0[E]$ into their connected components $D_1,\ldots,D_p$ and
			$E_1,\ldots,E_q$\linebreak respectively. By Lemma~\ref{lem:excessive-components}, each such
			component is unbalanced and excessive, hence atomic.
			\item Edges between atomic blocks from different regions may exist in $G_0$, but
			are forbidden with respect to every maximum matching of $G_0$.
			\item Every maximum matching of $G_0$ is the disjoint union of maximum matchings of the individual atomic blocks. Since each block is excessive, it satisfies the hole condition, hence admits a matching of size $\mathrm{small}(T)$ by the equivalence of the Hall and hole conditions (Corollary~\ref{cor:matching-from-hole}). No edges between blocks from different regions belong to any maximum matching (Corollary~\ref{cor:forbidden-dot}).
			\item The decomposition is invariant under the choice of maximum matching and under the choice of maximum hole in the excessive decomposition. In particular, the set of atomic blocks is uniquely determined; the blocks are arranged in a BTF that is unique up to permutation of topologically incomparable blocks (Theorem~\ref{thm:canonical-unique}, Corollary~\ref{cor:atomic-unique}).
		\end{enumerate}
	\end{theorem}

	This theorem is proved as Theorem~\ref{thm:atomic-decomp} in Section~\ref{sec:atomic}; its two refinement steps are constructed in Sections~\ref{sec:excessive-decomp} and~\ref{sec:atomic}, and the matching factorization of property~(4) is isolated as Lemma~\ref{lem:matching-factorization}.
	
	\begin{corollary}[Structural Symmetry of $D$ and $E$]\label{cor:DE-symmetry}
		The atomic blocks of $D$ and $E$ are the connected components of $G_0[D]$ and $G_0[E]$ respectively; in both cases they are unbalanced $m$-excessive bricks of the same structural type, differing in orientation ($D$-blocks are horizontal, $E$-blocks are vertical) and in their imbalance parameters. The decompositions are canonical and independent of the choice of maximum matching. The precise parameter relations (including $\mathrm{imb}(D) = \delta$ and $\mathrm{imb}(E) = (b-a)+\delta$) are recorded in Sections~\ref{sec:excessive-decomp} and~\ref{sec:atomic}.
	\end{corollary}

	The decomposition has two principal consequences. The first concerns the maximum proper independent sets; it is proved in Section~\ref{sec:consequences} (Corollary~\ref{cor:MPIS}).

	\begin{theorem}[Lattice of maximum proper independent sets]\label{thm:lattice-main}
		The family of maximum proper independent sets of $G_0$, ordered by the inclusion of their $A$-side projections, is isomorphic to the lattice of ideals of the $C$-block poset, and hence forms a distributive lattice.
	\end{theorem}

	The second consequence is the structural classification.
	
	As a consequence of the canonical decomposition, the bipartite graphs are organized into a finite taxonomy. The following classification records, for each combination of balance and characteristic, which blocks occur; the impossibility of the excluded combinations is part of the statement (proved in Theorem~\ref{thm:main_classification}).

	\begin{corollary}[Structural Classification]\label{cor:structural-classification}
		The universe of bipartite graphs with $|A| \leq |B|$ is partitioned into exactly eleven disjoint structural classes, according to the presence of $D$, $C$, and $E$ blocks in their canonical decomposition (Table~\ref{tab:block_composition}).
	\end{corollary}
	\begin{table}[h!]
		\centering
		\caption{Classification of the eleven structural classes based on balance, excessive-liminal-deficient type, and the number of $C$-blocks ($k$). The symbol ``--'' denotes structurally impossible combinations (see Theorem~\ref{thm:main_classification}).}
		\label{tab:block_composition}
		\renewcommand{\arraystretch}{1.5}
		\resizebox{\textwidth}{!}{%
			\begin{tabular}{|l|c|c|c|c|c|c|}
				\hline
				\multirow{2}{*}{\textbf{Characteristic}} & \multicolumn{3}{c|}{\textbf{Balance ($|A|=|B|$)}} & \multicolumn{3}{c|}{\textbf{Imbalance ($|A|<|B|$)}} \\ \cline{2-7} 
				& $\mathbf{k=0}$ & $\mathbf{k=1}$ & $\mathbf{k>1}$ & $\mathbf{k=0}$ & $\mathbf{k=1}$ & $\mathbf{k>1}$ \\ \hline
				\textbf{Excessive  ($m>0$)} & -- & $C$ & -- & $E$ & -- & -- \\ \hline
				\textbf{Liminal ($m=\delta=0$)} & -- & -- & $C_1, \dots, C_k$ & -- & $C, E$ & $C_1, \dots, C_k, E$ \\ \hline
				\textbf{Deficient ($-m=\delta>0$)} & $D, E$ & $D, C, E$ & $D, C_1, \dots, C_k, E$ & $D, E$ & $D, C, E$ & $D, C_1, \dots, C_k, E$ \\ \hline
			\end{tabular}%
		}
	\end{table}
	
	While independent of any specific representation, these results admit a natural geometric interpretation via the grid model. Figure~\ref{fig:gridmodel} illustrates the canonical block triangular arrangement of the $D, C$, and $E$ regions, serving as a visual roadmap for the structural analysis developed in Section~\ref{sec:grid}.
		
	\begin{figure}[htb]
		\centering
		\includegraphics[width=\linewidth,keepaspectratio]{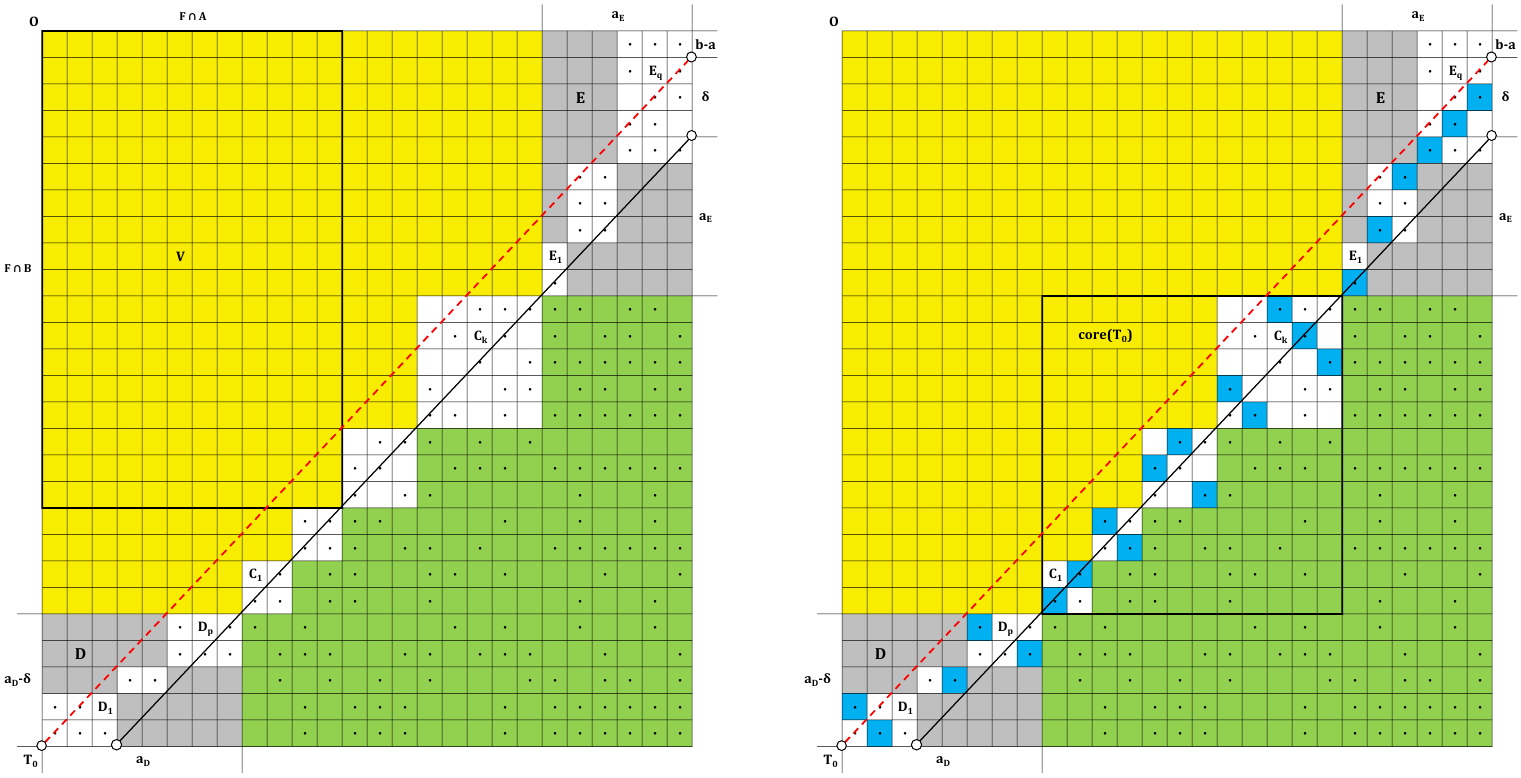}
		\caption{The canonical atomic decomposition in the grid model.}\label{fig:gridmodel}
	\end{figure}

	\noindent\textbf{Figure~\ref{fig:gridmodel} -- Visual Indicators:}
	\begin{itemize}[leftmargin=1.5em, itemsep=2pt]\small
		\item \textbf{Yellow:} Maximum holes (MPISs); no edge of $G_0$ appears here. Their common $\mathrm{vcount}$ $b+\delta$ determines the characteristic $m$ of $T_0$, where $m = -\delta$.
		\item \textbf{White:} Atomic blocks: square $C_1,\ldots,C_k$ (balanced, atomic by Lemma~\ref{lem:balanced-excessive-atomic}), and components $D_1,\ldots,D_p$, $E_1,\ldots,E_q$ of $G_0[D]$, $G_0[E]$ (unbalanced, atomic by Lemma~\ref{lem:excessive-components}).
		\item \textbf{Blue:} Edges of a maximum matching of $G_0$; every maximum matching is the union of maximum matchings of the atomic blocks.
		\item \textbf{Grey:} Absent edges between distinct atomic blocks within the same $D$- or $E$-region; their presence would contradict the connectivity of the components.
		\item \textbf{Green:} Edges between atomic blocks from different regions; these exist in $G_0$ but are forbidden with respect to every maximum matching.
	\end{itemize}
	\FloatBarrier

\section{The Grid Model}\label{sec:grid}
	
	We represent a bipartite graph $G_0 = (A,B,E)$ on the grid $T_0 := B \times A$, mapping each edge $(a,b) \in E$ to a dot at position $(b,a)$. This geometric representation encodes the adjacency structure as a dot set, enabling a structural interpretation of the graph in terms of axis-aligned subrectangles, termed \emph{bricks}.
	
	Unlike the classical Dulmage--Mendelsohn (DM) decomposition or the BTF approach of Pothen and Fan \cite{PothenFan1990}, which are built upon maximum matchings and alternating paths, the grid model is grounded directly in \emph{maximum holes} (largest dot-free subrectangles). In this framework, concepts such as the Strong Marriage Condition (SMC) and strong connectivity emerge as intrinsic structural properties rather than algorithmic primitives. Although the grid $T_0$ coincides as an array with the bipartite adjacency (bi-)matrix used in the matrix-theoretic treatment of the DM decomposition \cite{Murota2000}, the grid model is not merely a restatement of that view but a geometric reinterpretation of it: its primitive objects are not matrix entries but \emph{holes} (dot-free sub-rectangles), and the characteristic, the excessive property, and the BTF are all defined through maximum holes rather than through rank or matching. Table~\ref{tab:comparison} highlights the conceptual shifts that motivate this representation.
	
	\begin{table}[htbp]
		\centering
		\caption{Conceptual comparison of structural frameworks.}
		\label{tab:comparison}
		\small
		\begin{tabular*}{\textwidth}{@{\extracolsep{\fill}} l lll}
			\toprule
			& \textbf{DM Theory} & \textbf{Pothen (BTF)} & \textbf{Grid Model} \\
			\midrule
			Starting point      & Max. matching   & Directed graph  & Maximum holes \\
			Alternating paths   & Fundamental     & Implicit        & Implicit \\
			$D$-part refinement & None            & Algorithmic     & Structural \\
			Matching dependence & Invariant       & Invariant       & Canonical \\
			Motivation          & Theoretical     & Numerical       & Structural \\
			\bottomrule
		\end{tabular*}
	\end{table}
	\subsection{Detailed Definitions}\label{sec:definitions}
	
	\begin{definition}[Brick]\label{Brick}
		A \emph{brick} is a rectangular subgrid of $T_0$, that is, a set of the form
		\[
		T = Y \times X,
		\]
		where $Y \subseteq B$ and $X \subseteq A$ are nonempty, together with the dots contained in these cells. 
		The length of each side equals the number of unit squares it spans.
	\end{definition}
	
	\begin{definition}[vcount of a brick]
		For a brick $T$, let $\mathrm{vcount}(T)$ denote the total number of rows and columns spanned by $T$, so that if $T$ has size $r\times s$, then $\mathrm{vcount}(T)=r+s$.
		If it is clear which brick is meant, we simply write $\mathrm{vcount}$.
	\end{definition}
	
	\begin{definition}[Cells and dots]
		Each pair $(b,a) \in B \times A$ is called a \emph{cell}. 
		A \emph{dot} is placed in the cell $(b,a)$ if and only if $(a,b) \in E$.
	\end{definition}
	
	We denote by $T_0 = B \times A$ the ambient $b \times a$ grid, 
	whose rows are indexed by $B$ and columns by $A$. Dots represent the edges of $G_0$.
	
	\begin{definition}[Isotopes]
		Rows and columns (together with their corresponding vertices) may be permuted arbitrarily without changing the underlying bipartite graph. 
		Different representations obtained by such permutations are called \emph{isotopes} of each other.
	\end{definition}
	
	\begin{definition}[Induced brick]\label{InducedBrick}
		Let $Y\subseteq B$ and $X\subseteq A$. The set of cells $Y \times X$, together with the dots contained in them, is called the \emph{brick induced by} the vertex set $Y \cup X$.
	\end{definition}
	
	By the notation convention, $T$ serves as both the brick and its corresponding induced subgraph; we employ standard graph-theoretic notation for the latter, for example: $V(T)$, $E(T)$, $\alpha(T)$, etc.
	
	\begin{definition}[Visible Brick]
		A brick $T$ is called \emph{visible} if its rows and columns appear in consecutive positions. If, in addition, its top-left corner is at the origin, we say that $T$ is \emph{origin-aligned}.
	\end{definition}
	We will often assume that a given brick is made visible and origin-aligned by permuting rows and columns. Throughout, phrases such as ``move $V$ to the origin'' or ``$V$ permuted to the origin'' refer formally to this operation: the application of a row permutation and a column permutation (an isotope, in the sense above) that places the rows and columns of $V$ in consecutive positions starting at the origin. Such permutations do not change the underlying bipartite graph.
	
	\begin{definition}[Rows, columns, and strips of a brick]
		For a brick $T$ induced by a vertex set $W \subseteq A \cup B$, define
		\[
		\mathrm{row}(T) := W \cap B, \quad \mathrm{col}(T) := W \cap A,
		\]
		\[
		\mathrm{vstrip}(T) := B \times \mathrm{col}(T), \quad 
		\mathrm{hstrip}(T) := \mathrm{row}(T) \times A.
		\]
	\end{definition}
	
	\begin{definition}[Proper brick, hole]
		Let $T$ be a brick.
		We call $T$ a \emph{proper brick} if $\mathrm{row}(T) \subsetneq B$ and $\mathrm{col}(T) \subsetneq A$.
		A proper brick $T$ is called a \emph{hole} if it contains no dots.
	\end{definition}
	
	Thus, a hole corresponds to a zero submatrix in the adjacency matrix.
	
	\begin{remark}
		By Definition~\ref{Brick}, any brick $T$ in $T_0$ satisfies $\mathrm{row}(T) \neq \emptyset$ and $\mathrm{col}(T) \neq \emptyset$, 
		so degenerate bricks do not occur. 
		In general, one may have $\mathrm{row}(T) = B$ or $\mathrm{col}(T) = A$, 
		but this is not allowed for proper bricks or holes.
	\end{remark}
	
	\subsection{Proper Independent Sets and Proper Vertex Cover}
	Since an entire color class $A$ or $B$, while always independent, is structurally trivial and obscures Hall-type deficiencies, we restrict attention to independent sets meeting both color classes non-trivially. We call these \emph{proper independent sets}, abbreviated \textsc{PIS}; a maximum such set is an \textsc{MPIS}. This restriction excludes exactly the two trivial maximum independent sets $A$ and $B$; in particular, when $\delta > 0$ every maximum independent set is proper, so the MPISs coincide with the ordinary maximum independent sets.
	
	Since isolated vertices are excluded, every nonempty $X \subset A$ satisfies $N(X) \neq \emptyset$. 
	If $N(X) \subset B$, then $X \cup (B \setminus N(X))$ is a PIS of cardinality 
	\[
	|X \cup (B \setminus N(X))| = |X| + |B \setminus N(X)|.
	\]
	Since $b = |N(X)| + |B \setminus N(X)|$, we have
	\[
	|X| \,\circ\, |N(X)| \quad \Longleftrightarrow \quad
	|X| + |B \setminus N(X)| \,\circ\, b,
	\]
	for $\circ \in \{<, =, >\}$.
	
	Thus, there is a tight connection between the size of PISs and the length $b$ of the longer side of the grid $T_0$.
	
	\begin{definition}[Proper vertex cover]
		A vertex cover $L \subseteq A \cup B$ is called a \emph{proper vertex cover} if
		\[
		L \cap A \neq \emptyset, \quad L \cap B \neq \emptyset, \quad A \not\subseteq L, \quad B \not\subseteq L.
		\]
		
		Let $\tau_p(G_0)$ denote the minimum cardinality of a proper vertex cover, that is,
		\[
		\tau_p(G_0) = \min\{\, |L| : L \text{ is a proper vertex cover} \,\}.
		\]
	\end{definition}
	
	In the classical sense, where a vertex cover may contain an entire color class, 
	the value of $\tau(G_0)$ cannot exceed $a$, since the set $A$ itself is a vertex cover of size $a$. 
	
	In contrast, since $\tau_p(G_0)$ is defined only over proper vertex covers, it can be as large as $a+b-2$; for example, if $G_0$ contains exactly one pair of vertices (one from $A$ and one from $B$) that are not adjacent.
	
	We abbreviate: proper vertex cover as \textsc{PCS}, minimum proper vertex cover as \textsc{MPCS}.
	
	\par
	Proper independent sets and proper vertex covers are complements of each other within $V(G_0)$, hence the following relation between the two parameters holds:
	\[
	\alpha_p(G_0) + \tau_p(G_0) = |A| + |B|.
	\]

	\subsection{Matchings in the Grid Model}

	\begin{definition}
		A set of dots in an $r \times s$ brick $T$ ($s \leq r$) is called a \emph{matching} if any two distinct dots lie in different rows and different columns.
		A \emph{maximum matching} in $T$ is a matching of maximum cardinality.
	\end{definition}
	
		Clearly, the matchings of the graph $G$ and the matchings of the corresponding $r \times s$ brick $T$ are in one-to-one correspondence.
		A maximum matching in $G$ corresponds to a maximum matching in $T$, and vice versa.
	
		A maximum matching in $T$ can contain at most $s$ edges;
	otherwise, some column would contain two edges of the matching.
	
	In the brick representation, matching edges can be visualized by coloring their cells (typically blue).
	We say that two cells or two dots \emph{do not see each other} if they lie in different rows and different columns; for dots this means the corresponding edges share no endpoint, i.e., they are independent (a pair of matching edges).
	
	\subsection{Hall Condition and Hole Condition}
	\begin{definition}[Sub-brick-induced partition]
		Let $T$ be a proper brick. 
		Let $X = \mathrm{col}(T)$ and $Y = \mathrm{row}(T)$. 
		Then $X \subset A$ and $Y \subset B$ are nonempty, 
		and $T_0$ can be partitioned into four parts:
		\begin{itemize}
			\item $Y \times X$, the brick $T$ itself,
			\item $(B \setminus Y) \times X$, the \emph{vertical auxiliary brick} of $T$,
			\item $Y \times (A \setminus X)$, the \emph{horizontal auxiliary brick} of $T$,
			\item $(B \setminus Y) \times (A \setminus X)$, the \emph{remote mate} of $T$.
		\end{itemize}
		This decomposition is called the \emph{partition induced by the brick $T$}.
	\end{definition}	
	Figure~\ref{GM_PDF02} illustrates the partition induced by a hole $V$, 
	with $X = \mathrm{col}(V)$ and $Y = \mathrm{row}(V) = B \setminus N(X)$. 
	The four parts are denoted $V$, $T_1$, $T_2$, and $U$, respectively.
	
	\begin{proposition}[MPIS--hole correspondence]\label{prop:mpis-hole-corr}
		A maximum proper independent set (MPIS) in $G_0$ and a hole in $T_0$ with maximum $\mathrm{vcount}$ are in one-to-one correspondence.
	\end{proposition}
	
	\begin{proof}
		\textbf{First direction.} 
		Let $W$ be an MPIS in $G_0$, and let $X = A \cap W$. 
		Since $W$ is proper, $X \subset A$ and $B \setminus N(X) \subset B$, 
		so $V = (B \setminus N(X)) \times X$ is a hole with $\mathrm{col}(V) = X$ and $\mathrm{row}(V) = B \setminus N(X)$.
		Note that $\mathrm{vcount}(V) = |X| + |B \setminus N(X)| = |W|$.
		If $Z$ is a hole with $\mathrm{vcount}(Z) > |W|$, then $W' = \mathrm{col}(Z) \cup \mathrm{row}(Z)$ is a proper independent set of size $\mathrm{vcount}(Z) > |W|$, contradicting the maximality of $W$. Hence $V$ is a maximum hole.
		
		\textbf{Second direction.} 
		Let $V$ be a hole in $T_0$ with maximum $\mathrm{vcount}$. Then $W = \mathrm{row}(V) \cup \mathrm{col}(V)$ is a proper independent set with $|W| = \mathrm{vcount}(V)$.
		If a proper independent set $Z$ had $|Z| > |W|$, then the brick $(Z \cap B) \times (Z \cap A)$ would be a hole with $\mathrm{vcount} = |Z| > \mathrm{vcount}(V)$, contradicting the maximality of $V$. Hence $W$ is an MPIS in $G_0$.
	\end{proof}
	
	Thus, the cardinality of a maximum proper independent set $W$ equals $\mathrm{vcount}(W)$.
	Holes are typically colored yellow. 
	A hole with maximum $\mathrm{vcount}$ is called a \emph{maximum hole}, MH for short.
	In Figure~\ref{GM_PDF02}, the $\mathrm{vcount}$ of the hole $V$ is
	\[
	\mathrm{vcount}(V) = |X| + |B \setminus N(X)|.
	\]
	\begin{figure}[htb]
		\centering
		\includegraphics[width=\linewidth,keepaspectratio]{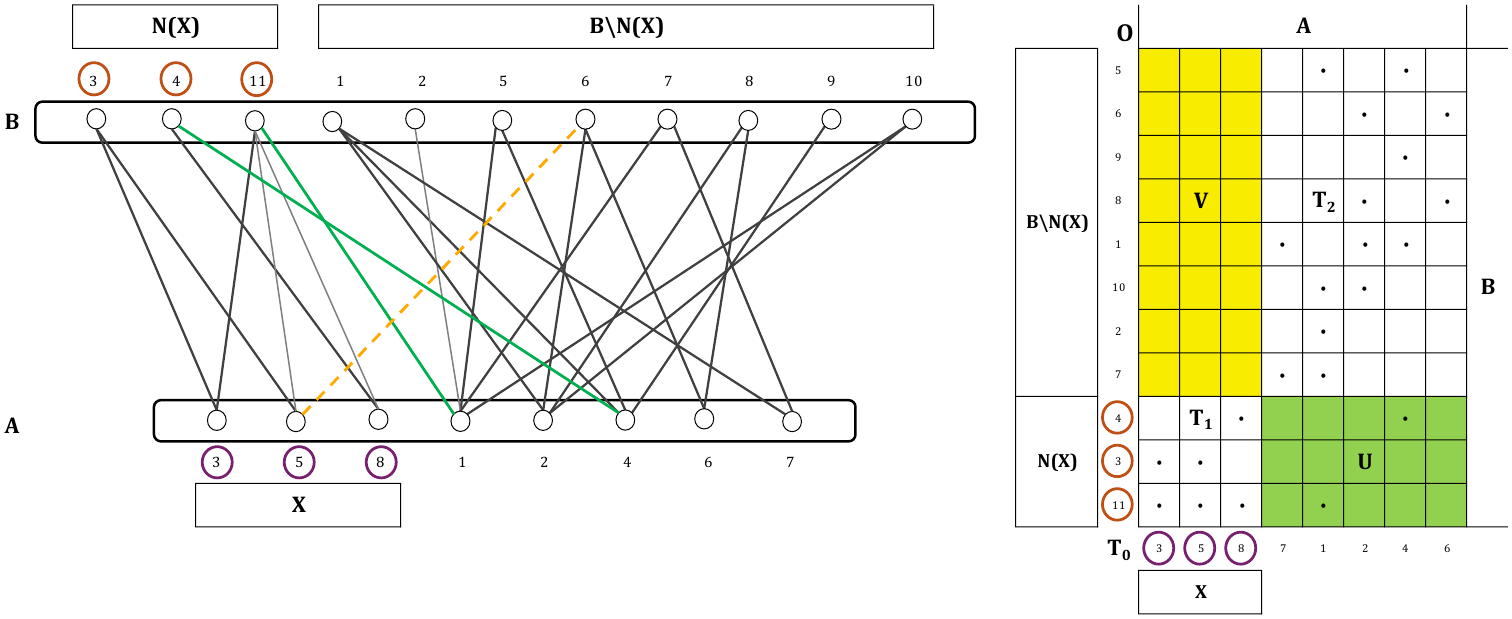}
		\caption{A hole with maximum $\mathrm{vcount}$ in the grid brick $T_0$}
		\label{GM_PDF02}
	\end{figure}
	
	Let $\alpha_p(T_0)$ denote the maximum $\mathrm{vcount}$ in the $b \times a$ grid brick $T_0$ corresponding to $G_0$. 
	Then there exists a hole $V$ such that 
	$
	\alpha_p(T_0) = \mathrm{vcount}(V)
	$,
	and consequently,
	$ 2 \leq \alpha_p(T_0) \leq a + b - 2. $
	\begin{corollary}
		By Proposition~\ref{prop:mpis-hole-corr}, we have
		$
		\alpha_p(T_0) = \alpha_p(G_0).
		$
		
	\end{corollary}
	
	For any brick $W$, define
	\[
	\mathrm{long}(W) := \max\{|\mathrm{row}(W)|, |\mathrm{col}(W)|\},
	\]
	\[
	\mathrm{small}(W) := \min\{|\mathrm{row}(W)|, |\mathrm{col}(W)|\}.
	\]

	Throughout, the ambient brick satisfies $a = |\mathrm{col}(T_0)| \le b = |\mathrm{row}(T_0)|$, so that $\mathrm{long}(T_0) = b$ and $\mathrm{small}(T_0) = a$; for a sub-brick $W$, $\mathrm{long}(W)$ and $\mathrm{small}(W)$ denote its own longer and shorter sides, which need not align with those of $T_0$.

	\begin{definition}[Hole condition]
		A brick $T$ satisfies the \emph{hole condition} 
		if it contains no hole with $\mathrm{vcount}$ greater than $\mathrm{long}(T)$.
	\end{definition}
	
	\begin{proposition}[Equivalence]
		The Hall condition \cite{Hall1935} and the hole condition are equivalent.
	\end{proposition}
	
	\begin{proof}
		\textbf{(Hall $\Rightarrow$ hole condition)}  
		Let the Hall condition hold, and let $V$ be any hole in $T_0$ with $X = \mathrm{col}(V)$. 
		Then $|N(X)| \ge |X|$, so 
		$|B \setminus N(X)| \le b - |X|$, and hence 
		$\mathrm{vcount}(V) = |X| + |B \setminus N(X)| \le b$.
		
		\textbf{(Hole condition $\Rightarrow$ Hall)}  
		Assume the hole condition holds. Let $X \subseteq A$ be arbitrary:
		
		\begin{enumerate}
			\item $X = \emptyset$: trivially, $|N(X)| = 0 \ge |X|$.
			\item $N(X) = B$: then $|N(X)| = |B| \ge |X|$ since $|B| \ge |A| \ge |X|$.
			\item $X \neq \emptyset$ and $N(X) \subset B$: set $V = (B \setminus N(X)) \times X$, which is a hole. By assumption, $\mathrm{vcount}(V) \le b$, i.e.,
			\[
			|X| + |B \setminus N(X)| \le b.
			\]
			Since $|N(X)| + |B \setminus N(X)| = b$, it follows that $|X| \le |N(X)|$.
		\end{enumerate}
		
		All cases are covered, so the hole condition implies the Hall condition.
	\end{proof}
	
	The following formulation remains valid even if the exact dimensions of the brick are not known.
	
	\begin{corollary}\label{cor:matching-from-hole}
		A brick $W$ contains a matching of size $\mathrm{small}(W)$ if and only if it contains no hole of $\mathrm{vcount}$ greater than $\mathrm{long}(W)$.
	\end{corollary}

	\begin{proof}
		Without loss of generality assume $W$ is vertical, that is, $|\mathrm{col}(W)| \leq |\mathrm{row}(W)|$; then $\mathrm{small}(W) = |\mathrm{col}(W)|$ and $\mathrm{long}(W) = |\mathrm{row}(W)|$. By Hall's theorem \cite{Hall1935}, $W$ contains a matching of size $|\mathrm{col}(W)|$ if and only if the Hall condition holds: $|N(X)| \geq |X|$ for all $X \subseteq \mathrm{col}(W)$. By the equivalence proved above, the Hall condition is equivalent to the hole condition: no hole of $\mathrm{vcount}$ greater than $|\mathrm{row}(W)| = \mathrm{long}(W)$.
	\end{proof}
	
	\begin{remark}
		Corollary~\ref{cor:matching-from-hole} is the central bridge used throughout: it converts every matching question into a question about hole sizes. It is the form in which the Hall/hole equivalence is applied to excessive blocks, to the BTF recursion, and to the matching factorization of Lemma~\ref{lem:matching-factorization}.
	\end{remark}

	\subsection{Hall Line, Characteristic}
	
	Let $T_0$ be a $b \times a$ brick with $a \le b$, and origin-aligned. 
	From the lower endpoint of the vertical side, draw a $45^\circ$ line passing through grid points at Manhattan distance $b$ from the origin; this is the \emph{Hall line}. 
	For a hole of $\mathrm{vcount} = b$ aligned to the origin by permutation, one of its corners lies at an interior grid point on this line. 
	The vertical auxiliary brick of a hole with $\mathrm{vcount} = b$ of size $(b-k)\times k$ is a square $k \times k$ ($T_1$ in Figure~\ref{GM_PDF03}); 
	the horizontal auxiliary brick $T_2$ is square iff $a=b$.
	
	\begin{figure}[htb]
		\centering
		\includegraphics[width=\linewidth,keepaspectratio]{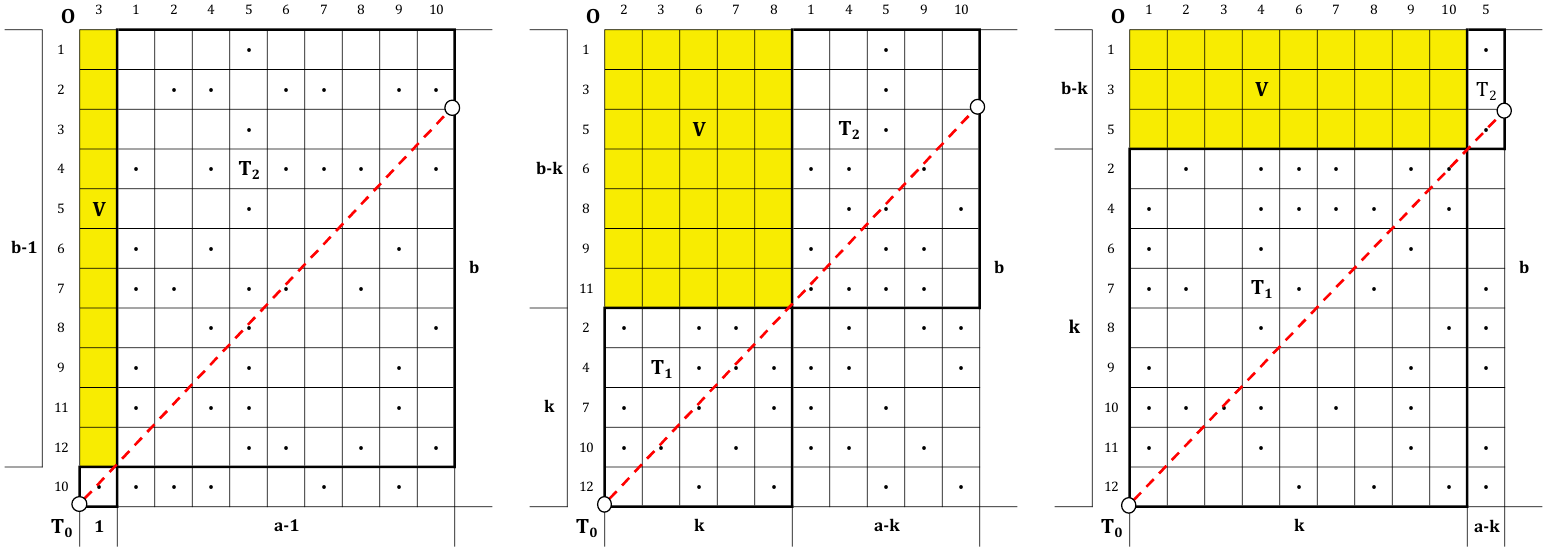}
		\caption{Hall line - Auxiliary bricks}\label{GM_PDF03}
	\end{figure}
	
	In the grid model the characteristic of $T_0$ (Definition~\ref{def:characteristic}) has a direct geometric reading. With $\alpha_p(T_0)$ the maximum hole $\mathrm{vcount}$ and $m = b - \alpha_p(T_0)$, draw a $45^\circ$ line at Manhattan distance $\alpha_p(T_0)$ from the origin; its portion inside $T_0$ is the \emph{characteristic segment} (a line segment). The sign of $m$ distinguishes the deficient ($m<0$), liminal ($m=0$), and excessive ($m>0$) cases as before.
	The characteristic segment indicates grid points relevant for maximum holes or MPISs, including the origin if applicable, excluding boundary points. 
	Bricks are classified by $m$: $m=0$ is \emph{liminal}, $m<0$ \emph{deficient}, $m>0$ \emph{$m$-excessive}, with the magnitude of deficit/excess given by the Manhattan distance from the Hall line. 
	These lines remain fixed under row/column permutations; grid points on the frame are not allowed except the origin.
	
	\begin{figure}[htb]
		\centering
		\includegraphics[width=\linewidth,keepaspectratio]{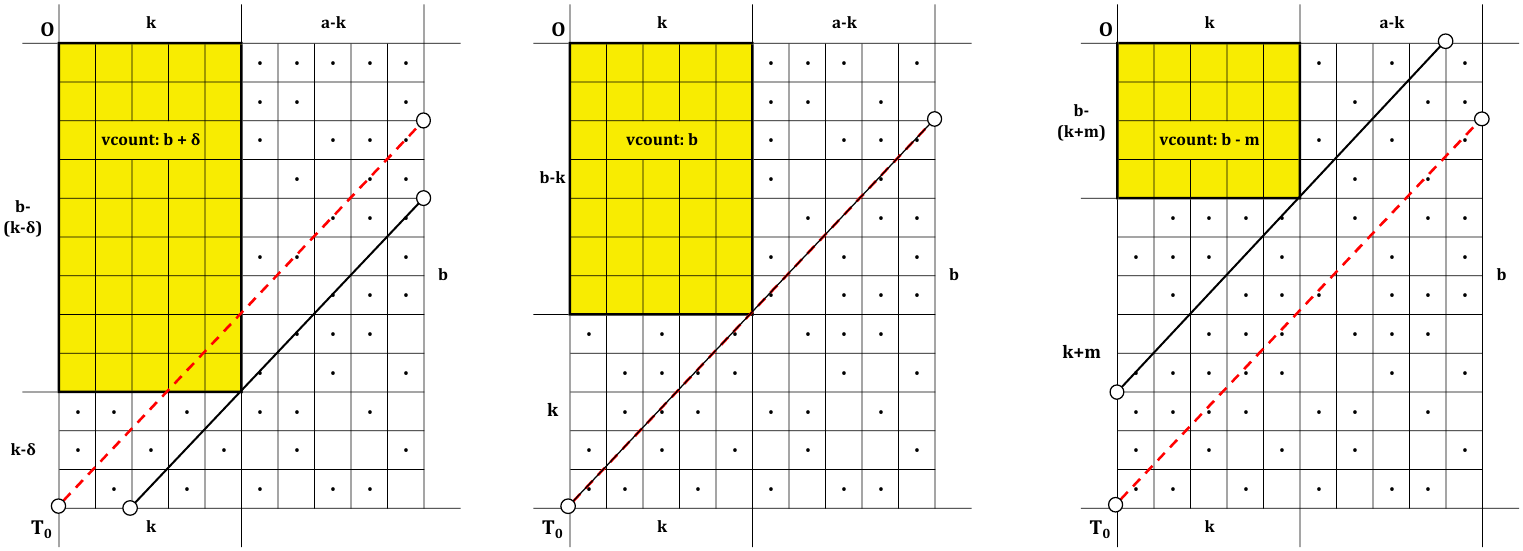}
		\caption{Hall line - characteristic segment}\label{GM_PDF04}
	\end{figure}
	
	For non-excessive bricks ($m \le 0$), define $\sigma_A(G_0) = \max_{X \subseteq A}(|X| - |N(X)|)$ and $\sigma_B(G_0) = \max_{Y \subseteq B}(|Y| - |N(Y)|)$,
	called the \emph{left} and \emph{right deficiency} of $G_0$, respectively.

	\begin{lemma}
		Let $T_0$ be non-excessive with deficit $\delta=-m\ge0$. 
		For a maximum hole $V$ with columns $X$ and rows $Y$, 
		\[
		|X|-|N(X)| = \delta \quad\text{and}\quad |Y|-|N(Y)| = (b-a)+\delta,
		\] 
		so $\sigma_A(G_0)=\delta$, $\sigma_B(G_0)=(b-a)+\delta$.
	\end{lemma}
	
	\begin{proof}
		Assume $m \le 0$, so that $\delta = -m \ge 0$. 
		Let $V$ be an MH of size $b + \delta$ in $T_0$, and let $X = \mathrm{col}(V)$, $Y = \mathrm{row}(V)$. Then
		\[
		|X|-|N(X)| = |X| + |B\setminus N(X)| - b = \delta.
		\]
		Since $V$ is a hole, no dot lies in $Y \times X$, so $N(Y) \cap X = \emptyset$, i.e., $N(Y) \subseteq A \setminus X$. Moreover, every vertex in $A \setminus X$ is adjacent to some vertex in $Y = B \setminus N(X)$ (otherwise the hole $V$ could be enlarged, contradicting its maximality), hence $N(Y) = A \setminus X$ and $|N(Y)| = a - |X|$. Since $|Y| = b - |N(X)| = b + \delta - |X|$, we obtain
		\[
		|Y|-|N(Y)| = (b + \delta - |X|) - (a - |X|) = (b-a) + \delta.
		\]

		Boundary cases ($X=\emptyset$, $X=A$, $Y=\emptyset$, $Y=B$) yield smaller values, 
		so MHs attain the maximal values:
		\[
		\sigma_A(G_0) = \delta \quad \text{and} \quad \sigma_B(G_0) = (b-a) + \delta.
		\]
	\end{proof}
	The difference satisfies the well-known equality:
	\[
	\sigma_B(G_0) - \sigma_A(G_0) = b - a.
	\]
	
	\begin{corollary}
		For a maximum hole $V$, the vertical auxiliary brick $T_1$ is a horizontal brick with $\mathrm{imb}(T_1)=\delta$, and the horizontal auxiliary brick $T_2$ is a vertical brick with $\mathrm{imb}(T_2)=(b-a)+\delta$. Maximal A and B side deficiencies appear in pairs along the MH sides, except for trivial cases with no hole ($X=\emptyset$, $X=A$, $Y=\emptyset$, $Y=B$).
	\end{corollary}
	
	\begin{remark}
		Our definitions of deficiency and $m$-excess are given in terms of the grid structure, but coincide with the classical quantities (deficiency and surplus) via the identities $|X|-|N(X)|$ and $|N(X)|-|X|$.

		A key feature of the hole-based approach is that $\mathrm{vcount}(V) = |\mathrm{row}(V)| + |\mathrm{col}(V)|$ is symmetric in the two color classes: it places rows and columns on equal footing, and the single invariant $m = |\mathrm{row}(T_0)| - \alpha_p(T_0)$ captures the structural imbalance without distinguishing the two sides. In contrast, the classical Hall condition $|N(X)| \geq |X|$ and the Strong Marriage Condition $|N(X)| \geq |X|+1$ explicitly distinguish the two sides via $X \subseteq A$. The grid model thus provides a representation in which the asymmetry between the two color classes does not appear at the level of the fundamental structural invariant, though it remains present in the orientation of the brick ($|\mathrm{col}(T_0)| \leq |\mathrm{row}(T_0)|$).

		More importantly, the grid model collapses the \emph{branching} of the classical theory into the sign of a single quantity. The deficient, liminal, and excessive regimes are classically governed by three separate conditions --- a Hall deficiency $|N(X)| < |X|$, K\H{o}nig-type equality, and positive surplus (the Strong Marriage Condition $|N(X)| \geq |X|+1$ in the balanced case) --- whereas here all three are the one statement that the maximum hole $\mathrm{vcount}$ is, respectively, greater than, equal to, or less than $b$. Replacing this case analysis by the sign of $m = b - \alpha_p(T_0)$ is precisely what allows the excessive and atomic decompositions, and the excessiveness of the connected components of $D$ and $E$ (Lemma~\ref{lem:excessive-components}), to be read off geometrically rather than re-derived condition by condition.
	\end{remark}
	
	\subsection{Excessive -- Expanding, Matchings}
	The $m$-expanding property requires $|N(X)| \ge |X| + m$ for all nonempty $X \subseteq \mathrm{col}(T_0)$, and is only meaningful when $\mathrm{imb}(T_0) \ge m$. The following definition, proposition, and corollary show that the $m$-excessive property is a strictly weaker, more general condition.

	\begin{definition}[$m$-expanding]\label{def:m-expanding}
		A brick $T_0$ is \emph{$m$-expanding} if $|N(X)| \ge |X| + m$ for every nonempty $X \subseteq \mathrm{col}(T_0)$.
	\end{definition}

	\begin{corollary}\label{cor:excessive-generalizes-expanding}
		The $m$-excessive property generalizes the $m$-expanding property, and the inclusion is strict: there exist $m$-excessive bricks that are not $m$-expanding. Figure~\ref{GM_PDF05} illustrates this: all four bricks are 3-excessive, but only the 3rd and 4th satisfy the 3-expanding property. In particular, $m$-excessiveness does not imply $m$-expanding whenever $m < \mathrm{imb}(T_0)$.
	\end{corollary}
	
	\begin{figure}[htb]
		\centering
		\includegraphics[width=\linewidth,keepaspectratio]{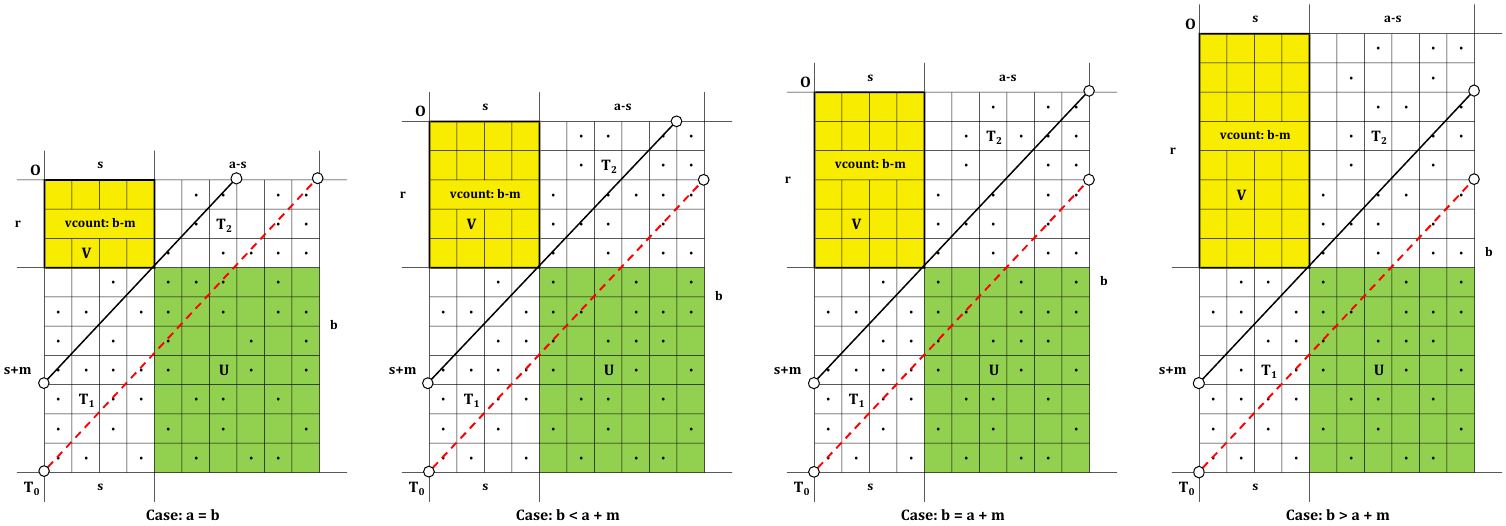}
		\caption{Excessive - expanding}\label{GM_PDF05}
	\end{figure}
	
	\begin{proposition}\label{excessive_expanding}
		Let $T_0$ be an unbalanced $m$-excessive brick, and define
		\[
		m_0 = \min\{m, \mathrm{imb}(T_0)\}.
		\]
		Then $T_0$ is $m_0$-expanding. Moreover, if $m \ge \mathrm{imb}(T_0)$, it is strictly $\mathrm{imb}(T_0)$-expanding.
	\end{proposition}
	
	\begin{proof}
		Consider any subset $X \subseteq A$:
		\begin{itemize}
			\item If $N(X) = B$, then $|N(X)| - |X| \ge |B| - |A| = \mathrm{imb}(T_0)$.
			\item If $N(X) \subset B$, then there exists a hole, so 
			$\mathrm{vcount}((B \setminus N(X)) \times X) \le b - m$, 
			which implies $b - |N(X)| + |X| \le b - m$, i.e., $|N(X)| \ge |X| + m$.
		\end{itemize}
		Hence, for all $X \subseteq A$, $|N(X)| - |X| \ge m_0$, and $T_0$ is $m_0$-expanding. 
		Furthermore, if $\mathrm{imb}(T_0) \le m$, then for $X = A$, 
		\[
		|N(X)| = |X| + \mathrm{imb}(T_0),
		\]
		i.e., $T_0$ is strictly $\mathrm{imb}(T_0)$-expanding.
	\end{proof}
	
	\begin{remark}\label{rem122}
		For a balanced graph $G$, the Strong Marriage Condition 
		\[
		|N(X)| \ge |X| + 1 \quad \text{for any non-empty } X \subset A
		\] 
		is equivalent to the corresponding $T_0$ brick being excessive, 
		since both conditions prohibit a maximum hole (MH) with $\mathrm{vcount}$ at least $n$, where $n = |A| = |B|$.
	\end{remark}
	
	\begin{lemma}\label{auxiliary_brick}
		Let $V$ be an MH in $T_0$. Then its two auxiliary bricks $T_1$ and $T_2$ are not deficient.
	\end{lemma}
	
	\begin{proof}
		Let $V$ be an $r \times s$ MH permuted to the origin, with its interior vertex lying on the characteristic segment. Let $T_1$ be the vertical and $T_2$ the horizontal auxiliary brick of $V$, and let $U$ denote the remote mate. 
		
		Consider the top-left corners of $T_1$ and $T_2$ as local origins. Let $V_1$ be an $r_1 \times s_1$ MH in $T_1$ and $V_2$ an $r_2 \times s_2$ MH in $T_2$. By the definition of a hole:
		\[
		r_1 < s + m, \quad s_1 < s, \quad r_2 < s, \quad s_2 < r + s.
		\]
		
		Then the combined bricks 
		\[
		(r+r_1) \times s_1 \quad \text{and} \quad r_2 \times (s+s_2)
		\]
		are themselves holes. Since $V$ is a maximum hole, their $\mathrm{vcount}$s cannot exceed that of $V$, i.e.,
		\[
		r + r_1 + s_1 \le r + s \quad \text{and} \quad r_2 + s + s_2 \le r + s,
		\]
		or equivalently,
		\[
		r_1 + s_1 \le s \quad \text{and} \quad r_2 + s_2 \le r.
		\]
		
		Hence, any hole in $T_1$ or $T_2$ has $\mathrm{vcount}$ at most equal to the longer side of the respective brick. Therefore, neither $T_1$ nor $T_2$ can be deficient.
	\end{proof}
	
	Intuitively, the portion of the characteristic segment inside $T_1$ lies at distance $s$ from the origin of $T_1$, so the interior vertex of any hole permuted to the origin cannot lie farther than this segment. Similarly, the portion inside $T_2$ lies at distance $r$ from the origin of $T_2$, restricting the hole's internal vertex within $T_2$.
	
	Moreover, the characteristic segment coincides with the Hall line of $T_1$ if the side of $T_1$ shared with $V$ is at least as long as its other side, and with the Hall line of $T_2$ if the side of $T_2$ shared with $V$ is at least as long as its other side (see Figure~\ref{GM_PDF06}).
	In other cases, the Hall line of the brick is farther from the origin than the characteristic segment. Thus, for both auxiliary bricks of $V$, either the excessive or the liminal property holds.
	
	\begin{figure}[htb]
		\centering
		\includegraphics[width=\linewidth,keepaspectratio]{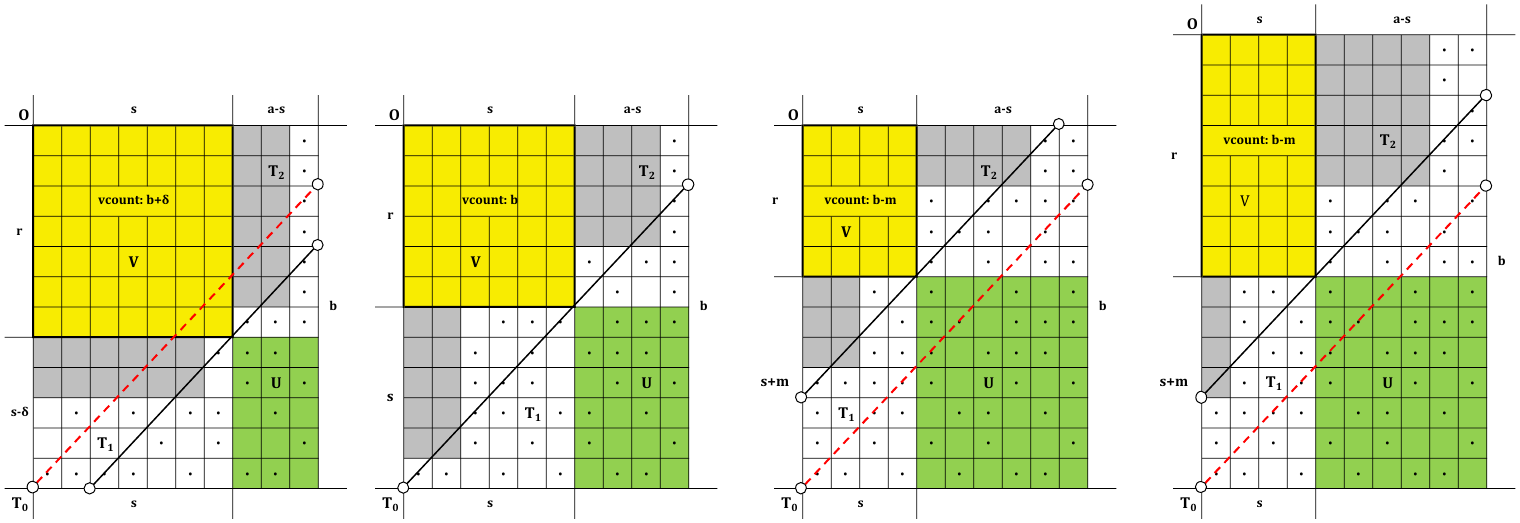}
		\caption{Non-deficient auxiliary bricks}\label{GM_PDF06}
	\end{figure}
	
	\begin{corollary}
		For a maximum hole $V$ in $T_0$, its auxiliary bricks $T_1$ and $T_2$ satisfy the hole condition, and hence also the Hall condition.
	\end{corollary}
	
	Consequently, if $T_0$ is excessive (Figure~\ref{GM_PDF06}, bricks 3 and 4), it contains a matching of size $a$, and all hole $\mathrm{vcount}$s are less than $b$.
	If $T_0$ is deficient ($\delta > 0$, brick 1) or liminal ($\delta = 0$, brick 2), the auxiliary bricks $T_1$ and $T_2$ still satisfy the Hall condition. 
	Therefore, $T_1$ contains a matching of size $s-\delta$, and $T_2$ contains a matching of size $a-s$. 
	The union of these matchings yields a matching of size
	\[
	(a-s) + (s-\delta) = a - \delta.
	\]
	
	\begin{lemma}\label{lem:matching-union}
		Let $V$ be a maximum hole (MH) in a non-excessive brick $T_0$, with vertical auxiliary brick $T_1$, horizontal auxiliary brick $T_2$, and remote mate $U$. Let $M$ be a matching in $T_0$. Then $M$ is a maximum matching in $T_0$ if and only if it is the union of a maximum matching $M_1$ in $T_1$ and a maximum matching $M_2$ in $T_2$.
	\end{lemma}
	
	\begin{proof}
		Let $V$ be an $r \times s$ MH. Since $T_0$ is non-excessive, its $\mathrm{vcount}$ satisfies $r+s = b+\delta$, with $\delta \ge 0$.
		
		\medskip
		(\emph{Forward direction})  
		Let $M_1$ be a maximum matching in $T_1$ of size $s-\delta$ and $M_2$ a maximum matching in $T_2$ of size $a-s$. Then their union $M = M_1 \cup M_2$ is a matching of size $a-\delta$, since no two elements see each other. It is maximum because:  
		\begin{itemize}
			\item The strip $U \cup T_2$ can contain at most $a-s$ matching elements.  
			\item The strip $U \cup T_1$ can contain at most $s-\delta$ matching elements.  
		\end{itemize}
		Any overlap $k$ in $U$ reduces the total size to $a-\delta-k$, so the maximum is achieved when $k=0$.
		
		\medskip
		(\emph{Reverse direction})  
		Suppose $M$ is a matching in $T_0$ with $k > 0$ elements in the remote mate $U$. Each element in $U$ is counted in both the strip $U \cup T_1$ and the strip $U \cup T_2$. Since the strip $U \cup T_2$ contains at most $a - s$ matching elements and the strip $U \cup T_1$ contains at most $s - \delta$ matching elements, and the $k$ elements in $U$ are counted in both strips, the total size of $M$ satisfies
		\[
		|M| \le (a-s) + (s-\delta) - k = a - \delta - k < a - \delta.
		\]
		Hence $M$ cannot be a maximum matching, and no element of $U$ can belong to a maximum matching in $T_0$.
	\end{proof}
	
	\begin{definition}
		Let $T$ be a brick. A dot in $T$ is called \emph{forbidden} if it does not appear in any maximum matching of $T$ (of size $\mathrm{small}(T)$).
	\end{definition}
	
	\begin{corollary}\label{cor:forbidden-dot}
		Let $T_0$ be a non-excessive brick. If a dot lies in the remote mate $U$ of some MH in $T_0$, then it is forbidden in $T_0$. This follows directly from the reverse direction of Lemma~\ref{lem:matching-union}.
	\end{corollary}
	
	\begin{corollary}
		Let $V$ be a hole with $\mathrm{vcount}$ $b+\delta$ in $T_0$ ($\delta \ge 0$). If neither of its auxiliary bricks $T_1$ and $T_2$ is deficient, then there exists a matching of size $a-\delta$, so that $V$ is an MH. Otherwise, there would exist a hole with $\mathrm{vcount}$ greater than $b+\delta$ in $T_0$, which would imply that the maximum matching size is less than $a-\delta$ (by Corollary~\ref{cor:matching-from-hole}), contradicting the assumption.
	\end{corollary}
	
	\begin{definition}[Covering brick]
		Let $W$ be a proper brick in $T_0$, and let 
		\[
		L = \mathrm{col}(W) \cup \mathrm{row}(W)
		\]
		be the union of its column and row indices. We say that $W$ is a \emph{covering brick} if every dot in $T_0$ intersects $L$, i.e., every dot belongs to at least one row or column of $W$. Equivalently, the union of the vertical strip $\mathrm{vstrip}(W)$ and the horizontal strip $\mathrm{hstrip}(W)$ covering $W$ contains all dots of $T_0$.
		
		The cardinality of $L$ equals $\mathrm{vcount}(W)$. If $W$ is a covering brick in $T_0$, then $L$ is a vertex cover in the graph $G$ corresponding to $T_0$, so $W$ is also called a \emph{proper cover set brick (PCS)}.
		
		A covering brick $W$ is a \emph{minimum covering brick} if no covering brick of smaller $\mathrm{vcount}$ exists. In this case, $L$ is a minimum vertex cover in $G$, and $W$ is also called a \emph{minimum proper cover set brick (MPCS)}.
	\end{definition}
	
	\begin{corollary}\label{cor:remote-mate-bijection}
		Clearly, the remote mate of a PIS is a PCS and vice versa; similarly, the remote mate of an MPIS is an MPCS and vice versa.
	\end{corollary}
	
	If the bricks $V$ and $U$ are remote mates of each other, then
	\[
	\mathrm{vcount}(V) + \mathrm{vcount}(U) = |A| + |B|.
	\]
	Hence, for the graph $G$ and its corresponding brick $T_0$, we have
	\[
	\alpha_p(G_0) + \tau_p(G_0) = |A| + |B|, \qquad
	\alpha_p(T_0) + \tau_p(T_0) = \mathrm{vcount}(T_0).
	\]
	
	\begin{corollary}
		In the non-excessive case $(\delta \ge 0)$, the remote mate of any MH $V$ is an MPCS of $\mathrm{vcount}$ $a-\delta$.
		By Lemma~\ref{lem:matching-union}, there exists a maximum matching of size $a-\delta$ partitioned by this MPCS. Therefore, for $\delta \ge 0$, we obtain the grid-model analogue of K\H{o}nig's theorem \cite{Konig1931}:
		\[
		\tau_p(T_0) = \nu(T_0).
		\]
		Furthermore, for any MPCS $U$, any maximum matching $M_0$ decomposes into a maximum matching $M_1$ in $T_1$ and a maximum matching $M_2$ in $T_2$, satisfying
		\[
		\mathrm{row}(T_1) = |M_1| = \mathrm{row}(U), \qquad
		\mathrm{col}(T_2) = |M_2| = \mathrm{col}(U).
		\]
	\end{corollary}
	
	\begin{figure}[htb]
		\centering
		\includegraphics[width=.8\linewidth,keepaspectratio]{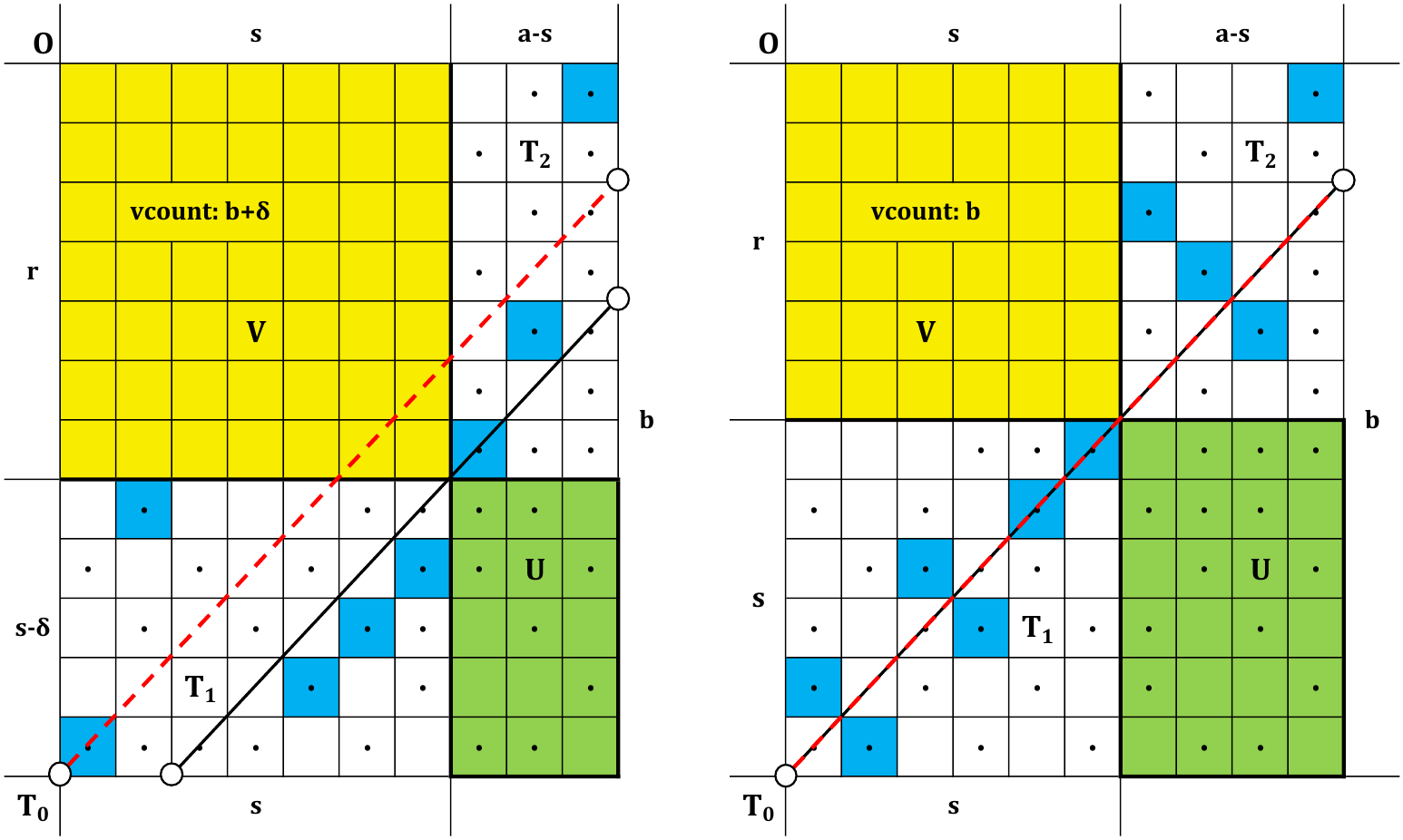}
		\caption{Grid-model analogue of K\H{o}nig's theorem for non-excessive bricks}\label{GM_PDF07}
	\end{figure}
	\FloatBarrier
	
	\subsection{Connectivity and Atomic Property}
	
	\begin{definition}
		A brick $T_0$ is \emph{connected} if it contains no sub-brick whose auxiliary bricks are holes.
	\end{definition}
	
	Clearly, $T_0$ is connected if and only if the corresponding bipartite graph $G$ is connected in the usual sense.
	
	\begin{definition}
		Let $T_0$ be a $b \times a$ brick ($a \le b$). We say that $T_0$ is \emph{$m$-extendable} if every matching of size $m$ can be extended to a maximum matching of size $a$.
	\end{definition}
	
	\begin{proposition}\label{extendable}
		Let $T_0$ be a $b \times a$ brick ($a \le b$) that is $m$-excessive. Then $T_0$ is $m$-extendable.
	\end{proposition}
	
	\begin{proof}
		If $m < a$, consider an arbitrary matching $M_m$ of size $m$. Let $S$ be the $m \times m$ square induced by $M_m$, and let $R$ be its remote mate, a brick of size $(b-m) \times (a-m)$.

		We claim that $R$ contains no hole with $\mathrm{vcount}$ greater than $b - m$. Indeed, since $S$ and $R$ occupy disjoint row and column indices, any hole in $R$ (a subrectangle of $R$ containing no dot) is also a hole in $T_0$: no dot from $S$ can appear in the row/column range of $R$, so the absence of dots is preserved in $T_0$. If $R$ contained a hole of $\mathrm{vcount} > b - m$, this would be a hole in $T_0$ of $\mathrm{vcount} > b - m$, contradicting the assumption that $T_0$ is $m$-excessive. Hence $R$ is non-deficient and contains a maximum matching $M_{a-m}$ of size $a-m$.

		The elements of $M_m$ and $M_{a-m}$ do not see each other (they lie in disjoint parts of $T_0$), so their union forms a matching of size $a$, proving that $T_0$ is $m$-extendable.
	\end{proof}
	
	\begin{remark}\label{rem:plummer}
		Extendability is Plummer's notion~\cite{Plummer1980}: a graph is \emph{$n$-extendable} if
		every matching of size $n$ extends to a \emph{perfect} matching. The classical definition
		therefore presupposes a graph carrying a perfect matching, and in the bipartite case this
		means a \emph{balanced} graph.

		The proof of Proposition~\ref{extendable} nowhere uses balance: it is carried out for an
		arbitrary $b\times a$ brick with $a\le b$, the perfect matching being replaced by a maximum
		matching of size $\mathrm{small}(T_0)=a$. Thus $m$-excessiveness is a sufficient condition
		for $m$-extendability, and it extends Plummer's notion canonically to unbalanced bipartite
		graphs. For $m=1$ it specialises to the classical statement that every edge lies in a maximum
		matching.

		This also answers the question raised after Definition~\ref{def:characteristic}, of why the
		characteristic $m$ rather than the Hall deficiency is the organizing invariant: besides
		localizing to sub-bricks, $m$ carries a matching-theoretic meaning. It is exactly the number
		of edges that may be chosen freely before extendability to a maximum matching can be lost.
	\end{remark}
	
	\begin{proposition}\label{prop:atomic-iff-connected}
		Let $T_0$ be a $b \times a$ brick ($a \le b$). Then $T_0$ is atomic if and only if it is connected and $m$-extendable for some $m > 0$.
	\end{proposition}

	\begin{proof}
		(\emph{Forward direction}) Suppose $T_0$ is atomic. By definition, it is connected and $m$-excessive for some $m > 0$. By Proposition~\ref{extendable}, $T_0$ is $m$-extendable.

		\medskip

		(\emph{Reverse direction}) Suppose $T_0$ is connected and $m$-extendable for some $m > 0$. Since $T_0$ is $m$-extendable, every dot can be extended to a matching of size $a$, so $T_0$ is not deficient. Hence $T_0$ is either liminal or excessive.

		Assume, for contradiction, that $T_0$ is liminal. Then it contains a hole $V$ of $\mathrm{vcount}$ $b$. Since $T_0$ is non-excessive ($m = 0$), Corollary~\ref{cor:forbidden-dot} applies: every dot in the remote mate $U$ of $V$ is forbidden (belongs to no maximum matching). Since $T_0$ is $m$-extendable ($m > 0$), every dot must belong to some maximum matching, so $U$ can contain no dots. Hence all dots of $T_0$ lie in the vertical auxiliary brick $T_1$ or the horizontal auxiliary brick $T_2$. Since $T_1$ and $T_2$ occupy disjoint strips with $V$ separating them, there are no edges between $V(T_1)$ and $V(T_2)$ in $T_0$, contradicting the assumption that $T_0$ is connected. Therefore, $T_0$ cannot be liminal, and since it is not deficient, it must be excessive. Together with connectedness, this means $T_0$ is atomic.
	\end{proof}
	
	\begin{remark}
		An atomic brick is at least $1$-extendable. Hence, for every dot there exists a maximum matching containing it, so no dot is forbidden.
	\end{remark}
	
	\begin{lemma}\label{lem:balanced-excessive-atomic}
		Every balanced excessive brick $T_0$ is connected, and therefore also atomic.
	\end{lemma}
	
	\begin{proof}
		The case $n=1$ (a single vertex on each side, with a single edge) is trivially connected and excessive, so we may assume $n \ge 2$.
		Suppose, for contradiction, that the $n\times n$ brick $T_0$ is not connected, and let $W_i$ be one of its components with $p$ columns and $q$ rows. Since $W_i$ is a connected component, no dot lies outside $W_i$ within its vertical strip, and no dot lies outside $W_i$ within its horizontal strip. Hence the vertical complement of $W_i$ in $T_0$ (of size $(n-q) \times p$) and the horizontal complement of $W_i$ in $T_0$ (of size $q \times (n-p)$) are both holes.

		Their $\mathrm{vcount}$s are $(n-q)+p$ and $q+(n-p)$ respectively, and their sum equals
		\[
		(n-q+p) + (q+n-p) = 2n.
		\]
		Since their two $\mathrm{vcount}$s sum to $2n$, at least one is at least the average $n = \mathrm{long}(T_0)$ --- a hole as large as the long side --- contradicting the assumption that $T_0$ is excessive. Therefore, $T_0$ must be connected, and by Proposition~\ref{prop:atomic-iff-connected}, it is atomic.
	\end{proof}
	
	\begin{corollary}\label{cor:SMC-connected}
		A balanced simple bipartite graph $G$ satisfying the Strong Marriage Condition is connected.
	\end{corollary}

	\begin{proof}
		By Remark~\ref{rem122}, the Strong Marriage Condition is equivalent to $G$ being excessive. By Lemma~\ref{lem:balanced-excessive-atomic}, every balanced excessive brick is connected.
	\end{proof}

	\begin{remark}
		This consequence of Lemma~\ref{lem:balanced-excessive-atomic} and Remark~\ref{rem122} follows immediately from the theory of elementary bipartite graphs \cite{LovaszPlummer1986}, where an elementary graph is defined as connected and such that every edge belongs to a perfect matching --- which is equivalent to being balanced and excessive. We state it explicitly here for clarity, as the connection via the hole condition may not be immediately apparent.
	\end{remark}

	\begin{lemma}\label{lem:excessive-components}
		Let $T_0$ be excessive but not connected. Then:
		\begin{enumerate}
			\item $T_0$ is unbalanced.
			\item Its components are atomic and unbalanced.
			\item Components of a vertical $T_0$ are vertical, and components of a horizontal $T_0$ are horizontal.
		\end{enumerate}
	\end{lemma}
	
	\begin{proof}
		We prove the statement for a vertical brick $T_0$. By Lemma~\ref{lem:balanced-excessive-atomic}, $T_0$ cannot be balanced, since a balanced excessive brick is connected.
		Since $G_0$ has no isolated vertices (by assumption in Section~\ref{sec:prelim}), each component $W_i$ satisfies $|\mathrm{col}(W_i)|\ge1$ and $|\mathrm{row}(W_i)|\ge1$.

		\medskip
		\textit{Every component is vertical.}
		Suppose, for contradiction, that some component $W_i$ is balanced ($p \times p$) or horizontal ($q \times p$ with $q \le p$). Since no dot lies in the vertical strip of $W_i$ outside $W_i$, the vertical complement of $W_i$ in $T_0$ is a hole of $\mathrm{vcount}$ $(b - q) + p \geq b - q + q = b = \mathrm{long}(T_0)$, contradicting the excessiveness of $T_0$. Hence every component is vertical (unbalanced with more rows than columns).

		\medskip
		\textit{Every component is excessive.}
		Suppose, for contradiction, that a component $W_i$ is not excessive. Then $W_i$ contains a hole $V_i$ of size $c \times d$ with $c + d \geq \mathrm{long}(W_i)$. Let $T_{i1}$ be the vertical complement of $V_i$ in $W_i$, a brick of size $e \times d$ where $e \leq d$ (since $W_i$ is vertical and $V_i$ accounts for at least $\mathrm{long}(W_i)$ of the $\mathrm{vcount}$). Since no dot lies in $\mathrm{vstrip}(T_{i1})$ outside $W_i$ (component property), the vertical complement of $T_{i1}$ in $T_0$ is a hole of $\mathrm{vcount}$
		\[
		(b - e) + d = b - (e - d) \geq b = \mathrm{long}(T_0),
		\]
		since $e \leq d$. This contradicts the excessiveness of $T_0$. Therefore, all components are excessive, and since they are connected, they are atomic.
	\end{proof}
	
	Figure~\ref{GM_PDF08} illustrates the case $e = d$, when $T_{i1}$ is square.
	
	\begin{figure}[htb]
		\centering
		\includegraphics[width=0.5\linewidth]{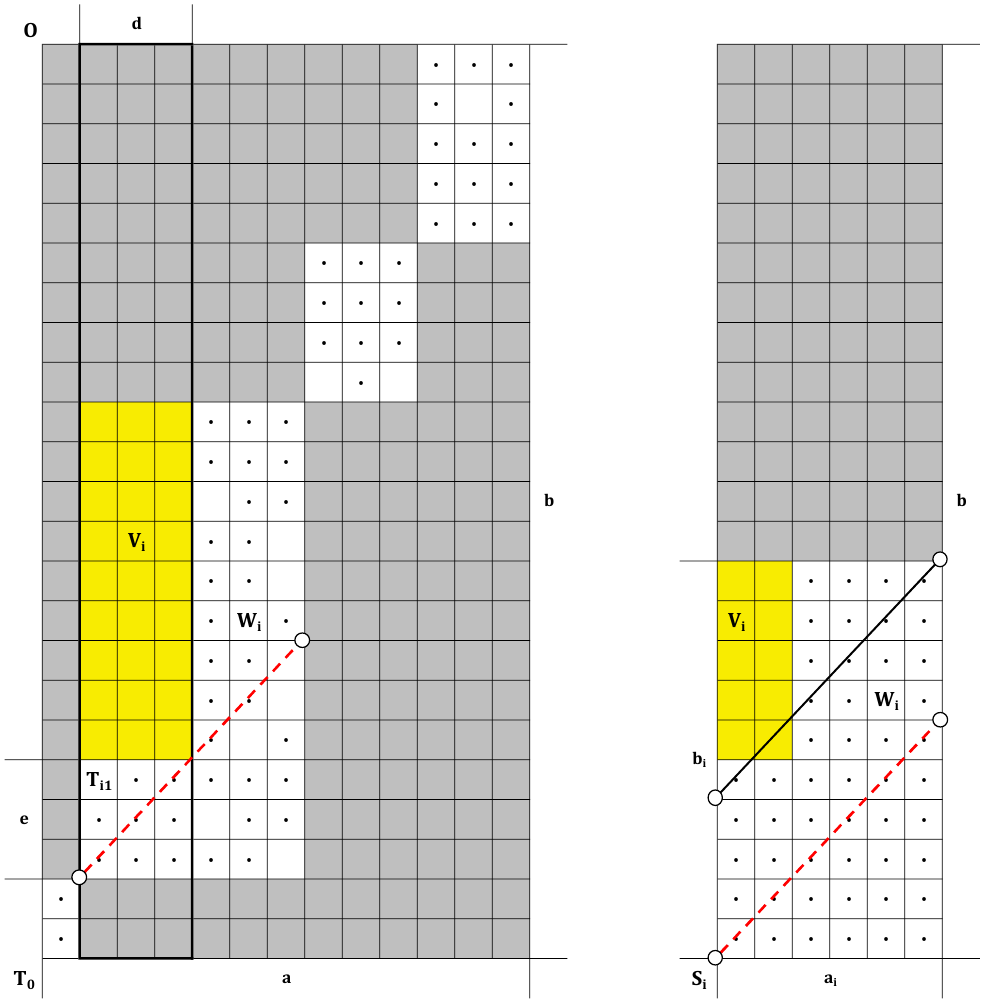}
		\caption{Atomic components}\label{GM_PDF08}
	\end{figure}
	
	Let $T_0$ be a $b\times a$ vertical brick that is $m$-excessive ($a < b$), and let $W_i$ ($i \in [p]$) be the components of $T_0$. By Lemma~\ref{lem:excessive-components}, each $W_i$ is vertical, excessive, and unbalanced. Let $m_i > 0$ denote the excess of $W_i$, $\mathrm{imb}(W_i) > 0$ its imbalance, $S_i$ the vertical strip of $W_i$, and define
	\[
	s_i = \min\{m_i, \mathrm{imb}(W_i)\}.
	\]
	
	\begin{proposition}\label{prop:excess-components}
		Each strip $S_i$ contains a hole with $\mathrm{vcount}$ $b - s_i$. Moreover, $T_0$ is $m$-excessive, where
		\begin{equation}\label{excess}
			m = \min\{s_1,\ldots ,s_p\},
		\end{equation}
		and the sum of the imbalances of the components equals that of $T_0$:
		\begin{equation}\label{imbalance}
			\mathrm{imb}(W_1)+\cdots +\mathrm{imb}(W_p) = \mathrm{imb}(T_0) = b-a.
		\end{equation}
	\end{proposition}
	
	\begin{proof}
		We consider a vertical brick $T_0$; the horizontal case is analogous.
		
		\medskip
		\noindent\textbf{Characteristic of each strip.} 
		Each component $W_i$ is $m_i$-excessive. In the vertical strip $S_i$ of $W_i$, outside $W_i$ there are no dots, hence the vertical auxiliary brick of $W_i$ within $S_i$ is a hole with $\mathrm{vcount}$
		\[
		b - |\mathrm{row}(W_i)| + |\mathrm{col}(W_i)| = b - \mathrm{imb}(W_i).
		\]
		This shows that the characteristic of $S_i$ is at most $\mathrm{imb}(W_i)$.
		If $W_i$ contains no hole, then the only hole in $S_i$ is the vertical auxiliary brick of $W_i$, which has $\mathrm{vcount}$ $b - \mathrm{imb}(W_i)$. So $\mathrm{char}(S_i) = \mathrm{imb}(W_i)$ and $s_i = \mathrm{imb}(W_i)$ is achieved.
		Otherwise, let $V_i$ be a maximum hole in $W_i$, with vertical auxiliary brick $T_1$ of size $|\mathrm{row}(W_i)| - |\mathrm{row}(V_i)|$. Then there exists a hole with $\mathrm{vcount}$
		\[
		b - (|\mathrm{row}(W_i)| - |\mathrm{row}(V_i)|) + |\mathrm{col}(V_i)| 
		= b - m_i
		\]
		in $S_i$, showing $\mathrm{char}(S_i) \le b - (b - m_i) = m_i$. Combined with the upper bound $\mathrm{char}(S_i) \le \mathrm{imb}(W_i)$ from above, the characteristic of $S_i$ satisfies
		\[
		s_i = \min\{\mathrm{imb}(W_i), m_i\},
		\]
		and this value is achieved: a hole of $\mathrm{vcount}$ $b - s_i$ exists in $S_i$ (either the vertical auxiliary brick of $W_i$ if $s_i = \mathrm{imb}(W_i)$, or the hole constructed above if $s_i = m_i$).
		
		\medskip
		\noindent\textbf{Maximum hole in $T_0$.} 
		Let $m = \min\{s_1,\ldots ,s_p\}$. Let $j$ be an index with $s_j = m$. The strip $S_j$ contains a hole of $\mathrm{vcount}$ $b - m$ (constructed above), so $\mathrm{vcount}(V) \ge b - m$ for some hole $V$ in $T_0$. It remains to show that no hole has $\mathrm{vcount}$ greater than $b - m$.

		Let $V$ be any hole in $T_0$, $X = \mathrm{col}(V)$, $Y = \mathrm{row}(V)$. Define $X_i = \mathrm{col}(W_i)\cap X$ and $Y_i = \mathrm{row}(W_i)\cap Y$ for $i \in [p]$. Since the components $W_i$ partition both $A$ and $B$ (being connected components of the induced subgraph), the sets $X_i$ are pairwise disjoint and partition $X$, and the sets $N(X_i) \subseteq \mathrm{row}(W_i)$ are pairwise disjoint (as the row sets of different components are disjoint). Moreover, since $V$ is a hole, no dot lies in $Y \times X$, hence $Y \cap N(X) = \emptyset$, which gives $Y = B \setminus N(X)$ and
		\[
		|N(X)| = \sum_{i \in [p]} |N(X_i)|, \qquad |X| = \sum_{i \in [p]} |X_i|.
		\]
		Each strip $S_i$ is $s_i$-excessive, so $|N(X_i)| \ge |X_i| + s_i$ for every $i$. Therefore,
		\[
		\mathrm{vcount}(V) = |X| + b - |N(X)| = \sum_{i} |X_i| + b - \sum_{i} |N(X_i)| \le b - \sum_{i} s_i \le b - m,
		\]
		since $\sum_i s_i \ge m$ (all $s_i \ge m$ by definition of $m$). This shows $\mathrm{vcount}(V) \le b - m$.

		Equality $\mathrm{vcount}(V) = b - m$ is achieved by taking $X_i = \mathrm{col}(W_i)$ for each $i$ with $s_i = m$ (and $X_i = \emptyset$ otherwise), for which $|N(X_i)| = |X_i| + s_i$ holds since $S_i$ is exactly $s_i$-excessive. The corresponding $Y = B \setminus N(X)$ yields a hole of $\mathrm{vcount}$ $b - m$ in $T_0$.
		
		\medskip
		\noindent\textbf{Imbalance decomposition.}
		By Lemma~\ref{lem:excessive-components}, the components $W_i$ are vertical and their row and column sets partition those of $T_0$, hence
		\[
		\mathrm{imb}(T_0) = \sum_{i=1}^p \mathrm{imb}(W_i),
		\]
		which proves equation~\eqref{imbalance}. The argument for horizontal bricks is analogous, with directions swapped.
	\end{proof}
	
	\begin{corollary}\label{near_balanced}
		Let $T_0$ be an excessive brick with $\mathrm{imb}(T_0) = 1$ (near-balanced). Then, by equation~(\ref{imbalance}), $T_0$ is connected, since if it had at least two components, its imbalance would be at least $2$.
	\end{corollary}

\subsection{Embedded Directed Graphs}\label{sec:embedding}

	Let $T_0$ be a balanced $n\times n$ non-deficient brick. Then $T_0$ has a maximum matching of size $n$. Color the cells of an arbitrary maximum matching blue, and number its elements from $1$ to $n$. The brick $T_0$, together with the blue cells, can then be interpreted as representing a directed graph on $n$ vertices.
	
	\medskip
	\noindent\textbf{Vertex correspondence.} 
	The matching elements correspond to the vertices of the graph $G$. 
	
	\medskip
	\noindent\textbf{Edge correspondence.} 
	Each non-matching dot in $T_0$ represents a directed edge from the matching element in its row (a $Y$-vertex) to the matching element in its column (an $X$-vertex). We denote such an edge by $(Y,X)$; for example, $(5,8)$ or $(8,6)$ as in Figure~\ref{GM_PDF09}.
	
	\medskip
	\noindent\textbf{Identification.} 
	Since every row and every column contains exactly one blue element, each white cell of $T_0$ can be uniquely identified with a pair $(Y,X)$, where $Y$ is the index of the matching element in its row and $X$ is the index of the matching element in its column. If the cell contains a dot, it corresponds to the edge $(Y,X)$. 
	
	\begin{remark}\label{rem:matching-choice}
		A different maximum matching clearly represents a different directed graph on the same vertex set, since the row and column assignments change accordingly.
	\end{remark}
	
	\begin{definition}\label{def:embedded-graph}
		An $n\times n$ non-deficient brick $T_0$ is called an \emph{embedding brick}, and a directed graph $H$ defined by an arbitrary maximum matching of size $n$ is called an \emph{embedded graph}.
	\end{definition}
	
	The original graph $G$ can also be transformed into a directed graph by orienting the edges of a perfect matching from their endpoints in $A$ to their endpoints in $B$, and the remaining edges in the opposite direction. This yields a directed graph $H^*$. The embedded graph $H$ defined above is a minor of $H^*$, obtained by contracting the edges of the matching.
	
	From \cite{Tarjan1972} we know that the strongly connected components (SCCs) partition the vertex set of $G$ into equivalence classes. This partition is unique, up to the ordering of the SCCs.
	
	If the SCCs are topologically ordered, then every edge points to a later component (forward edge).
	Consequently, the bricks in $T_0$ corresponding to the SCCs are in \emph{block triangular form (BTF)} if and only if the SCCs are topologically ordered.
	Figure~\ref{GM_PDF09} illustrates a balanced embedding brick $T_0$ with a perfect matching (blue cells), the embedded directed graph $G$, its decomposition into SCCs, and the BTF of $T_0$ obtained from the topological ordering of the SCCs. In the vertical strip of each block, there are no dots above the block, reflecting the forward-edge structure.
	
	\begin{figure}[htb]
		\centering
		\includegraphics[width=\linewidth]{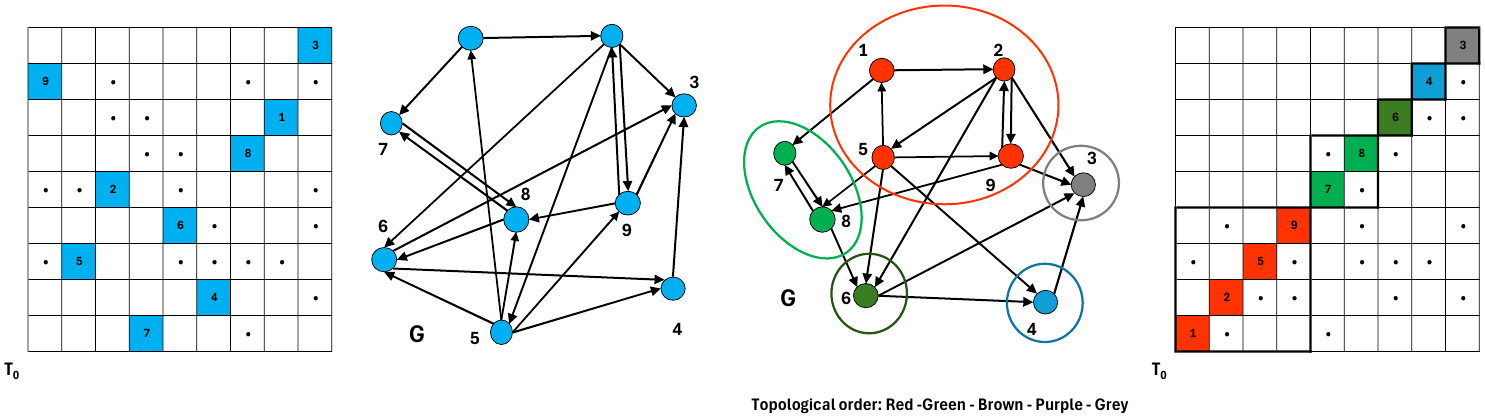}
		\caption{BTF obtained from the SCCs of an embedded directed graph}\label{GM_PDF09}
	\end{figure}
	
	\begin{proposition}\label{prop103}
		A directed graph on $n$ vertices is strongly connected if and only if its $n\times n$ embedding brick is atomic, where the dots in the (blue) cells corresponding to the vertices are also treated as ordinary dots.
	\end{proposition}
	
	\begin{proof}
		(\emph{Forward direction}) Let $T_0$ be an atomic $n\times n$ brick and $M$ an arbitrary maximum matching in $T_0$. Suppose, for contradiction, that the corresponding embedded graph $H$ is not strongly connected. Then there exist vertices $A$ and $B$ such that there is no directed path from $A$ to $B$. Let $\mathrm{Access}(A)$ denote the set of vertices reachable from $A$, and permute rows and columns so that $\mathrm{Access}(A)$ occupies the leftmost columns and bottom rows (so $A$ lies in the lower-left corner). 
		
		The square $T_1$ induced by $\mathrm{Access}(A)$ cannot cover all of $T_0$, since $B \notin \mathrm{Access}(A)$. Then the horizontal complement of $T_1$ exists and contains no dots, hence it forms a hole with $\mathrm{vcount}$ $n$, which contradicts the atomicity of $T_0$. Therefore, $H$ is strongly connected.
		
		\medskip
		
		(\emph{Reverse direction}) Suppose $H$ is strongly connected. Then its embedding brick $T_0$ cannot be deficient because a maximum matching of size $n$ exists, and thus it contains no hole with $\mathrm{vcount}$ greater than $n$. It remains to show that $T_0$ contains no hole with $\mathrm{vcount}$ $n$ either. Suppose, for contradiction, that such a hole $V$ exists. Move $V$ to the origin. Let $T_1$ and $T_2$ be its vertical and horizontal auxiliary bricks. By Corollary~\ref{cor:forbidden-dot}, all dots in the distant pair of $V$ are forbidden if $V$ is liminal, so the blue cells (representing the vertices of $H$) must lie in $T_1$ and $T_2$. The remaining dots correspond to edges from $T_1$ to $T_2$, while there are no edges inside $V$. Consequently, there is no edge whose row corresponds to a blue element in $T_2$ and whose column corresponds to a blue element in $T_1$. Therefore, no vertex in $T_2$ can reach any vertex in $T_1$, contradicting strong connectivity as depicted in Figure~\ref{GM_PDF10}. 
		
		Hence, no hole with $\mathrm{vcount}$ $n$ exists, so $T_0$ is excessive. Since $T_0$ is also balanced, by Lemma~\ref{lem:balanced-excessive-atomic} it is atomic.
	\end{proof}
	\begin{figure}[htb]
		\centering
		\includegraphics[width=\linewidth,keepaspectratio]{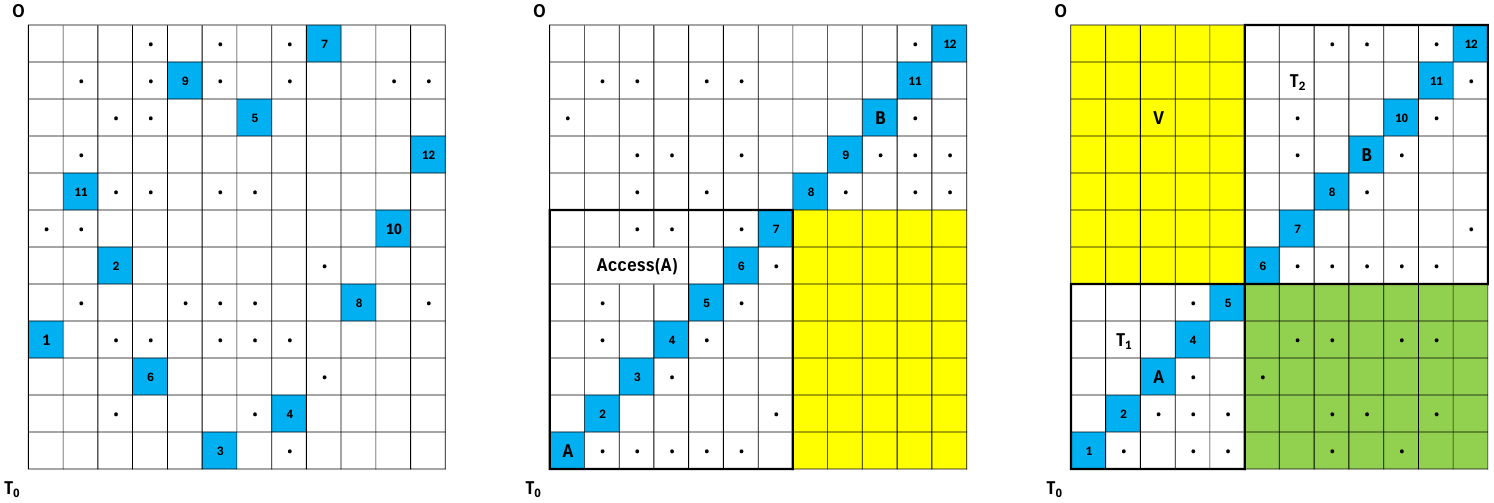}
		\caption{Excessive -- SCC}\label{GM_PDF10}
	\end{figure}
	
	\begin{corollary}\label{excessive_SCC}
		In an $n\times n$ excessive embedding brick $T_0$, the maximum matchings of size $n$ are in one-to-one correspondence with the strongly connected simple directed graphs embeddable in $T_0$.
	\end{corollary}
\FloatBarrier

\section{Canonical Excessive Decomposition}\label{sec:excessive-decomp}
	In this section we construct the excessive decomposition using an MPIS-driven procedure. The algorithm uses an oracle $\mathcal{A}$ that computes a maximum proper independent set for a given brick; such an oracle exists (e.g., by exhaustive search), and its existence suffices for the structural analysis. We do not analyze computational complexity here. We note, however, that for bipartite graphs the oracle $\mathcal{A}$ is efficiently implementable: by the K\H{o}nig--Egerv\'ary theorem, a maximum proper independent set corresponds to a minimum vertex cover, which in turn is obtained from a maximum matching via the standard alternating-path construction. Using the Hopcroft--Karp algorithm \cite{HopcroftKarp1973}, this runs in $O(|E|\sqrt{|V|})$ time, so the entire BTF decomposition is computable in polynomial time. Since the procedure invokes the oracle at most once per extracted block, and there are at most $k+2$ blocks, the full decomposition runs in $O\bigl((k+2)\,|E|\sqrt{|V|}\bigr)$ time, which is polynomial in the size of $G_0$.
	
	Based on the algorithm $\mathcal{A}$, one can define a finite procedure that produces an isotope of the brick $T_0$ as a decomposition into excessive bricks arranged in block-triangular form (BTF), i.e., in the vertical strip of each brick there are no points above it, as illustrated in Figure~\ref{GM_PDF11}.
	
	\subsection{MPIS-driven procedure}\label{MH}
	Let $T_0$ be a $b\times a$ brick, where $a \leq b$. Let $\mathcal{A}$ be an algorithm that, on a given brick, outputs a pair $(H_1,H_2)$ consisting of the row and column indices of a maximum proper independent set (MPIS), if a hole exists in the brick; otherwise it returns $(\emptyset,\emptyset)$. Note that $H_1$ and $H_2$ are always simultaneously empty or simultaneously non-empty.
	
	The following MPIS-driven decomposition produces an isotope of $T_0$ in block-triangular form (BTF) in which the excessive bricks decompose $T_0$. The procedure maintains a global ordered list $\mathcal{L}$ of blocks, initially empty: each time the recursion reaches an excessive brick (the base case $\mathrm{vcount}(V) < \mathrm{long}(T)$, at which $T$ is not split further), that brick is appended to $\mathcal{L}$. The output is $\mathcal{L}$, the sequence of excessive blocks listed in block-triangular order.
	
	\begin{algorithm}[H]
		\caption{MPIS-driven BTF\_DECOMP}\label{mpis:btf-decomp}
		\begin{algorithmic}[1]
			
			\Procedure{BTF\_decomp}{$T$}
			\State $(H_1,H_2) \gets \mathcal{A}(T)$
			
			\If{$\mathrm{vcount}(H_1 \times H_2) < \mathrm{long}(T)$}
			\State append $T$ to $\mathcal{L}$ \Comment{$T$ is excessive: a leaf block}
			\State \Return
			\EndIf        
			
			\State $V \gets H_1 \times H_2$        
			\State Align $V$ to upper-left grid point of $T$ by permutation
			
			\State \Call{BTF\_decomp}{$\mathrm{vertical\_auxiliary}(V)$}
			\State \Call{BTF\_decomp}{$\mathrm{horizontal\_auxiliary}(V)$}
			\EndProcedure
			
			\State $\mathcal{L} \gets (\,)$ \Comment{global ordered list of blocks, initially empty}
			\State \Call{BTF\_decomp}{$T_0$}
			\State \Return $\mathcal{L}$ \Comment{excessive blocks in BTF order}
			
		\end{algorithmic}
	\end{algorithm}
	
	\begin{proposition}[Termination of the BTF recursion]
		The procedure MPIS-driven \textsc{BTF\_decomp} terminates after finitely many steps for any initial brick $T_0$.
	\end{proposition}
	
	\begin{proof}
		Since $T_0$ is finite, the number of possible bricks is finite. 
		Consider the set of bricks processed during the recursion, and let $T$ be a $b \times a$ brick on which the procedure is called.
		
		The algorithm stops when $\mathrm{vcount}(H_1 \times H_2) < \mathrm{long}(T)$, which includes the case $H_1 = H_2 = \emptyset$ (since then $\mathrm{vcount} = 0 < \mathrm{long}(T)$).
		
		Suppose the stopping condition does not hold. Then $V = H_1 \times H_2$ is nonempty with size $r \times s$ where $r, s > 0$, and $\mathrm{vcount}(V) \ge \mathrm{long}(T)$. The procedure recursively calls itself on the two auxiliary bricks of $V$ in $T$: the vertical auxiliary brick of size $(b-r) \times s$ with $\mathrm{vcount} = (b-r)+s < a+b$, and the horizontal auxiliary brick of size $r \times (a-s)$ with $\mathrm{vcount} = r+(a-s) < a+b$. Both have strictly smaller $\mathrm{vcount}$ than $T$ (since $r > 0$ and $s > 0$). Hence, by strict decrease of the $\mathrm{vcount}$ at each recursive call, the recursion terminates after finitely many steps.
	\end{proof}
	
	\begin{theorem}[Structural Properties of the Decomposition] \label{thm:structural_properties}
		Let $T_0$ be a non-excessive bipartite brick. The MPIS-driven \textsc{BTF\_decomp} decomposition of $T_0$ yields a block-triangular form consisting of a sequence of excessive bricks with the following properties:
		\begin{enumerate}
			\item There is at most one horizontal excessive brick $D$ in the first position. If it exists, its size is $(a_D - \delta) \times a_D$, where $a_D = |\mathrm{col}(D)|$ and $\delta > 0$.
			\item There are $k \geq 0$ square excessive bricks $C_1, \dots, C_k$, each of them is a square.
			\item There is at most one vertical excessive brick $E$ in the last position. If it exists, its size is $((b-a)+ \delta + a_E) \times a_E$, where $a_E = |\mathrm{col}(E)|$ and $(b-a)+ \delta > 0$.
		\end{enumerate}
		The bricks follow the topological ordering, where any of the components may be empty.
	\end{theorem}
	
	\begin{proof}
		Let $T_0$ be a non-excessive brick. Construct a $V$ MH in $T_0$ using the algorithm $\mathcal{A}$. 
		Then for the $r\times s$ hole $V$ we have $\mathrm{vcount}(V) \geq b$, hence $T_0$ is either deficient or liminal, so 
		\[
		\delta = b - \mathrm{vcount}(V) \geq 0.
		\]
		
		The vertical auxiliary brick $T_1$ of $V$ is an $s \times (s-\delta)$ horizontal brick, hence 
		\[
		\mathrm{imb}(T_1) = \delta.
		\] 
		The horizontal auxiliary brick $T_2$ is an $r \times (a-s)$ vertical brick, hence
		\[
		\mathrm{imb}(T_2) = r - (a-s) = (r+s) - a = b + \delta - a = (b-a) + \delta.
		\]
		
		Moreover, 
		\[
		\mathrm{long}(T_1) = |\mathrm{col}(V)| \quad \text{and} \quad \mathrm{long}(T_2) = |\mathrm{row}(V)|,
		\] 
		so the sides of $V$ can be regarded as the long sides of $T_1$ and $T_2$. 
		Therefore, the portion of the characteristic segment of $T_0$ lying in $T_1$ and $T_2$ is exactly the Hall line of the respective brick.
		
		By Statement~\ref{auxiliary_brick}, the bricks $T_1$ and $T_2$ cannot be deficient. 
		Thus, after the first step of the algorithm, all subsequently obtained auxiliary bricks are non-deficient.
		
		If a brick is excessive, the recursion stops along that branch. 
		If $T_1$ or $T_2$ is liminal, then it contains a hole $V_1$ or $V_2$ with 
		\[
		\mathrm{vcount}(V_1) = \mathrm{long}(T_1) \quad \text{or} \quad \mathrm{vcount}(V_2) = \mathrm{long}(T_2),
		\] 
		respectively. After making the hole origin-aligned by permutation, the corner grid point opposite to the origin lies on the characteristic line of $T_0$.
		
		Let $T_{11}$ and $T_{12}$ denote the auxiliary bricks of $V_1$ in $T_1$, and $T_{21}$ and $T_{22}$ those of $V_2$ in $T_2$. 
		These are also non-deficient, so the procedure can continue recursively.

		Since $V$ has size $r \times s$, we have $\mathrm{long}(T_1) = s$ and $\mathrm{long}(T_2) = r$. If $T_1$ is liminal, then $V_1$ has $\mathrm{vcount}(V_1) = s$, say $V_1$ has size $r_1 \times s_1$ with $r_1 + s_1 = s$. The vertical auxiliary brick $T_{12}$ of $V_1$ in $T_1$ has size $(s - r_1) \times s_1 = s_1 \times s_1$, which is a square. Similarly, if $T_2$ is liminal, then $V_2$ has $\mathrm{vcount}(V_2) = r$, say $V_2$ has size $r_2 \times s_2$ with $r_2 + s_2 = r$, and the horizontal auxiliary brick $T_{21}$ of $V_2$ in $T_2$ has size $r_2 \times (r - s_2) = r_2 \times r_2$, which is also a square.

		The imbalance of $T_{11}$ and $T_{22}$ equals $\delta$ and $(b-a)+\delta$ respectively at each recursive step, since a square block is extracted and the remaining imbalance is preserved.
		
		By construction, the first element of the brick sequence is always a horizontal brick (if it exists), and the last is a vertical brick (if it exists).
		
		If a square brick is liminal, then it contains a hole of size equal to its side length (e.g., brick $T_{12}$ in Figure~\ref{GM_PDF11}), and both resulting bricks are squares.
		At the end of the procedure, we obtain a decomposition in which all bricks are excessive. 
		From the construction it follows that there is at most one non-square horizontal excessive brick $D$ in the first position (if $\delta > 0$), with 
		\[
		\mathrm{imb}(D) = \delta,
		\] 
		and at most one non-square vertical excessive brick $E$ in the last position (if $b - a + \delta > 0$), with 
		\[
		\mathrm{imb}(E) = (b-a) + \delta.
		\]
		
	\end{proof}
	
	\begin{figure}[htb]
		\centering
		\includegraphics[width=\linewidth,keepaspectratio]{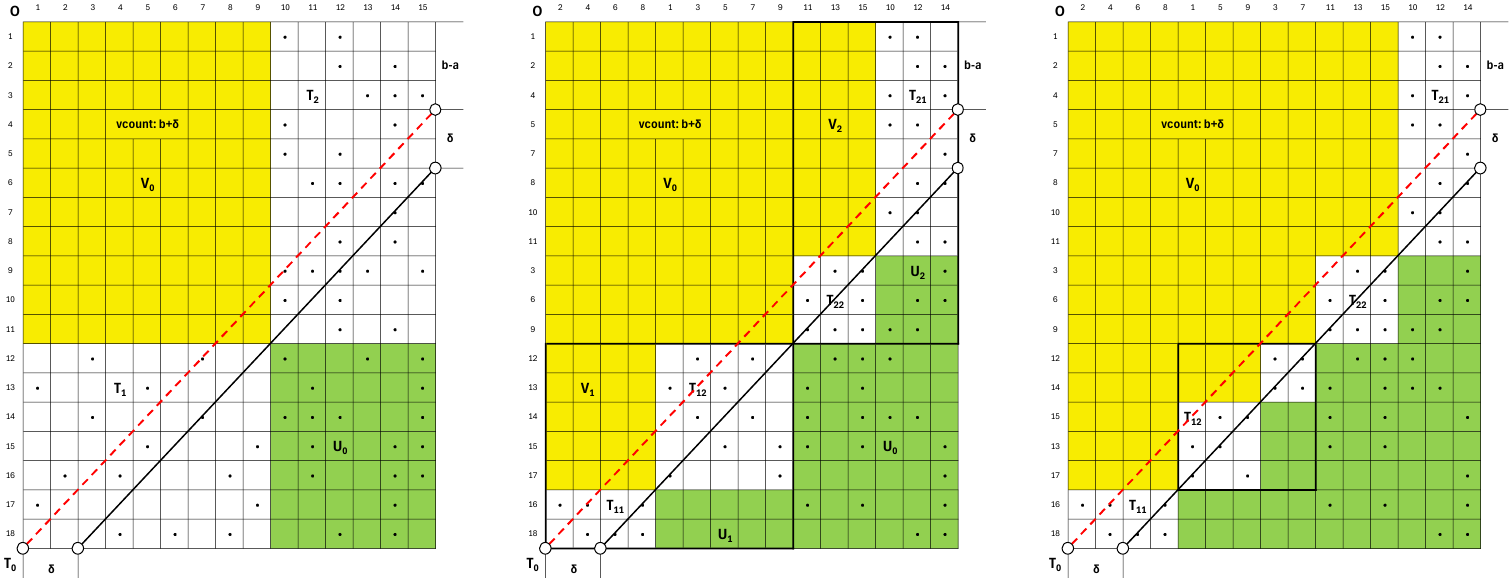}
		\caption{MPIS-driven decomposition}\label{GM_PDF11}
	\end{figure}
		
	Let $A_0 = \mathrm{col}(D)$ and $B_0 = \mathrm{row}(D)$ and for all $i\in [k]$:
	\[
	A_i = A_0 \cup \mathrm{col}(C_1) \cup \cdots \cup \mathrm{col}(C_i), \qquad
	B_i = B_0 \cup \mathrm{row}(C_1) \cup \cdots \cup \mathrm{row}(C_i).
	\]
	
	Then for all $i\in \{0\} \cup [k]$:
	\[
	\mathrm{vcount}((B\setminus B_i) \times A_i) = b+\delta.
	\] 
	so, each brick is origin-aligned, visible MH in $T_0$.
	
	The resulting sequence is a block triangular form (BTF): in the grid model this means that in the vertical strip of each block $B_i$ there are no dots above $B_i$ (equivalently, no edges to an earlier block $B_j$ with $j < i$). We verify this by an invariant of the recursion. Maintain the list of bricks produced so far, starting from the single brick $T_0$. At each step a non-excessive brick $T$ in the list is replaced, in place, by its vertical auxiliary brick followed by its horizontal auxiliary brick; an excessive brick is a leaf and is left untouched. It suffices to show that each replacement preserves the ordering.

	First, the two replacement bricks are correctly ordered \emph{relative to each other}: between them, the only backward region (from the horizontal auxiliary to the vertical auxiliary) is the maximum hole $V$, which is dot-free, so there is no backward edge; the forward edges all lie in the remote mate of $V$ and are forbidden with respect to every maximum matching (Corollary~\ref{cor:forbidden-dot}). Second, the replacement preserves the ordering \emph{relative to every other block} $B_j$ in the list: both replacement bricks are sub-bricks of $T$ (their rows and columns are subsets of those of $T$), so any edge between a replacement brick and $B_j$ is in particular an edge between $T$ and $B_j$, and therefore runs in the same (forward) direction that $T$ and $B_j$ already had. Hence each step preserves the BTF ordering, and by induction the final list of excessive bricks $D, C_1,\ldots,C_k, E$ is topologically ordered.

	\begin{definition}[Core]
		The sub-brick induced by the sequence of square bricks $C_1,\ldots,C_k$ is called the \emph{core}, denoted by $\mathrm{core}(T_0)$.
	\end{definition}

	The unique BTF arrangement of the excessive bricks $D, C_1,\ldots,C_k, E$ produced by the MPIS-driven procedure is the \emph{excessive decomposition} of $T_0$ (Definition~\ref{def:exc-decomp-prelim}); the construction just given establishes its existence and uniqueness.
	
	\begin{corollary}
		The result of an MPIS-driven procedure is a BTF if and only if the bricks $C_1,\ldots,C_k$ (if any) are topologically ordered.
		The excessive decomposition is a canonical structural characterization, and efficient computation is possible via existing matching algorithms.
	\end{corollary}

	\begin{corollary}\label{cor:unique-DE}
		If $\delta > 0$, then $D$ is the unique brick for which $\mathrm{vstrip}(D)$ contains a hole of $\mathrm{vcount}$ $b + \delta$. If $(b-a)+\delta > 0$, then $E$ is the unique brick for which $\mathrm{hstrip}(E)$ contains a hole of $\mathrm{vcount}$ $b + \delta$.
	\end{corollary}

	\begin{proof}
		By Theorem~\ref{thm:structural_properties}, the excessive decomposition produces at most one horizontal brick $D$ (with $\mathrm{imb}(D) = \delta$) and at most one vertical brick $E$ (with $\mathrm{imb}(E) = (b-a)+\delta$). The vertical strip of $D$ contains a hole of $\mathrm{vcount}$ $b + \delta$ by construction (the remote mate of the MH used to extract $D$). No other vertical strip can contain such a hole: the $C$-blocks and $E$ are excessive, hence contain no hole of $\mathrm{vcount} \geq \mathrm{long}$, and $D$ is the unique horizontal block. The argument for $E$ is symmetric.
	\end{proof}

\section{Classification of bipartite graphs}\label{sec:classes}
	To each bipartite brick $T_0$ with $a \le b$ the canonical decomposition assigns three structural invariants: its \emph{imbalance class} (balanced, $b=a$, or unbalanced, $b>a$); its \emph{characteristic type} (excessive $m>0$, liminal $m=0$, or deficient $m<0$); and its number of $C$-blocks $k \in \{0,1,{>}1\}$. The block composition is determined by these invariants through the presence rules
	\[
		\begin{aligned}
			D \text{ present} &\iff \text{deficient},\\
			C \text{ present} &\iff k \ge 1,\\
			E \text{ present} &\iff \text{unbalanced or deficient},
		\end{aligned}
	\]
	which follow from $\mathrm{imb}(D) = \delta$ and $\mathrm{imb}(E) = (b-a)+\delta$ (Theorem~\ref{thm:structural_properties}). Of the $2 \times 3 \times 3 = 18$ a priori combinations, the two structural laws below exclude exactly seven, leaving the eleven classes of Table~\ref{tab:block_composition} (in which ``--'' marks the excluded combinations).

	\begin{lemma}[Excessive irreducibility]\label{lem:excessive-irreducible}
		If $T_0$ is excessive ($m>0$), its excessive decomposition is trivial: $T_0$ is its own unique block and has no $D$-block. If $T_0$ is balanced it is a single $C$-block ($k=1$); if unbalanced, a single $E$-block ($k=0$).
	\end{lemma}
	\begin{proof}
		Every hole of an excessive brick has $\mathrm{vcount} < \mathrm{long}(T_0)$, so the stopping condition of \textsc{BTF\_decomp} (Algorithm~\ref{mpis:btf-decomp}) holds at the first call and $T_0$ is not split. Being connected and excessive, $T_0$ is atomic (Proposition~\ref{prop:atomic-iff-connected}): a square $C$-block if balanced, a vertical $E$-block if unbalanced. Since $\delta < 0$, there is no Hall deficiency, hence no $D$-block.
	\end{proof}

	\begin{lemma}[Liminal core]\label{lem:liminal-core}
		If $T_0$ is liminal ($\delta=0$), its decomposition contains at least one $C$-block ($k \ge 1$); if moreover $T_0$ is balanced, then $k \ge 2$.
	\end{lemma}
	\begin{proof}
		A liminal brick has a hole of $\mathrm{vcount} = b = \mathrm{long}(T_0)$, so \textsc{BTF\_decomp} performs at least one step and extracts a square $C$-block from the vertical auxiliary brick of that hole; hence $k \ge 1$. A single $C$-block is an excessive square (Lemma~\ref{lem:balanced-excessive-atomic}); if $T_0$ were balanced with $k=1$ it would equal that square and be excessive, contradicting $m=0$. Hence balanced liminal forces $k \ge 2$.
	\end{proof}

	\begin{theorem}[Eleven-class classification]\label{thm:main_classification}
		Every bipartite brick $T_0$ with $a \le b$ has a well-defined invariant triple (imbalance class, characteristic type, $C$-count $k \in \{0,1,{>}1\}$), and distinct triples give disjoint classes. Subject to Lemmas~\ref{lem:excessive-irreducible} and~\ref{lem:liminal-core}, the triple realizes exactly the eleven values of Table~\ref{tab:block_composition}; each is attained, and by the Atomic Decomposition Theorem (Theorem~\ref{thm:main_decomp}) it is a complete, canonical invariant of the block composition of $T_0$. The seven excluded triples are precisely those violating one of the two lemmas.
	\end{theorem}

	\begin{proof}[Proof of Theorem~\ref{thm:main_classification}]
		The triple is a function of $T_0$, so every brick lies in exactly one class; we account for the $18$ combinations row by row.

		\emph{Excessive row.} Lemma~\ref{lem:excessive-irreducible} forces $k=1$ when balanced and $k=0$ when unbalanced, excluding the four triples (balanced, excessive, $k \in \{0,{>}1\}$) and (unbalanced, excessive, $k \in \{1,{>}1\}$). The two survivors are a single $C$ and a single $E$.

		\emph{Liminal row.} Lemma~\ref{lem:liminal-core} excludes $k=0$ for both balances and, in the balanced case, also $k=1$ --- the three triples (balanced, liminal, $k \in \{0,1\}$) and (unbalanced, liminal, $k=0$). The three survivors are balanced with $k>1$ (a $C$-core) and unbalanced with $k \in \{1,{>}1\}$ (a $C$-core followed by $E$).

		\emph{Deficient row.} Here $\delta > 0$, so a $D$-block is present and neither lemma applies; all $2 \times 3 = 6$ combinations occur, each comprising $D$, a (possibly empty) $C$-core, and $E$.

		This removes $4 + 3 = 7$ combinations and leaves $2 + 3 + 6 = 11$, with block compositions given by the presence rules and tabulated in Table~\ref{tab:block_composition}. Realizability is exhibited by Figures~\ref{GM_PDF21}--\ref{GM_PDF25}, one brick per class; together with the exclusions, the eleven classes are exactly the nonempty ones.
	\end{proof}

	\begin{remark}[Why eleven, and minimality]
		The two lemmas are the only obstructions: the count $11 = 18 - 7$ is generated entirely by excessive irreducibility (four exclusions) and the liminal-core condition (three exclusions). The classification is \emph{minimal} --- no two classes share an invariant triple, and each of balance, characteristic type, and $k \in \{0,1,{>}1\}$ affects the block composition --- and \emph{exhaustive and disjoint}, hence a genuine partition of the universe of bipartite bricks. Two classes are classical: the balanced excessive class ($k=1$) is exactly the elementary (matching-covered) bipartite graphs of Lov\'asz--Plummer, and the deficient class with $k=0$ consists of graphs whose canonical decomposition has no $C$-block. The remaining classes have no previously named counterpart; the grid model describes them uniformly.
	\end{remark}
	
	The following figures illustrate one example from each of the 11 structural classes.

	\begin{figure}[H]
		\centering
		\includegraphics[width=.6\linewidth,keepaspectratio]{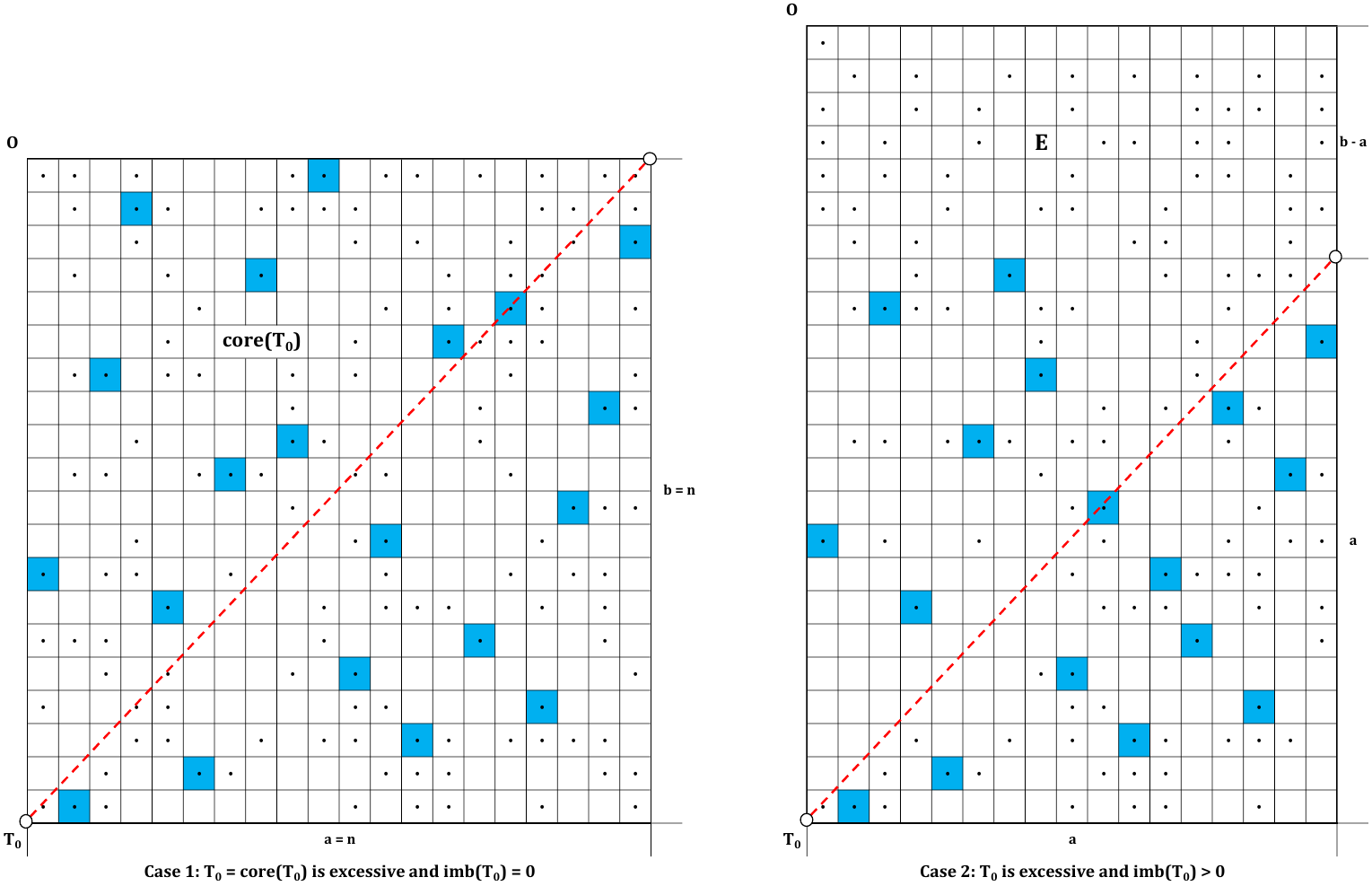}
		\caption{DM-irreducible cases}\label{GM_PDF21}
	\end{figure}

	\begin{figure}[H]
		\centering
		\includegraphics[width=.8\linewidth,keepaspectratio]{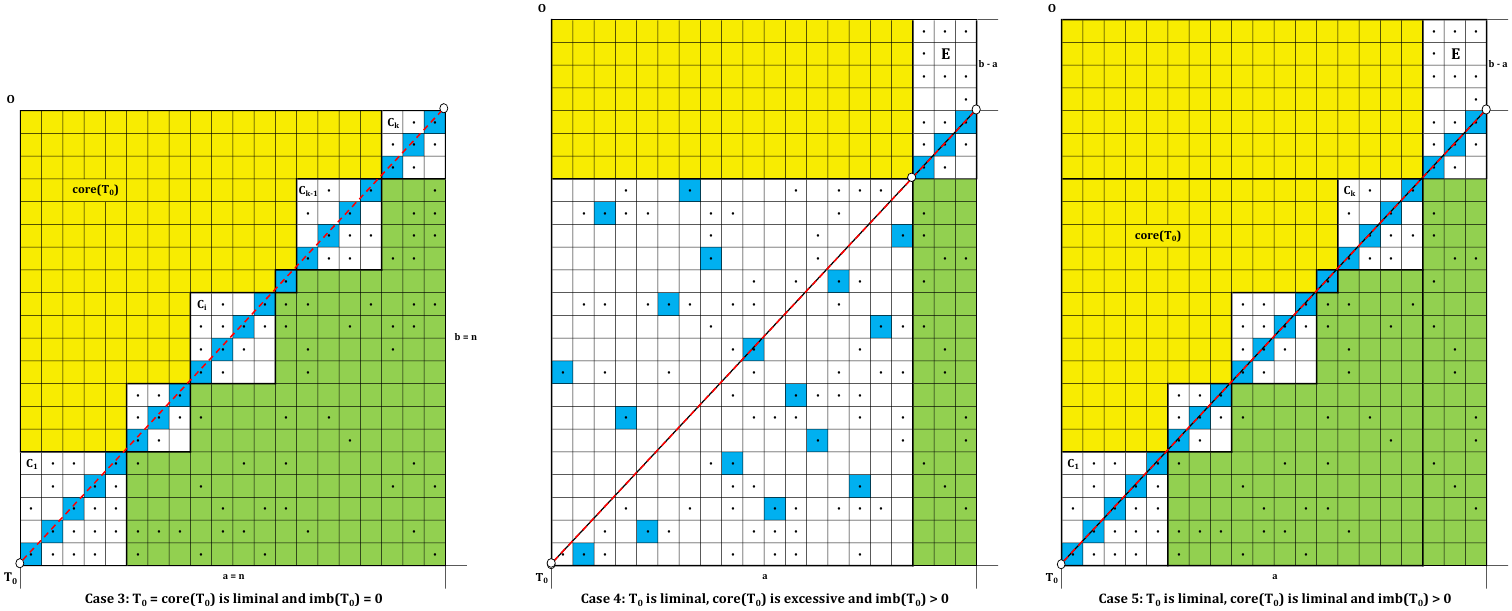}
		\caption{DM-reducible cases, no block $D$}\label{GM_PDF22}
	\end{figure}

	\begin{figure}[H]
		\centering
		\includegraphics[width=.6\linewidth,keepaspectratio]{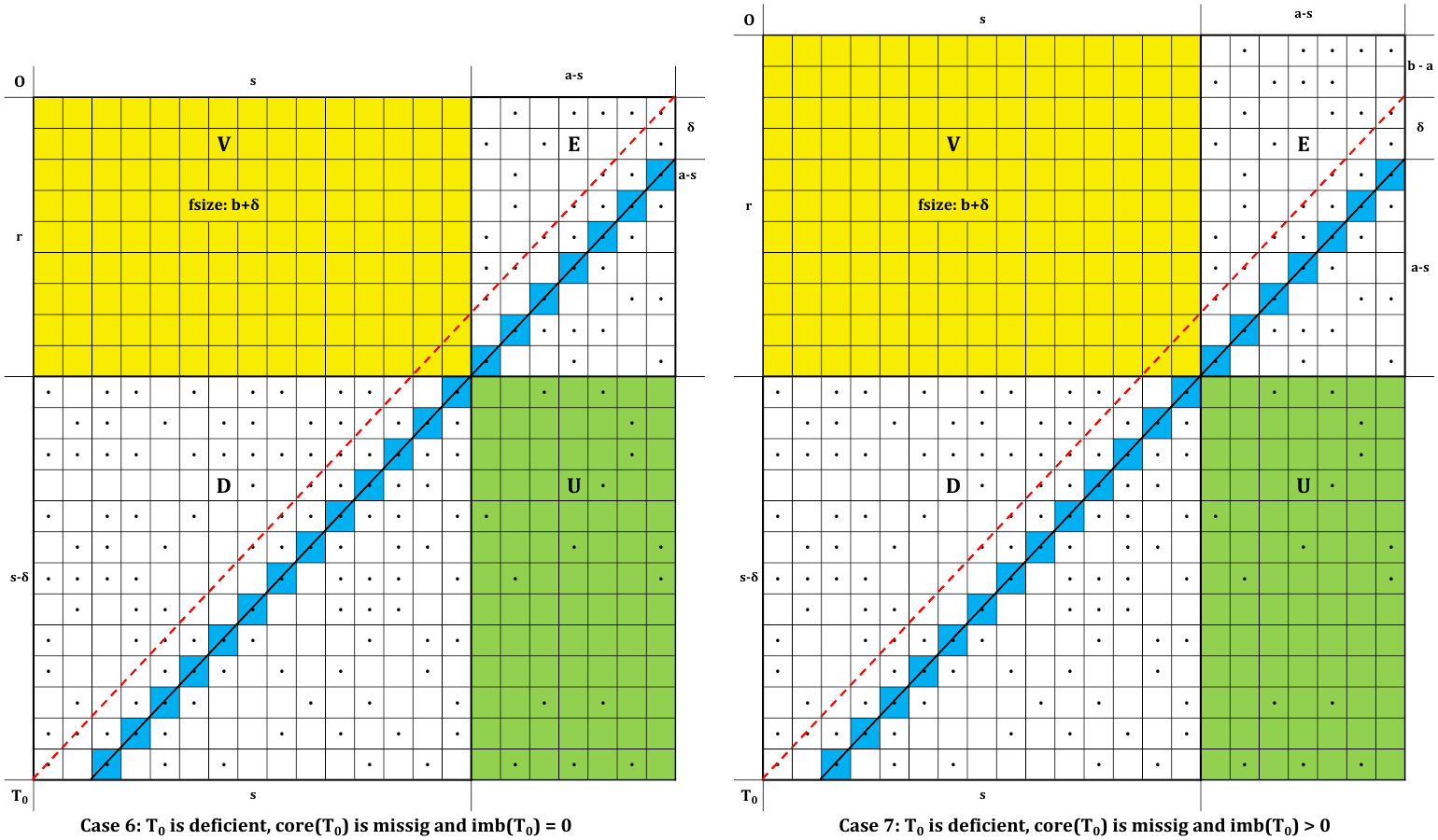}
		\caption{DM-reducible cases, no block $C$}\label{GM_PDF23}
	\end{figure}

	\begin{figure}[H]
		\centering
		\includegraphics[width=.6\linewidth,keepaspectratio]{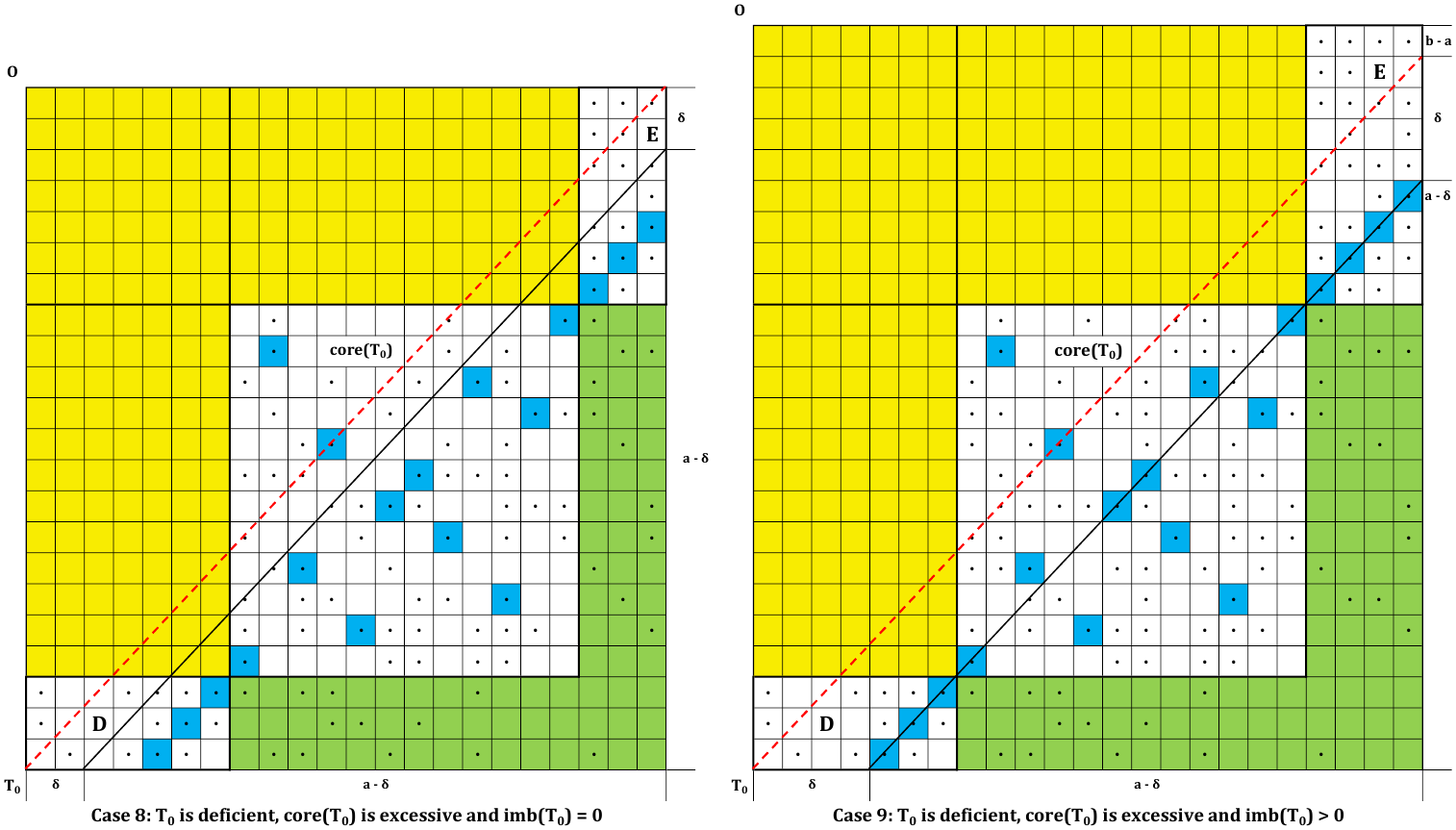}
		\caption{DM-reducible cases}\label{GM_PDF24}
	\end{figure}

	\begin{figure}[H]
		\centering
		\includegraphics[width=.6\linewidth,keepaspectratio]{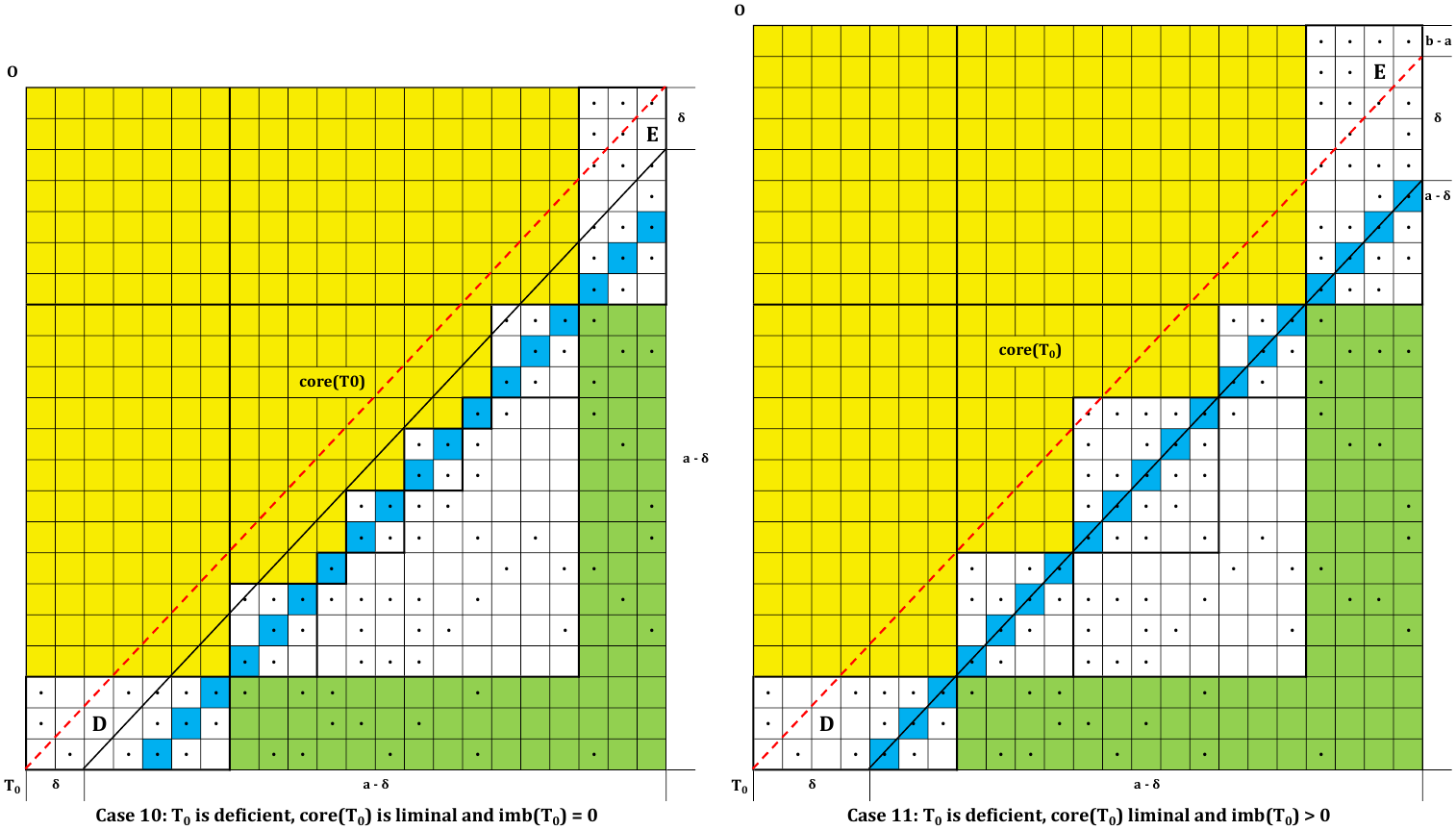}
		\caption{DM-reducible cases (continued)}\label{GM_PDF25}
	\end{figure}
\FloatBarrier

The eleven classes above exhaust all structural possibilities for a bipartite graph under the excessive decomposition.

\section{Structural Consequences and Symmetry}\label{sec:consequences}
\begin{theorem}[Structural symmetry of $D$ and $E$ blocks]
	In cases where both $D$ and $E$ blocks are present, the roles of $D$ and $E$ are structurally equivalent. Each can be obtained from the other by interchanging the two vertex classes.

	More precisely, any structural statement formulated in terms of $D$ blocks remains valid when $D$ and $E$ are interchanged, provided the corresponding parameters are interpreted symmetrically.
\end{theorem}

\begin{proof}
	Both $D$ and $E$ are excessive unbalanced bricks produced by the MPIS-driven decomposition (Theorem~\ref{thm:structural_properties}): $D$ is horizontal with $\mathrm{imb}(D) = \delta$, and $E$ is vertical with $\mathrm{imb}(E) = (b-a)+\delta$. By Lemma~\ref{lem:excessive-components}, the atomic sub-blocks of $D$ are the connected components of $G_0[D]$, each an unbalanced excessive brick; the same holds for $E$. The sub-blocks of $D$ are horizontal and those of $E$ are vertical, inheriting the orientation of their parent block. Their excess values relate as in Proposition~\ref{prop:excess-components}, and no edges exist between distinct sub-blocks within the same region.

	The symmetry does not rely on a graph automorphism, but rather on the structural parallelism of the MPIS construction applied to the two sides. Therefore, any structural statement about $D$ and its sub-blocks $D_1,\ldots,D_p$ that depends only on the excessive-unbalanced structure remains valid for $E$ and its sub-blocks $E_1,\ldots,E_q$, with the parameter correspondence $\delta \leftrightarrow (b-a)+\delta$.
\end{proof}

\begin{corollary}[Matching structure induced by the atomic decomposition]
	\label{lem:max-matching-unique}
	In each atomic block $D_i$ there exists a maximum matching of size $\mathrm{small}(D_i)$, in each atomic block $E_j$ one of size $\mathrm{small}(E_j)$, and in each square block $C_l$ there exists a perfect matching. The edges outside the atomic blocks (i.e., in the green region) are not covered by any maximum matching.
	
	Hence, a matching in $T_0$ is maximum if and only if it is the union of maximum matchings of the individual atomic blocks (Theorem~\ref{thm:atomic-decomp}, property~(4)).
\end{corollary}

Moreover, by Corollary~\ref{cor:forbidden-dot}, every dot lying in the remote mate of any MH of $T_0$ is forbidden. The dots outside the atomic blocks (the green region, see Figure~\ref{fig:gridmodel}) lie precisely in such remote mates --- each is the remote mate of the MH produced by the MPIS-driven procedure at the step when the enclosing block was extracted --- and therefore belong to no maximum matching.
Thus the dots of $T_0$ (i.e., the edges of $G$) are partitioned into two classes:
\begin{itemize}
	\item those belonging to excessive bricks, each contained in at least one maximum matching,
	\item and the forbidden dots, which belong to no maximum matching.
\end{itemize}

Let us construct the DM-decomposition of a graph $G$ and the MPIS-driven excessive decomposition of the corresponding brick $T_0$. The proof below shows that the BFS-type reachability construction underlying the classical DM-decomposition (cf.~\cite{PothenFan1990}) is precisely the grid-model realization of the MPIS-driven procedure, thereby establishing their equivalence.

Then the following holds:

\begin{theorem}[DM--excessive equivalence]\label{thm:dm-excessive}
	The MPIS-driven excessive decomposition of a bipartite graph coincides with the decomposition obtained from its Dulmage--Mendelsohn decomposition by refining the $C$-part. More precisely:
	\begin{itemize}
		\item if $\delta > 0$, then the brick $D$ in the first position of the excessive decomposition corresponds exactly to the $D$-part of the DM-decomposition,
		\item if $(b-a)+\delta > 0$, then the brick $E$ in the last position corresponds exactly to the $E$-part of the DM-decomposition,
		\item all further excessive bricks are squares (if any); together they induce the middle $C$-part, and these excessive squares coincide with the refined blocks of the $C$-part in the DM-decomposition.
	\end{itemize}
\end{theorem}

\begin{proof}
	We prove both directions of the equivalence.

	\medskip
	\noindent\textbf{(MPIS $\Rightarrow$ DM).}
	Since each block in the excessive decomposition is excessive, it satisfies the hole condition and hence admits a matching of size $\mathrm{small}(T)$ (Corollary~\ref{cor:matching-from-hole}). Since no edges between blocks from different regions belong to any maximum matching (Corollary~\ref{cor:forbidden-dot}), every maximum matching in $T_0$ can be expressed as the union of maximum matchings of the excessive bricks. Let $M_0$ be an arbitrary maximum matching, and let $M_D$ denote the elements of $M_0$ that lie in $D$.
	
	\begin{itemize}
		\item Let $X_0$ be the set of vertices in $A$ not matched by $M_D$, and define $\delta = |X_0|$. Then $|M_0| = a - \delta$. Let $Y_1 = N(X_0)$, and let $X_1$ denote the set of their $M_0$-matched vertices. Clearly $|X_1| = |Y_1|$ and $X_0 \cap X_1 = \emptyset$.
		
		Consider successive increments on the $Y$-side:
		\[
		Y_2 = N(X_1) \setminus Y_1, \quad
		X_2 = \text{the set of $M_0$-matched vertices of } Y_2,
		\]
		\[
		Y_3 = N(X_2) \setminus (Y_1 \cup Y_2), \ \ldots
		\]
		until 
		\[
		Y_i = N(X_i) \setminus (Y_1 \cup \cdots \cup Y_{i-1})
		\]
		becomes empty.
		
		By construction, the sets $X_0, X_1, \ldots, X_i$ are pairwise disjoint. Let
		\[
		X = X_1 \cup \cdots \cup X_i, \quad
		Y = Y_1 \cup \cdots \cup Y_i.
		\]
		Then $|X| = |Y|$ and $Y \subset B$, since unmatched vertices in $B$ are not contained in $Y$. Otherwise, there would exist an $M_0$-alternating augmenting path, contradicting the maximality of $M_0$.

		Note that $X_0$ is the set of $A$-vertices left unmatched by $M_0$ (not just by $M_D$): since $M_0$ is a maximum matching of size $a - \delta$, exactly $\delta$ vertices of $A$ are unmatched, and these are precisely $X_0$. The subsequent BFS construction depends only on $M_0$ and $X_0$, not on the particular choice of $M_D$ within $D$. Hence the sub-brick induced by $X_0 \cup X \cup Y$ is determined by $M_0$ alone. It has imbalance $|X_0| = \delta$, and its vertical auxiliary brick is a hole of $\mathrm{vcount}$ $b + \delta$. By Corollary~\ref{cor:unique-DE}, this is precisely the block $D$ of the excessive decomposition.
		
		\item The number of unmatched vertices on the right side is $b - a + \delta$. Applying the same construction as above, we obtain a sub-brick induced by the corresponding vertex sets, which is independent of the choice of matching, has imbalance $(b-a)+\delta$, and its horizontal auxiliary brick is a hole of $\mathrm{vcount}$ $b + \delta$. By Corollary~\ref{cor:unique-DE}, this is precisely the block $E$ of the excessive decomposition.
		
		\item The remaining vertices induce the middle $C$-part, which is square. The edges of $M_0$ contained in $C$ form a perfect matching $M_C$ in $C$. By Corollary~\ref{excessive_SCC}, the strongly connected components of the directed graph induced by $M_C$ correspond exactly to the bricks $C_1, \ldots, C_k$.
	\end{itemize}

	\medskip
	\noindent\textbf{(DM $\Rightarrow$ MPIS).}
	Suppose $D$ is the block produced by the DM-decomposition, with $\mathrm{imb}(D) = \delta$. Then $D$ is an excessive brick and the hole above it has $\mathrm{vcount}$ $b + \delta$. The MPIS-driven procedure selects a maximum hole in $T_0$, which has $\mathrm{vcount}$ $b + \delta$ (since $\delta$ is the maximum Hall deficiency, by definition of the characteristic). The vertical auxiliary brick of this hole is excessive: if it were larger than $D$ (i.e., contained some $C$-blocks as well), it would not be excessive, since the hole above it would still have $\mathrm{vcount}$ $b + \delta$, contradicting the stopping condition of the MPIS-driven procedure. Hence the vertical auxiliary brick produced by the MPIS-driven procedure coincides with $D$. The argument for $E$ is symmetric.
\end{proof}

\begin{lemma}[Occurrences of Hall deficiencies of size $\delta$]\label{lem:hall-deficiency-occurrences}
	Let $X \subseteq A$ be a set such that $|N(X)| = |X| - \delta$, 
	and hence the brick $(B \setminus N(X)) \times X$ 
	is a hole of $\mathrm{vcount}\ b+\delta$. 
	Then either $X = \mathrm{col}(D)$, 
	or $X$ can be expressed as the union of $\mathrm{col}(D)$ 
	and some of the sets $\mathrm{col}(C_i)$.
\end{lemma}
\begin{proof}
	As usual, we identify each brick with its corresponding induced subgraph (Section~\ref{sec:definitions}). Let $X_D = X \cap \mathrm{col}(D)$, $X_i = X \cap \mathrm{col}(C_i)$ for $i \in [k]$, and $X_E = X \cap \mathrm{col}(E)$. Since $(B \setminus N(X)) \times X$ is a hole of $\mathrm{vcount}\ b + \delta$, we have $|X| - |N(X)| = \delta$.

	\textit{Step 1: $X_E = \emptyset$.}
	The block $E$ is excessive, so for any $\emptyset \neq X_E \subseteq \mathrm{col}(E)$ we have $|N(X_E)| > |X_E|$ in $G_0$ (since $E$ excessive means no hole of $\mathrm{vcount} \geq \mathrm{long}(E)$ exists in $E$, hence $|N(X_E)| > |X_E|$ for all nonempty $X_E \subsetneq \mathrm{col}(E)$, and $|N(\mathrm{col}(E))| > |\mathrm{col}(E)|$ similarly). Thus any nonempty $X_E$ contributes a net positive to $|N(X)| - |X|$, preventing $|X| - |N(X)|$ from reaching $\delta$. Hence $X_E = \emptyset$.

	\textit{Step 2: each $X_i$ is either $\emptyset$ or $\mathrm{col}(C_i)$.}
	Each $C_i$ is excessive, so for any $\emptyset \neq X_i \subsetneq \mathrm{col}(C_i)$ we have $|N_{C_i}(X_i)| \geq |X_i| + 1$ within $C_i$. This extra neighbor within $\mathrm{row}(C_i)$ lies in $N(X)$ but not in $B \setminus N(X)$, strictly reducing $|X| - |N(X)|$ below $\delta$. Hence $X_i \in \{\emptyset, \mathrm{col}(C_i)\}$.

	\textit{Step 3: $X_D = \mathrm{col}(D)$.}
	By Corollary~\ref{cor:unique-DE}, $\mathrm{col}(D)$ is the unique subset of $A$ that achieves the maximum Hall deficiency $\delta$: for any $X_D \subsetneq \mathrm{col}(D)$ we have $|X_D| - |N(X_D)| < \delta$. Since $X_E = \emptyset$ and $X_i \in \{\emptyset, \mathrm{col}(C_i)\}$, the contribution of the $C$-blocks to $|X| - |N(X)|$ is non-positive (each $C_i$ is excessive, hence $|N(\mathrm{col}(C_i))| \geq |\mathrm{col}(C_i)|$). Therefore $|X| - |N(X)| \leq |X_D| - |N(X_D)| < \delta$ unless $X_D = \mathrm{col}(D)$. Hence $X_D = \mathrm{col}(D)$.

	\textit{Step 4: the set $J = \{i : X_i = \mathrm{col}(C_i)\}$ is an ideal of the partial order induced by green edges of $\mathrm{core}(T_0)$.}
	In the BTF, green edges between $C$-blocks run forward: from $\mathrm{col}(C_i)$ to $\mathrm{row}(C_j)$ with $i < j$ in topological order. The partial order on $C$-blocks is defined by reachability along such forward green paths: $C_i \leq C_j$ if there is a directed path of green edges from $C_i$ to $C_j$ (with $i < j$). The set $J$ is an ideal of this order if $C_j \in J$ and $C_i \leq C_j$ implies $C_i \in J$.

	Suppose for contradiction that $J$ is not downward-closed. Then there exist $C_a, C_b$ with $C_a \leq C_b$, $C_b \in J$, but $C_a \notin J$, so there is a directed path of green edges $C_a = P_0 \to P_1 \to \cdots \to P_\ell = C_b$ with $\ell \geq 1$. Since $P_0 \notin J$ and $P_\ell \in J$, some consecutive pair satisfies $P_t \notin J$ and $P_{t+1} \in J$; write $C_i := P_t$ and $C_j := P_{t+1}$, joined by a direct green edge $C_i \to C_j$, so $\mathrm{col}(C_i)$ has a neighbor in $\mathrm{row}(C_j)$. Since $C_j \in J$, we have $\mathrm{col}(C_j) \subseteq X$, and since $C_j$ is excessive, $\mathrm{row}(C_j) \subseteq N(\mathrm{col}(C_j)) \subseteq N(X)$. Since $C_i \notin J$, $\mathrm{col}(C_i) \cap X = \emptyset$. Adding $\mathrm{col}(C_i)$ to $X$ increases $|X|$ by $|\mathrm{col}(C_i)|$, while $|N(X)|$ increases by at most $|\mathrm{col}(C_i)| - 1$ (at least one neighbor of $\mathrm{col}(C_i)$ in $\mathrm{row}(C_j)$ is already in $N(X)$). Hence $|X \cup \mathrm{col}(C_i)| - |N(X \cup \mathrm{col}(C_i))| > \delta$, contradicting maximality of $\delta$. Hence $J$ is downward-closed.
\end{proof}

\begin{definition}[C-block poset]\label{def:cblock-poset}
	Let $C_1,\ldots,C_k$ be the $C$-blocks of the excessive decomposition of $T_0$, arranged in topological order. A \emph{green edge} is an edge of $G_0$ that runs from $\mathrm{col}(C_i)$ to $\mathrm{row}(C_j)$ for some $i < j$; such edges exist in $G_0$ but belong to no maximum matching (Corollary~\ref{cor:forbidden-dot}). Define a partial order $\leq$ on $\{C_1,\ldots,C_k\}$ by
	\[
		C_i \leq C_j \;\overset{\mathrm{def}}{\iff}\; \text{there exists a directed path of green edges from } C_i \text{ to } C_j.
	\]
	The resulting partially ordered set $(\{C_1,\ldots,C_k\}, \leq)$ is called the \emph{$C$-block poset} of $T_0$. A subset $J \subseteq \{1,\ldots,k\}$ is an \emph{ideal} of this poset if it is downward-closed: $C_j \in J$ and $C_i \leq C_j$ implies $C_i \in J$.
\end{definition}

\begin{remark}
	The relation $\leq$ is indeed a partial order. Reflexivity ($C_i \leq C_i$) holds by the trivial path of length zero. Transitivity follows from concatenation of directed paths. Antisymmetry holds because the BTF imposes a strict topological ordering: since all green edges run from $C_i$ to $C_j$ with $i < j$, no directed cycle among the $C$-blocks is possible, so $C_i \leq C_j$ and $C_j \leq C_i$ implies $i = j$.
\end{remark}

\begin{proposition}\label{prop:ideal-hole}
	For every ideal $J \subseteq \{1,\ldots,k\}$ of the $C$-block poset, the set
	\[
	X_J = \mathrm{col}(D) \cup \bigcup_{i \in J} \mathrm{col}(C_i)
	\]
	satisfies $|N(X_J)| = |X_J| - \delta$, and hence $(B \setminus N(X_J)) \times X_J$ is a hole of $\mathrm{vcount}$ $b + \delta$ in $T_0$.
\end{proposition}

\begin{proof}
	By Lemma~\ref{lem:hall-deficiency-occurrences} (Step 3), $|N_{G_0}(\mathrm{col}(D))| = |\mathrm{col}(D)| - \delta$. Since $J$ is downward-closed, no green edges run from $\mathrm{col}(C_i)$ ($i \in J$) to $\mathrm{row}(C_j)$ ($j \notin J$). The row-sets $\mathrm{row}(D), \mathrm{row}(C_1), \ldots, \mathrm{row}(C_k), \mathrm{row}(E)$ are disjoint by the BTF, so the neighborhoods are pairwise disjoint. Hence
	\[
	N_{G_0}(X_J) = N_{G_0}(\mathrm{col}(D)) \cup \bigcup_{i \in J} N_{G_0}(\mathrm{col}(C_i)),
	\]
	giving $|N_{G_0}(X_J)| = (|\mathrm{col}(D)| - \delta) + \sum_{i \in J} |\mathrm{col}(C_i)| = |X_J| - \delta$. \qedhere
\end{proof}

\begin{proposition}[Bijection between ideals and maximum holes]\label{prop:ideal-bijection}
	The map $J \mapsto V_J = (B \setminus N(X_J)) \times X_J$ is a bijection from the set of ideals of the $C$-block poset onto the set of maximum holes of $T_0$.
\end{proposition}

\begin{proof}
	\emph{Well-definedness:} By Proposition~\ref{prop:ideal-hole}, every ideal $J$ produces an MH $V_J$.

	\emph{Injectivity:} If $J \ne J'$, then $X_J \ne X_{J'}$ since the column sets $\mathrm{col}(C_i)$ of distinct $C$-blocks are pairwise disjoint (Lemma~\ref{lem:excessive-components}). Hence $\mathrm{col}(V_J) = X_J \ne X_{J'} = \mathrm{col}(V_{J'})$, so $V_J \ne V_{J'}$.

	\emph{Surjectivity:} By Step~3 of Lemma~\ref{lem:hall-deficiency-occurrences}, every MH has column set of the form $X_J$ for some ideal $J$.
\end{proof}

\begin{corollary}\label{cor:MPIS}
Since for every hole $V$ with $\mathrm{vcount}$ $b+\delta$ we have $X = \mathrm{col}(V)$ satisfying $|N(X)| = |X| - \delta$, the maximum proper independent sets (MPIS) of $G_0$ can be described as follows:
	\begin{enumerate}
		\item Every maximum proper independent set $I$ contains $\mathrm{col}(D)$ (all column-vertices of $D$) and $\mathrm{row}(E)$ (all row-vertices of $E$).
		
		\item There exists a downward-closed subset $J \subseteq \{1,\dots,k\}$ in the poset of $C$-blocks (ordered by reachability along green edges) such that
		\[
		I \cap C
		=
		\bigcup_{i \in J} \mathrm{col}(C_i)
		\;\cup\;
		\bigcup_{i \notin J} \mathrm{row}(C_i).
		\]
		
		\item Under inclusion of ideals ($J \subseteq J'$), the selected subsets on the $A$-side form an increasing family
		\[
		\bigcup_{i \in J} \mathrm{col}(C_i) \subseteq \bigcup_{i \in J'} \mathrm{col}(C_i),
		\]
		while the selected subsets on the $B$-side form a decreasing family
		\[
		\bigcup_{i \notin J} \mathrm{row}(C_i) \supseteq \bigcup_{i \notin J'} \mathrm{row}(C_i).
		\]
	\end{enumerate}
	
\end{corollary}

Consequently, the family of maximum proper independent sets, ordered by inclusion of their $A$-side projections (equivalently, by inclusion of the corresponding ideals $J$; see Remark~\ref{rem:lattice}), is isomorphic to the lattice of ideals of the poset of $C$-blocks (ordered by reachability along green edges), and thus forms a distributive lattice by Birkhoff's representation theorem \cite{Birkhoff1937}. The bijection is given by: each ideal $J$ determines an MPIS via the formula in (2), and each MPIS determines an ideal via Lemma~\ref{lem:hall-deficiency-occurrences}. While the distributive lattice structure of the maximum independent sets of a bipartite graph is classically known \cite{Minasyan2013}, the identification of the underlying poset as the $C$-block poset is the contribution here. This complements the classical viewpoint, in which the maximum independent sets (equivalently, minimum vertex covers) are organized through alternating paths relative to a \emph{fixed} maximum matching \cite{LovaszPlummer1986}: the present description replaces that matching-dependent order by the matching-independent reachability order on the $C$-blocks, so the generating poset is a canonical invariant of $G_0$ rather than an artifact of a chosen matching, and the lattice operations are simply union and intersection of order ideals.

\begin{remark}[Lattice structure of MPISs]
\label{rem:lattice}
Recall (see \cite{DaveyPriestley2008} for background on lattices and order) that a \emph{lattice} is a partially ordered set in which every pair of elements has a supremum ($\vee$) and an infimum ($\wedge$). A lattice is \emph{distributive} if
\[
  x \wedge (y \vee z) = (x \wedge y) \vee (x \wedge z)
\]
holds for all elements. By Birkhoff's representation theorem \cite{Birkhoff1937}, every finite distributive lattice is isomorphic to the lattice of ideals of a finite poset, with $J \vee J' = J \cup J'$ and $J \wedge J' = J \cap J'$.

It should be noted that the ordering on MPISs used in the preceding Corollary is \emph{not} set-inclusion of the MPISs themselves as subsets of $A \cup B$. Item~(3) shows that under $J \subseteq J'$ the $A$-side part of the MPIS \emph{grows} while the $B$-side part \emph{shrinks}, so two MPISs $I$ and $I'$ with $J \subsetneq J'$ satisfy neither $I \subseteq I'$ nor $I' \subseteq I$ as sets. The partial order is instead defined on the $A$-side projection:
\[
  I \leq I' \;\overset{\mathrm{def}}{\iff}\; I \cap A \subseteq I' \cap A
  \;\iff\; J \subseteq J',
\]
and the isomorphism with the ideal lattice is with respect to this order.
\end{remark}

\begin{theorem}[Intersection of all maximum holes and all minimum covers]
\label{thm:ExD-intersection}
	Assume $\delta > 0$ and $(b-a)+\delta > 0$, so that both $D$ and $E$ are present.
	Then:
	\begin{enumerate}
		\item The intersection of all maximum holes in $T_0$ is exactly the brick
		$\mathrm{row}(E)\times\mathrm{col}(D)$:
		\[
			\bigcap_{\substack{V\subseteq T_0\\V\;\mathrm{MH}}} V
			\;=\;
			\mathrm{row}(E)\times\mathrm{col}(D).
		\]
		\item The intersection of all minimum proper cover set bricks in $T_0$ is exactly
		the brick $\mathrm{row}(D)\times\mathrm{col}(E)$:
		\[
			\bigcap_{\substack{U\subseteq T_0\\U\;\mathrm{MPCS}}} U
			\;=\;
			\mathrm{row}(D)\times\mathrm{col}(E).
		\]
	\end{enumerate}
\end{theorem}

\begin{proof}
	We first recall the two facts used below: every MPIS contains both $\mathrm{col}(D)$ and $\mathrm{row}(E)$ (Corollary~\ref{cor:MPIS}), and the column set of every maximum hole has the form $\mathrm{col}(D)\cup\bigcup_{i\in J}\mathrm{col}(C_i)$ for some ideal $J$ of the $C$-block poset (Lemma~\ref{lem:hall-deficiency-occurrences}).
	\textbf{(1)}\;
	($\supseteq$)\;
	By Corollary~\ref{cor:MPIS}(1), every MPIS contains $\mathrm{col}(D)$ and
	$\mathrm{row}(E)$. Since every MH corresponds to an MPIS
	(Proposition~\ref{prop:mpis-hole-corr}),
	we have $\mathrm{row}(E)\times\mathrm{col}(D)\subseteq V$ for every MH $V$.

	($\subseteq$)\;
	Let $(r,c)\notin\mathrm{row}(E)\times\mathrm{col}(D)$.
	\begin{itemize}
		\item If $c\notin\mathrm{col}(D)$: by Corollary~\ref{cor:unique-DE},
		      $\mathrm{vstrip}(D)$ contains an MH $V$.
		      Since $\mathrm{vstrip}(D)=B\times\mathrm{col}(D)$, we have
		      $\mathrm{col}(V)=\mathrm{col}(D)$, so $(r,c)\notin V$.
		\item If $r\notin\mathrm{row}(E)$: by Corollary~\ref{cor:unique-DE},
		      $\mathrm{hstrip}(E)$ contains an MH $V$.
		      Since $\mathrm{hstrip}(E)=\mathrm{row}(E)\times A$, we have
		      $\mathrm{row}(V)=\mathrm{row}(E)$, so $(r,c)\notin V$.
	\end{itemize}

	\textbf{(2)}\;
	By Corollary~\ref{cor:remote-mate-bijection}, the remote-mate map is a bijection between
	MHs and MPCSs, with $\mathrm{row}(U)=B\setminus\mathrm{row}(V)$ and
	$\mathrm{col}(U)=A\setminus\mathrm{col}(V)$. Therefore
	\[
		\bigcap_{U}U
		= \Bigl(B\setminus\!\bigcup_{V}\mathrm{row}(V)\Bigr)\times\Bigl(A\setminus\!\bigcup_{V}\mathrm{col}(V)\Bigr),
	\]
	the intersections and unions ranging over all MHs $V$ and MPCSs $U$.
	The largest MH column set, attained by the full ideal $J=\{1,\ldots,k\}$, is
	\[
		\bigcup_V\mathrm{col}(V)=A\setminus\mathrm{col}(E),
	\]
	since $A$ is partitioned into $\mathrm{col}(D)$, the sets $\mathrm{col}(C_i)$, and $\mathrm{col}(E)$; hence $A\setminus\bigcup_V\mathrm{col}(V)=\mathrm{col}(E)$. Dually,
	\[
		\bigcup_V\mathrm{row}(V)=B\setminus\bigcap_V N(\mathrm{col}(V))
		=B\setminus N(\mathrm{col}(D))=B\setminus\mathrm{row}(D),
	\]
	using $\mathrm{row}(D)=N(\mathrm{col}(D))$ (the $D$-block is the vertical auxiliary brick of $V_\emptyset$). Hence $\bigcap_U U=\mathrm{row}(D)\times\mathrm{col}(E)$. \qedhere
\end{proof}

\begin{remark}
	Theorem~\ref{thm:ExD-intersection} has two notable consequences.

	\textbf{(a) Graph-theoretic reformulation.}
	The common part of all MPISs of $G_0$ is precisely
	$\mathrm{col}(D)\cup\mathrm{row}(E)$: every $A$-vertex of the $D$-block and every
	$B$-vertex of the $E$-block belong to every MPIS, and no other vertex does.
	Dually, by part~(2), the common part of all minimum proper vertex covers is
	precisely $\mathrm{row}(D)\cup\mathrm{col}(E)$.

	\textbf{(b) The brick $\mathrm{row}(E)\times\mathrm{col}(D)$ is always a hole.}
	Since every MH is dot-free and $\mathrm{row}(E)\times\mathrm{col}(D)$ is contained
	in every MH, this brick is dot-free.
	Equivalently, \emph{there is no edge in $G_0$ between any $E$-vertex and any
	$D$-vertex}.
\end{remark}

\begin{remark}[Geometric interpretation of ideals]
	When an MPIS $V$ in $T_0$ is made visible by permuting it to the origin, its vertical auxiliary brick $T_1$ contains $D$ (if it exists) together with precisely the $C$-blocks belonging to the corresponding ideal $J$, while the horizontal auxiliary brick $T_2$ contains the $C$-blocks not in $J$ together with $E$ (if it exists). Thus, the BTF permutation that makes $V$ visible simultaneously exhibits the ideal geometrically: $D$ and the ideal $J$ occupy the lower-left part of the grid, while the complement of $J$ and $E$ occupy the upper-right part. In this way, each MPIS corresponds to a visible partition of the grid into two regions separated by the hole $V$.

	Two blocks $C_i$ and $C_j$ can be swapped (yielding another valid BTF) if and only if there is no directed path of green edges from $C_i$ to $C_j$. Moreover, for every topologically ordered sequence of bricks $C_1, \ldots, C_k$, all MPISs determined by the associated increasing--decreasing system are simultaneously visible in the BTF form.
\end{remark}

\begin{corollary}[Choice-independence of the embedded SCCs]\label{cor:scc-invariant}
	For any maximum matching $M_C$ of the core, the strongly connected components of the embedded graph $H$ are exactly the $C$-blocks. In particular, the partition of $V(C)$ into SCCs does not depend on the choice of $M_C$.
\end{corollary}

\begin{proof}
	Edges between distinct $C$-blocks are forward edges of the BTF order and therefore lie in no perfect matching of the core; hence every maximum matching $M_C$ is the disjoint union of perfect matchings of the individual $C$-blocks (see Lemma~\ref{lem:matching-factorization}). Restricted to a $C$-block $C_l$, $M_C$ is thus a perfect matching of $C_l$, and since $C_l$ is atomic, the embedded subgraph on $C_l$ is strongly connected (Proposition~\ref{prop103}); so $C_l$ lies in a single SCC. As distinct $C$-blocks are joined only by forward edges, no directed cycle crosses a block boundary, so distinct blocks lie in distinct SCCs. Hence the SCCs are exactly the $C$-blocks. A different $M_C$ changes the matching within each block --- and thus the embedded graph $H$ (Remark~\ref{rem:matching-choice}) --- but produces the same partition into SCCs.
\end{proof}

\begin{theorem}\label{thm:canonical-unique}
	The excessive decomposition is canonical: it is uniquely determined by the deficiency structure of $G_0$, up to permutation of the $C$-blocks.
\end{theorem}

\begin{proof}
	\textbf{Step 1.} The value $\delta = \max_{X \subseteq A}(|X| - |N(X)|)$ is uniquely determined by $G_0$.

	\textbf{Step 2.} By Lemma~\ref{lem:hall-deficiency-occurrences}, every hole of $\mathrm{vcount}$ $b+\delta$ has column set $X = \mathrm{col}(D) \cup \bigcup_{i \in J} \mathrm{col}(C_i)$ for some ideal $J$; in particular $\mathrm{col}(D)$ is the same for all such holes. Hence the $D$-block is uniquely determined.

	\textbf{Step 3.} By the symmetric argument (applied to the $B$-side), the $E$-block is uniquely determined.

	\textbf{Step 4.} The core vertex set $V(C) = V(G_0) \setminus (V(D) \cup V(E))$ is therefore uniquely determined as the complement of the two uniquely determined blocks.

	\textbf{Step 5.} The $C$-blocks are the strongly connected components of the embedded directed graph $H$ induced by any maximum matching $M_C$ of $G_0$ restricted to the core (Definition~\ref{def:embedded-graph}, Figure~\ref{GM_PDF09}): the matching elements of $M_C$ serve as the vertices of $H$, and each non-matching dot in the core corresponds to a directed edge of $H$. By Proposition~\ref{prop103}, a sub-brick of the core is atomic if and only if its embedded graph is strongly connected, so the SCCs of $H$ are exactly the $C$-blocks. By Corollary~\ref{cor:scc-invariant}, this partition of $V(C)$ does not depend on the choice of $M_C$; hence the collection of $C$-blocks is uniquely determined by $G_0$.
\end{proof}

In the grid model, $D$, $C$ (the core), and $E$ correspond to horizontal, square, and vertical bricks, respectively. The blocks $D$ and $E$ are unbalanced excessive bricks, each of which may further decompose into atomic sub-blocks; the $C$-blocks are balanced excessive squares, hence atomic.

\begin{corollary}[DM-irreducible--Atomic equivalence]\label{cor:dm-irred}
	A connected bipartite graph $G_0$ is DM-irreducible if and only if it is atomic, i.e., if and only if $T_0$ is excessive.
\end{corollary}

\begin{proof}
	If $G_0$ is DM-irreducible, then the excessive decomposition consists of a single block, which must be excessive (all blocks in the decomposition are excessive by Theorem~\ref{thm:structural_properties}). Since $G_0$ is connected and excessive, it is atomic by Proposition~\ref{prop:atomic-iff-connected}.

	Conversely, if $G_0$ is atomic, it is connected and $m$-excessive for some $m > 0$. By the classification theorem (Theorem~\ref{thm:main_classification}, Table~\ref{tab:block_composition}), an excessive connected brick falls into exactly one of two DM-irreducible classes: balanced with $k=1$ (a single $C$-block) or unbalanced with $k=0$ (a single $E$-block). In both cases the excessive decomposition consists of $G_0$ itself, so $G_0$ is DM-irreducible.
\end{proof}

\section{Atomic Decomposition}\label{sec:atomic}

The excessive decomposition (Section~\ref{sec:excessive-decomp}) yields blocks $D, C_1,\ldots,C_k, E$. The square blocks $C_1,\ldots,C_k$ are already atomic by Lemma~\ref{lem:balanced-excessive-atomic} and Proposition~\ref{prop:atomic-iff-connected}. The blocks $D$ and $E$ are excessive but not necessarily connected; their atomic refinement consists of the connected components of $G_0[D]$ and $G_0[E]$, which are unbalanced and excessive by Lemma~\ref{lem:excessive-components}, hence atomic. No edges exist between distinct components within the same region (grey cells); edges between components from different regions (green cells) are forbidden with respect to every maximum matching.

\begin{lemma}[Block-wise matching factorization]\label{lem:matching-factorization}
	Let $D_1,\ldots,D_p$, $C_1,\ldots,C_k$, $E_1,\ldots,E_q$ be the atomic blocks of $G_0$. Then
	\[
		\nu(G_0) \;=\; \sum_{\text{blocks } T} \mathrm{small}(T),
	\]
	and every maximum matching $M$ of $G_0$ uses no edge joining two distinct blocks, restricts to a maximum matching of each block, and, conversely, the disjoint union of maximum matchings of the individual blocks is a maximum matching of $G_0$.
\end{lemma}

\begin{proof}
	The BTF recursion (Algorithm~\ref{mpis:btf-decomp}) splits a non-excessive brick at a maximum hole $V$ into its vertical and horizontal auxiliary bricks $T_1, T_2$. By Lemma~\ref{lem:matching-union}, a matching of the brick is maximum if and only if it is the union of a maximum matching of $T_1$ and a maximum matching of $T_2$; in particular $\nu = \nu(T_1) + \nu(T_2)$, and every dot in the remote mate of $V$ is forbidden (Corollary~\ref{cor:forbidden-dot}). Iterating down to the leaves of the recursion --- the excessive blocks $D, C_1,\ldots,C_k, E$ --- a maximum matching of $G_0$ is the union of maximum matchings of these blocks, no two of which are joined by a matching edge (the joining dots are remote-mate dots of the holes used along the way, hence forbidden), and $\nu(G_0) = \mathrm{small}(D) + \sum_{l} \mathrm{small}(C_l) + \mathrm{small}(E)$.

	The blocks $D$ and $E$ refine into their connected components $D_1,\ldots,D_p$ and $E_1,\ldots,E_q$, which are vertex-disjoint with no edges between them. As the rows of $D$ (resp.\ the columns of $E$) partition over its components, $\mathrm{small}(D) = \sum_i \mathrm{small}(D_i)$ and $\mathrm{small}(E) = \sum_j \mathrm{small}(E_j)$, so a maximum matching of $D$ (resp.\ $E$) is the disjoint union of maximum matchings of its components. Combining the two layers gives $\nu(G_0) = \sum_T \mathrm{small}(T)$ over all atomic blocks $T$ and the stated factorization; equality of sizes forces the restriction of any maximum matching to each block to be maximum there.
\end{proof}

\begin{theorem}[Atomic Decomposition]\label{thm:atomic-decomp}
	Starting from the excessive decomposition of $T_0$, the decomposition obtained by splitting $D$ and $E$ into their connected components in $G_0[D]$ and $G_0[E]$ respectively, and retaining the square blocks $C_1,\ldots,C_k$, yields the atomic decomposition of $G_0$. This decomposition has the following properties:
	\begin{enumerate}
		\item Every block is atomic: each $D_i$ and $E_j$ is unbalanced and $m$-atomic for some $m > 0$, and each $C_l$ is balanced and $m$-atomic for some $m > 0$. The value of $m$ may differ between blocks.
		\item No edges exist between distinct atomic blocks within the same region (grey cells in the grid model).
		\item Edges between atomic blocks from different regions (green cells) may exist but are forbidden with respect to every maximum matching of $G_0$.
		\item Every maximum matching of $G_0$ is the union of maximum matchings of the individual atomic blocks.
		\item The decomposition is canonical: it is independent of the choice of maximum matching and of the choice of maximum hole in the excessive decomposition.
	\end{enumerate}
\end{theorem}

\begin{proof}
	We verify the five properties in turn.

		\emph{(1) Atomicity.} The square blocks $C_l$ are balanced and excessive, hence connected and atomic (Lemma~\ref{lem:balanced-excessive-atomic}, Proposition~\ref{prop:atomic-iff-connected}). The blocks $D$ and $E$ are excessive; their connected components $D_i, E_j$ are unbalanced and excessive, hence atomic (Lemma~\ref{lem:excessive-components}).

		\emph{(2) No intra-region edges.} Distinct $D$-blocks (resp.\ $E$-blocks) are distinct connected components of $G_0[D]$ (resp.\ $G_0[E]$); by definition no edge joins them.

		\emph{(3) Inter-region edges are forbidden.} Each such (green) edge lies in the remote mate of some maximum hole of the excessive decomposition, hence belongs to no maximum matching (Corollary~\ref{cor:forbidden-dot}).

		\emph{(4) Matching factorization.} By Lemma~\ref{lem:matching-factorization}, every maximum matching of $G_0$ is the disjoint union of maximum matchings of the atomic blocks, and $\nu(G_0) = \sum_T \mathrm{small}(T)$.

		\emph{(5) Canonicity.} The vertex sets $V(D)$, $V(C_1),\ldots,V(C_k)$, $V(E)$ are uniquely determined, independently of the chosen maximum hole (Theorem~\ref{thm:canonical-unique}); their refinement into the components $D_i, E_j$ depends only on $G_0$, hence is independent of the chosen maximum matching, the DM decomposition being matching-independent \cite[Thm~3.2.4]{LovaszPlummer1986}. Corollary~\ref{cor:atomic-unique} records uniqueness of the full block collection.
\end{proof}

\begin{corollary}[Uniqueness of the atomic decomposition]\label{cor:atomic-unique}
	The atomic decomposition of $G_0$ is unique. More precisely:
	\begin{enumerate}
		\item The vertex sets $V(D)$, $V(C_1),\ldots,V(C_k)$, $V(E)$ of the excessive decomposition are uniquely determined (Theorem~\ref{thm:canonical-unique}).
		\item The partition of $V(D)$ into $V(D_1),\ldots,V(D_p)$ and the partition of $V(E)$ into $V(E_1),\ldots,V(E_q)$ are uniquely determined, as they are the connected components of the induced subgraphs $G_0[D]$ and $G_0[E]$ respectively.
		\item The collection $\{C_1,\ldots,C_k\}$ of $C$-blocks is uniquely determined as a set; their order in the BTF is unique up to permutation of topologically incomparable blocks.
	\end{enumerate}
\end{corollary}

\begin{proposition}[Atomic bricks are the minimal excessive units]\label{prop:atomic-minimal}
	A brick is atomic if and only if it is excessive and cannot be partitioned into two or more excessive sub-bricks. That is, the atomic bricks are exactly the minimal elements of the excessive decomposition: applying the excessive decomposition to an atomic brick returns the brick itself.
\end{proposition}

\begin{proof}
	($\Rightarrow$) Let $T$ be atomic, i.e., connected and $m$-excessive with $m>0$. Suppose $T$ were partitioned into excessive sub-bricks $W_1,\ldots,W_r$ with $r\ge2$ arranged in BTF. Since $T$ is connected, some edge joins two distinct parts; but in a BTF of excessive bricks the only inter-block edges are forbidden (Corollary~\ref{cor:forbidden-dot}), and removing them would disconnect $T$, contradicting connectivity. Hence $r=1$, and the excessive decomposition of $T$ is $T$ itself.

	($\Leftarrow$) Suppose $T$ is excessive and admits no partition into two or more excessive sub-bricks. If $T$ were disconnected, its connected components would be excessive by Lemma~\ref{lem:excessive-components} (for unbalanced $T$) or would each be balanced excessive, in either case giving a partition into $\ge2$ excessive sub-bricks, a contradiction. Hence $T$ is connected and excessive, i.e., atomic.
\end{proof}

\begin{remark}
	Proposition~\ref{prop:atomic-minimal} justifies the term \emph{atomic}: an atomic brick is indivisible with respect to the excessive decomposition, just as the excessive decomposition is the finest BTF-decomposition into excessive blocks. The atomic decomposition is thus the unique finest decomposition of $G_0$ into excessive blocks.
\end{remark}

\begin{remark}[No single equivalence relation underlies the atomic decomposition]
\label{rem:no-single-equiv}
	The atomic blocks are \emph{not} the classes of a single equivalence relation on $V(G_0)$. Rather, the decomposition is a two-step construction involving two distinct partition principles:
	\begin{itemize}
		\item For the \emph{$C$-blocks}: mutual reachability in the directed graph induced by a maximum matching on the core (strong connectivity). This is an equivalence relation on $V(C)$.
		\item For the \emph{$D$- and $E$-blocks}: graph connectivity of the induced subgraphs $G_0[D]$ and $G_0[E]$. This is an equivalence relation on $V(D)$ and $V(E)$ separately, but it is a different relation from the one used for the $C$-blocks.
	\end{itemize}
	There is no single equivalence relation on $V(G_0)$ whose classes are precisely the atomic blocks.
\end{remark}

\begin{remark}[Atomic decomposition and canonical structure]
	The atomic decomposition is not merely a technical refinement of the
	Dul\-mage--Men\-del\-sohn decomposition, but the outcome of a natural structural
	hierarchy. The excessive decomposition identifies the regions $D$, $C$, $E$
	and their block structure; the atomic decomposition further resolves $D$ and
	$E$ into their irreducible connected components. If all forbidden edges ---
	corresponding to the green region of the grid model --- are removed from $G_0$,
	the remaining graph decomposes exactly into the atomic blocks.

	\textbf{Relation to the DM decomposition.}
	The DM decomposition of $G_0$ yields three regions: $D$, $C$ (as a single
	undifferentiated block), and $E$. The atomic decomposition \emph{refines} this
	in the following precise sense: each atomic block $D_i$, $C_l$, $E_j$ is a
	vertex-induced subgraph of the corresponding DM region ($D$, $C$, $E$
	respectively), and each DM region is the disjoint union of its atomic blocks.
	The atomic decomposition \emph{coincides} with the DM decomposition if and only
	if $p = q = 1$ and $k = 1$, i.e., $G_0[D]$ and $G_0[E]$ are each connected and
	the core consists of a single SCC. In all other cases, the atomic decomposition
	is strictly finer.

	Within each atomic block $T$ the characteristic $m$ is fixed and no further
	decomposition is possible: by excessiveness there is no hole of $\mathrm{vcount} \ge
	\mathrm{long}(T)$, so the block cannot be split into excessive sub-blocks
	(Proposition~\ref{prop:atomic-minimal}). The grey and yellow regions are structurally
	absent for distinct reasons: yellow cells correspond to maximum holes enforced
	by the excessive decomposition, while grey cells correspond to the absence of
	edges between distinct atomic blocks within the same region, enforced by the
	connectivity structure of $G_0[D]$ and $G_0[E]$.

	The blocks $D$ and $E$ in the excessive decomposition are excessive but may be disconnected; their atomic sub-blocks may differ in their individual excess values $m_i > 0$. The excess of $D$ (resp.\ $E$) as a whole equals the minimum of the excesses of $D_1,\ldots,D_p$ (resp.\ $E_1,\ldots,E_q$), by Proposition~\ref{prop:excess-components}.

	The atomic blocks arise directly from the Hall-deficient and excessive structure; matching-theoretic concepts provide an a posteriori characterization (every maximum matching decomposes into maximum matchings of the atomic blocks). Thus, the atomic decomposition yields a canonical normal form of bipartite graphs with respect to the Hall condition, proper independent sets, and characteristic invariants. The structure of the atomic decomposition is illustrated in Figure~\ref{fig:gridmodel}.
\end{remark}

\section{Conclusion and Further Directions}
The present work establishes three results that refine or extend the prior literature; the precise relationship to prior work is detailed in Section~\ref{sec:related}. First, the connected components of $G_0[D]$ and $G_0[E]$ are proved to be excessive (Lemma~\ref{lem:excessive-components}), yielding a structural characterization of the atomic blocks and a unified geometric framework via the grid model; this is the main structural result of this paper. Second, a complete classification of bipartite graphs into exactly eleven structural classes is derived, with formal impossibility arguments for the seven excluded combinations. Third, the family of maximum proper independent sets is identified explicitly with the ideal lattice of the $C$-block poset, realized geometrically in the grid model.

Beyond the comparison in Section~\ref{sec:related}, one further point is worth noting: the minimum vertex covers visible in a fixed BTF are enumerated in \cite[p.~140]{LovaszPlummer1986}, whereas the complete family across all BTF-isotopes is here shown to form a distributive lattice, isomorphic to the ideal lattice of the $C$-block poset.

The aim of the paper is to establish this structural framework rather than specific applications. Nevertheless, the grid-model perspective suggests several directions for further research. The atomic decomposition of the $D$- and $E$-blocks may have implications for matrix-theoretic applications: in Murota's framework \cite{Murota2000}, the $D$- and $E$-blocks represent degrees of freedom and redundancy in structured systems of equations; the present work shows that each connected component is an irreducible unit of underdetermination or overdetermination, potentially yielding refined solvability criteria. The grid model also provides a natural setting for augmentation problems, such as determining the minimum number of edges needed to make a reducible bipartite graph DM-irreducible \cite{Berczi2018}. Further directions include a Dilworth-type analysis of the $C$-block poset --- in particular, the minimum number of chains into which the blocks decompose, equal by Dilworth's theorem \cite{Dilworth1950} to the size of the largest antichain of mutually incomparable blocks --- and visually traceable formulations of matching algorithms such as Hopcroft--Karp \cite{HopcroftKarp1973}.

\begin{remark}[Extension to multigraphs]
	The grid model represents simple graphs, since each cell holds at most one dot. The structural arguments use only the hole-based characteristic, the Hall condition, and the connectivity of components --- none of which depends on the simplicity assumption --- so the results carry over to bipartite multigraphs, the multiplicities affecting only which dots may be filled.
\end{remark}

	\bibliographystyle{plain}
	\bibliography{references}
	
\end{document}